\documentclass[12pt,reqno]{amsart}
\usepackage[english]{babel}
\usepackage{amsmath}
\usepackage{amsthm}
\usepackage{anysize}
\usepackage{mathtools}
\usepackage{mathrsfs}
\usepackage{amssymb}
\usepackage{xcolor}
\usepackage{xfrac}
\usepackage{esint}
\usepackage{graphicx}
\usepackage{float}
\usepackage{array}
\usepackage{enumitem}
\usepackage{tikz-cd}
\usepackage{centernot}
\usepackage{stmaryrd}
\usepackage{comment}
\excludecomment{confidential}
\usepackage{epstopdf}

\usepackage{graphicx,color}
\usepackage{tikz,pgfplots,pgf}
\usepackage[font=small]{caption}

\newcommand{\sn}{\mathrm{sn}}

\newtheorem{theorem}{Theorem}[section]
\newtheorem{lemma}{Lemma}[section]
\newtheorem{proposition}{Proposition}[section]

\theoremstyle{definition}

\theoremstyle{definition}
\newtheorem{remark}{Remark}[section]
\newtheorem{definition}{Definition}[section]

\usepackage[english]{babel}
\usepackage{amsmath}
\usepackage{amsthm}
\usepackage{mathtools}
\usepackage{mathrsfs}
\usepackage{amssymb}
\usepackage{xcolor}
\usepackage{xfrac}
\usepackage{esint}
\usepackage{graphicx}
\usepackage{float}
\usepackage{array}
\usepackage{enumitem}
\usepackage[new]{old-arrows}
\usepackage[all,cmtip]{xy}
\usepackage{tikz-cd}
\usepackage{centernot}
\usepackage{stmaryrd}
\usepackage{comment}
\usepackage{bbm}
\usepackage{mathabx}

\newcommand{\mc}{\mathcal}

\newcommand{\R}{\mathbb{R}}
\newcommand{\N}{\mathbb{N}}

\renewcommand{\l}{\lambda}

\renewcommand{\a}{\alpha}
\renewcommand{\b}{\beta}
 \renewcommand{\o}{\omega}

\numberwithin{equation}{section}

\makeatletter
\@namedef{subjclassname@2020}{\textup{2020} Mathematics Subject Classification}
\makeatother

\def\a{\alpha}
\def\b{\beta}

\def\g{\gamma}

\def\l{\lambda}

\def\o{\omega}

\theoremstyle{definition}

\theoremstyle{remark}

\newcommand{\Z}{\mathbb{Z}}

\begin{document}

\title[Bifurcation of travelling--rotating Schrödinger maps]{Bifurcation and global continuation of travelling--rotating Schrödinger maps on the sphere}

\author[J. C. Sampedro]{Juan Carlos Sampedro} \thanks{J. C. S. has been supported by the Ministry of Science and Innovation of Spain under the
	Research Grant PID2024–155890NB-I00}
\address{Departamento de Matemática Aplicada y Ciencias de la Computación \\
	Avenida de los Castros 46 \\
	Universidad de Cantabria (UC) \\
	Santander, 39005, Spain.}
\email{juancarlos.sampedro@unican.es}

\author[L. Vega]{Luis Vega}
\address{Departamento de Matemáticas \\ Universidad del Pais Vasco (UPV/EHU) \\ Apdo 644, 48080, Bilbao, Spain \\ \& BCAM, Mazarredo 14, 48009, Bilbao, Spain}
\email{lvega@bcamath.org}

\keywords{Schrödinger maps, vortex filament equation, travelling--rotating profiles, bifurcation, global continuation, elliptic functions}

\subjclass[2020]{35Q55, 35B32, 35B10, 53E40, 76B47}

\begin{abstract}
	We study travelling--rotating solutions of the Schr\"odinger map equation into
	the sphere, viewed as tangent profiles of rigid vortex filaments. Two first
	integrals reduce the profile equation to a scalar cubic equation for the
	vertical component, giving an elliptic-function description and explicit
	closure conditions. We prove bifurcation from the equatorial branch at
	\(\lambda_k=R\sqrt{k^2-1}\), \(k\ge2\), and establish a global continuation
	alternative inside the regular non-polar class. The possible boundary
	mechanisms are pole contact, vertical collapse, and double-root degeneration.
	Numerical continuation of the equatorial branches suggests convergence to
	the north-pole boundary. Up to gauge, the reconstructed vortex filaments are of
	Kida type.
\end{abstract}

\maketitle

\section{Introduction}

The binormal flow, or Vortex Filament Equation,
\[
\mathbf X_t=\mathbf X_s\times \mathbf X_{ss},
\]
is a classical geometric model for the motion of thin vortex filaments under the
localized induction approximation. Its origins go back to the work of Da Rios on
vortex filaments, and we refer to the classical works
\cite{DaRios1906,ArmsHama1965,Betchov1965} and to the monograph
\cite{Saffman1992} for background on vortex dynamics and the localized induction
approximation. A fundamental feature of the equation is its relation, through the Hasimoto transformation, with the one-dimensional cubic nonlinear Schr\"odinger equation \cite{Hasimoto1972}. This connection places the vortex filament equation within a rich integrable and geometric framework, see for instance \cite{LangerPerline1991}.

When the filament is parametrized by arclength, its unit tangent vector
\(\mathbf T=\mathbf X_s\) evolves according to the Schr\"{o}dinger map equation into
the sphere,
\begin{equation}
	\label{Eq1.1}
	\mathbf T_t=\mathbf T\times \mathbf T_{ss},\qquad \|\mathbf T(t,s)\|\equiv 1.
\end{equation}
Thus the geometry of the vortex filament can be studied, at the tangent level,
as a curve evolving on the sphere.

A classical problem is to understand solutions which move without change of form. For the filament equation, Kida obtained a rich family of rigid travelling--rotating filaments, described in terms of elliptic integrals and including, as special or limiting cases, circular rings, helicoidal filaments, Euler elasticae and solitary-wave-type filaments \cite{Kida1981}. Related geometric aspects of elastic curves and filament dynamics also appear in \cite{LangerSinger1984,LangerPerline1991}.

More recently,
García and Vega studied rotating and slipping solutions of the Schrödinger map
equation by means of stereographic coordinates and Crandall--Rabinowitz
bifurcation theory, constructing local branches near circles and helices and
comparing them with Kida's solutions \cite{GarciaVega2024}. 

More broadly, the binormal flow and its relation with Schr\"odinger maps have been studied from several viewpoints, including self-similar dynamics, singular solutions, rough-data theory and special polygonal solutions, see for example \cite{GutierrezRivasVega2003,BanicaVega2009,BanicaVega2012,BanicaVega2015, DeLaHozVega2014,JerrardSmets2012}.

The purpose of the present paper is complementary to these works. We do not aim
at giving merely another explicit elliptic parametrization of rigid vortex
filaments. Instead, we work directly with the intrinsic equation satisfied by
the tangent profile on the sphere and use its first integrals to study
periodic travelling--rotating profiles from the point of view of bifurcation
and global continuation. In particular, our goal is to understand how the closed tangent profiles are
organized into local and global branches, how those branches may lose
regularity, and what the corresponding boundary mechanisms mean
geometrically. 

We consider travelling--rotating tangent profiles of the form
\begin{equation}
	\label{Eq1.2}
	\mathbf T(t,s)=\mathscr R^{\Omega t}\mathbf x(\xi),\qquad\xi=s-at,
\end{equation}
where \(a,\Omega\in\mathbb R\), \(a\neq0\), and
\(\mathscr R^{\Omega t}\) denotes the rotation of angle \(\Omega t\) around
the \(x_3\)-axis. Substitution of \eqref{Eq1.2} into
\eqref{Eq1.1} gives the profile equation
\begin{equation}
	\label{Eq1.3}
	\mathbf x\times \mathbf x''-\Omega e_3\times \mathbf x+a\mathbf x'=0.
\end{equation}
Here and throughout, primes denote differentiation with respect to the travelling variable \(\xi\). For the Schrödinger map equation one has \(\|\mathbf x\|\equiv1\). In the analysis below we keep a general radius
\(R>0\) and study the same profile equation on the sphere $\mathbb{S}^{2}_{R}$, that is,
\[
\|\mathbf x(\xi)\|\equiv R.
\]
This additional parameter is useful for the bifurcation and scaling arguments.
The usual arclength-parametrized vortex filament interpretation corresponds to
\(R=1\).

The first observation is that \eqref{Eq1.3} possesses two
elementary first integrals. The first one is the vertical angular momentum
\[
C=(\mathbf x\times \mathbf x')_3+a x_3,
\]
and the second one is an energy-type integral
\[
E=\frac a2\|\mathbf x'\|^2-\Omega (e_3\times\mathbf x)\cdot \mathbf x'.
\]
Together with the constraint \(\|\mathbf x\|=R\), these two quantities reduce
the dynamics to a scalar equation for the vertical component $x_3$.
More precisely, one obtains
\begin{equation}
	\label{Eq1.4}
	(x_3')^2=P(x_3),
\end{equation}
where \(P\) is an explicit cubic polynomial depending on the parameters and the
first integrals. Thus, as it is well-known \cite{Kida1981}, the qualitative behaviour of the profile is encoded in
the root structure of \(P\). The regular oscillatory regime corresponds to
three real roots
\[
e_1<e_2<e_3, \quad-R<e_1<e_2<R,
\]
with \(x_3\) oscillating between \(e_1\) and \(e_2\). Degenerate regimes occur
when one of the turning levels reaches a pole, when the oscillation collapses,
or when two roots merge into a separatrix configuration. This reduction gives explicit formulae in terms of Jacobi elliptic functions for the vertical component. In the regular case one may write,
after a suitable choice of origin,
\[
x_3(\xi)=e_1+(e_2-e_1)\sn^2(\omega(\xi-\xi_0)\mid k),\quad\omega^2=\frac{\Omega}{2R^2}(e_3-e_1),\quad k=\frac{e_2-e_1}{e_3-e_1},
\]
where \(\mathrm{sn}\) is the Jacobi elliptic sine function, see Appendix~\ref{ApA} and the classical reference \cite{AS}. The vertical motion closes after \(q\) oscillations when
\[
\frac{2qK(k)}{\omega}=2\pi,
\]
where $K$ is the complete elliptic integral of the first kind. On the other hand, writing a non-polar profile in cylindrical coordinates around the $e_3$-axis,
\[
\mathbf{x}(\xi)=\left(\sqrt{R^2-x_3(\xi)^2}\cos\phi(\xi),\sqrt{R^2-x_3(\xi)^2}\sin\phi(\xi),x_3(\xi)\right),
\]
the azimuthal angle $\phi$ satisfies an equation of the form
\[
\phi'(\xi)=\frac{C-a x_3(\xi)}{R^2-x_3(\xi)^2},
\]
and the full profile closes only when $q\,\Delta\phi=2\pi p$, $p\in\mathbb Z$,
where \(\Delta\phi\) is expressed through elliptic integrals of the third kind $\Pi$.
Thus the explicit elliptic classification leads, for closed profiles, to a
coupled transcendental system involving \(K\) and \(\Pi\), see for instance \eqref{Sis}. These formulae are valuable for describing individual profiles, but they are not well suited to understanding how the closed profiles are organized into local and global branches. This is the reason for complementing the elliptic classification with a bifurcation-theoretic analysis.

We then turn to bifurcation from the equatorial branch. For this purpose it is
convenient to distinguish between the travelling variable \(\xi=s-at\) and the
fixed \(2\pi\)-periodic variable
\[
\eta=\omega\xi,\qquad\omega:=\frac{\Omega}{a}.
\]
Equivalently, we introduce the bifurcation parameter
\[
\lambda:=\frac{a^2}{\Omega},\qquad\omega=\frac{a}{\lambda}.
\]
Thus a \(2\pi\)-periodic profile in the variable \(\eta\) gives a physical
travelling profile in the variable \(\xi\) by the rescaling $\mathbf x(\xi)=\mathbf y(\omega\xi)$.
In the variable \(\eta\), the equatorial branch is $\mathbf y_0(\eta)=(R\cos\eta,R\sin\eta,0)$. Writing nearby profiles as $\mathbf y=\mathbf y_0+\mathbf u$,
the fixed-period problem becomes a nonlinear equation $\mathcal G(\lambda,\mathbf u)=0$
on a suitable space of \(2\pi\)-periodic \(H^2\)-perturbations. The natural
augmented operator \(\mathcal G\) contains both the differential equation and
the spherical constraint. However, after the constraint is imposed, the
differential equation has only two independent components. We therefore
introduce a projected augmented operator \(\widehat{\mathcal G}\), obtained by
keeping the two tangential components of the differential equation together
with the constraint. Locally near the equator, the zero sets of
\(\widehat{\mathcal G}\) and \(\mathcal G\) coincide.

The linearization of the projected problem loses invertibility precisely at
\[
\lambda_k:=R\sqrt{k^2-1},\qquad k\ge2.
\]
The full problem also contains the phase degeneracy coming from rotations
around the \(x_3\)-axis. We remove this degeneracy by working first in a
reversible symmetry class. In that class the kernel at \(\lambda=\lambda_k\) is one-dimensional and the Crandall--Rabinowitz transversality condition holds \cite{CR1971}, see also \cite{Kielhofer2012, LG} for a general account of bifurcation theory in Banach spaces.
Restoring the residual phase symmetry afterwards gives, after fixing the trivial equatorial phase, a two-dimensional local sheet of physical travelling profiles. The corresponding result is the following.

\begin{theorem}[Local sheet of travelling--rotating profiles]
	\label{Th1.1}
	Fix an integer \(k\ge2\), and set
	\[\lambda_k:=R\sqrt{k^2-1},\qquad\Omega_k:=\frac{a^2}{\lambda_k},\qquad\omega_k:=\frac{\Omega_k}{a}=\frac{a}{\lambda_k}.
	\]
	Then there exist \(\delta>0\), a smooth function
	\[
	\Lambda_k:(-\delta,\delta)\to(0,+\infty),\qquad\Lambda_k(0)=\lambda_k,
	\]
	and a smooth two-parameter family
	\[
	\Psi_k:(-\delta,\delta)\times\mathbb T\longrightarrow H^2_{\mathrm{loc}}(\mathbb R,\mathbb R^3),\qquad(t,\varphi)\longmapsto \mathbf x_{t,\varphi},
	\]
	with the following properties.
	\begin{enumerate}
		\item[{\rm(i)}]
		For every \((t,\varphi)\in(-\delta,\delta)\times\mathbb T\), the curve
		\(\mathbf x_{t,\varphi}\) is a \(T_\xi(t)\)-periodic solution of
		\eqref{Eq1.3} with
		\[
		\Omega(t):=\frac{a^2}{\Lambda_k(t)},\qquad \omega(t):=\frac{\Omega(t)}{a}=\frac{a}{\Lambda_k(t)},
		\qquad
		T_\xi(t):=\frac{2\pi}{\omega(t)}
		=\frac{2\pi\Lambda_k(t)}{a},
		\]
		and satisfies $\|\mathbf x_{t,\varphi}(\xi)\|\equiv R$.
		\vspace{2pt}
		\item[{\rm(ii)}]
		At \(t=0\), $\mathbf x_{0,\varphi}(\xi)
		=
		\big(R\cos(\omega_k\xi),R\sin(\omega_k\xi),0\big)$
		for every \(\varphi\in\mathbb T\).
		\vspace{2pt}
		\item[{\rm(iii)}]
		As \(t\to0\), one has $\Lambda_k(t)=\lambda_k+o(1)$,
		and, uniformly with respect to \(\varphi\in\mathbb T\),
		\begin{equation}
			\label{Eq1.5}
		\mathbf x_{t,\varphi}(\xi)
		=
		\begin{pmatrix}
			R\cos(\omega(t)\xi)\\[2pt]
			R\sin(\omega(t)\xi)\\[2pt]
			0
		\end{pmatrix}
		+
		t
		\begin{pmatrix}
			-\sin\!\big(k(\omega(t)\xi-\varphi)\big)\sin(\omega(t)\xi)\\[2pt]
			\phantom{-}\sin\!\big(k(\omega(t)\xi-\varphi)\big)\cos(\omega(t)\xi)\\[2pt]
			-\dfrac{k}{\sqrt{k^2-1}}\cos\!\big(k(\omega(t)\xi-\varphi)\big)
		\end{pmatrix}
		+
		o(t)
		\end{equation}
		in \(H^2_{\mathrm{loc}}(\mathbb R,\mathbb R^3)\).
		\vspace{2pt}
		\item[{\rm(iv)}]
		The map $(t,\varphi)\to(\Omega(t),\mathbf x_{t,\varphi})$
		parametrizes a two-dimensional smooth family of nontrivial profiles, with the trivial equatorial phase fixed, issuing
		from
		\[
		\left(
		\Omega_k,\,
		\big(R\cos(\omega_k\xi),R\sin(\omega_k\xi),0\big)
		\right).
		\]
	\end{enumerate}
\end{theorem}

The proof is carried out in the fixed variable \(\eta\). The expansion \eqref{Eq1.5} in
Theorem~\ref{Th1.1} is obtained by pulling back the local
Crandall--Rabinowitz branch through the rescaling
\[
\eta=\omega(t)\xi.
\]
The additional parameter \(\varphi\) corresponds to the \(SO(2)\)-phase symmetry
of the full problem.

Other trivial families, such as non-equatorial parallels of the tangent sphere
which reconstruct as circular helices in the vortex filament, can also be used as base branches for
local bifurcation. These are closely related to the perturbations of helices
studied in \cite{GarciaVega2024}. In this paper we concentrate on the
equatorial branch, for which the subsequent global continuation inside the
regular oscillatory class has a particularly transparent formulation.

The local result is only the starting point. The main goal of the second part of
the paper is to continue these branches globally inside a natural regular
non-polar class. Let
\[
\mathcal Z:=\mathcal G^{-1}(0)
\]
be the zero set of the fixed-period problem in the variable \(\eta\). We define
\(\mathcal Z_{\rm reg}\subset\mathcal Z\) as the subset of non-polar solutions
for which the vertical component $x_3$ is genuinely oscillatory and the associated
cubic polynomial has the regular root configuration
\[
e_1<e_2<e_3,\qquad -R<e_1<e_2<R, \qquad x_3(\mathbb T)=[e_1,e_2].
\]
On this regular class there are two integer invariants. The first one is the
horizontal degree
\[
p(\mathbf y)=\frac1{2\pi}\int_0^{2\pi}\frac{y_1y_2'-y_2y_1'}{y_1^2+y_2^2}\,d\eta,
\]
and the second one is the vertical nodal number \(q\), namely the number of full
oscillations of \(x_3\) over one spatial period. We prove that both \(p\) and
\(q\) are locally constant on \(\mathcal Z_{\rm reg}\). Therefore they are
constant on every connected component of the regular class.

For each \(k\ge2\), let \(\mathscr C_k^{\rm reg}\) be a connected component of
\(\mathcal Z_{\rm reg}\) containing one of the local nontrivial solutions
bifurcating from \((\lambda_k,0)\), or one of its phase translates. The local
expansion above implies that $p=1$, $q=k$
near the bifurcation point, and hence on the whole component
\(\mathscr C_k^{\rm reg}\).

The key estimate for global continuation is an a priori bound at fixed vertical
nodal number. Namely, if \(q\) is fixed and \(0<\lambda\le\Lambda\), then every
regular solution with vertical nodal number \(q\) satisfies a uniform
\(W^{2,\infty}\), and hence \(H^2\), bound depending only on \(R,\Lambda,q\).
The proof uses the elliptic period formula associated with
\eqref{Eq1.4}.

In order to include the possible limit \(\lambda=0\), we use the extended
boundary
\[
\partial_{\rm ext}\mathcal Z_{\rm reg}:=\overline{\mathcal Z_{\rm reg}}^{\,[0,+\infty)\times H^2}\setminus\mathcal Z_{\rm reg}.
\]
The result is in the spirit of global bifurcation alternatives, although here the continuation is carried out inside the regular non-polar class and uses the explicit first integrals and elliptic period estimates, compare with the classical global alternative of Rabinowitz \cite{Rabinowitz1971}. The global continuation result is the following.

\begin{theorem}[Global alternative inside the regular class]
	Let \(k\ge2\), and let \(\mathscr C_{k}^{\rm reg}\) be a connected component of
	\(\mathcal Z_{\rm reg}\) containing one of the nontrivial local solutions
	bifurcating from $(\lambda_k,0)$, $\lambda_{k}=R\sqrt{k^{2}-1}$,
	or one of its phase translates.
	Then \(p=1\), \(q=k\) on \(\mathscr C_{k}^{\rm reg}\).
	Moreover, one of the following alternatives holds:
	\begin{enumerate}
		\item[\rm(i)] $\mathscr C_{k}^{\rm reg}$
		is relatively compact in \(\mathcal Z_{\rm reg}\cup \{(\l_k,0)\}\);
		\item[\rm(ii)] there exists a sequence $\{(\lambda_n,\mathbf u_n)\}_{n\in\N}\subset\mathscr C_{k}^{\rm reg}$
		such that $\lambda_n\to+\infty$;
		
		\item[\rm(iii)] there exists a sequence $\{(\lambda_n,\mathbf u_n)\}_{n\in\N}\subset\mathscr C_{k}^{\rm reg}$ and a point $(\lambda,\mathbf u)\in
		\partial_{\rm{ext}}\mathcal Z_{\rm reg}\setminus\{(\l_k,0)\}$
		such that
		\[
		\lim_{n\to+\infty}(\lambda_n,\mathbf u_n)=(\lambda,\mathbf u) \quad \text{in }
		[0,+\infty)\times H^{2}(\mathbb T,\mathbb R^{3}).
		\]
	\end{enumerate}
\end{theorem}

We also describe the boundary appearing in the preceding alternative. Suppose
that a sequence of regular solutions converges in
\((0,+\infty)\times H^2\) to a solution which is no longer regular. Let
\[
e_{1,n}<e_{2,n}<e_{3,n}
\]
be the roots of the associated vertical cubic. Then, after passing to a
subsequence, loss of regularity occurs through at least one of the following
mechanisms:
\[
e_{1,n}\to -R, \qquad e_{2,n}\to R, \qquad e_{2,n}-e_{1,n}\to0, \qquad e_{3,n}-e_{2,n}\to0.
\]
These correspond respectively to contact with the south pole, contact with the
north pole, collapse of the vertical oscillation, and separatrix degeneration.
At fixed vertical nodal number and bounded \(\lambda\), the separatrix
alternative can occur only together with collapse of the vertical oscillation.
In particular, for the component bifurcating from the equator, where \(p=1\),
for positive limiting parameter, the boundary alternatives are contact
with one of the poles or collapse back to the equatorial branch.

This boundary description is complemented by a numerical
continuation of the elliptic closure equations associated with the components
bifurcating from the equator. In the normalized case \(R=1\), and for
\(k=2,\ldots,10\), the computations indicate that the branches issuing from
the equatorial configuration approach the north-pole boundary, namely
\(e_2=1\), at finite values of \(\lambda\). In the computed range, no numerical
evidence is found for escape to \(\lambda=+\infty\), approach to
\(\lambda=0\), return to the equator away from the initial bifurcation point,
or contact with the south pole, see Figure \ref{F4}.

Finally, we return from the tangent indicatrix to the vortex filament itself. In
the unit-sphere case \(R=1\), a tangent profile \(\mathbf x\) can be integrated
to a filament profile \(\mathbf X_0\) with $\mathbf X_0'=\mathbf x$.
Then the rigid travelling--rotating filament equation is an integral of
\eqref{Eq1.3}: there exists a constant vector \(\mathbf d\) such
that
\[
\mathbf x\times\mathbf x'-\Omega e_3\times\mathbf X_0+a\mathbf x =\mathbf d.
\]
Conversely, differentiating this equation gives precisely
\eqref{Eq1.3}. Moreover, when \(\Omega\neq0\), the horizontal
part of \(\mathbf d\) can be removed by the natural freedom of adding a constant
to the primitive \(\mathbf X_0\). After this gauge normalization, the reconstructed filaments are rigid travelling--rotating filaments of Kida type. Thus the contribution of the paper is not the mere existence of such rigid motions, but rather the bifurcation-theoretic and global-continuation organization of their tangent indicatrices.

The paper is organized as follows. In Section~\ref{S2} we derive the first integrals of the travelling--rotating profile equation and reduce the problem to a scalar cubic equation for the vertical component. This gives the explicit elliptic-function classification of non-polar profiles and the corresponding closure conditions. In Section~\ref{S3} we study bifurcation from the equatorial branch. After fixing the period and introducing the parameter \(\lambda=a^2/\Omega\), we formulate the problem as an augmented operator equation, perform a reversible symmetry reduction, and apply the Crandall--Rabinowitz theorem. In Section~\ref{S4} we introduce the regular non-polar class, define the horizontal degree and the vertical nodal number, and prove the global continuation alternative together with the description of the possible boundary degeneracies.  In Section~\ref{SN} we return to the elliptic closure equations derived in the previous sections. After writing them as the finite-dimensional systems
\(\mathscr S_{p,q}\), we numerically continue the branches
\(\mathscr S_{1,k}\) issuing from the equatorial root configuration. Finally, in Section~\ref{S5} we explain how the tangent profiles reconstruct rigid travelling--rotating solutions of the Vortex Filament Equation. Two
appendices collect, respectively, the elliptic-function conventions used in the paper and the elementary classification of the degenerate regimes \(a=0\) and \(\Omega=0\).

We close the introduction with a comment on the style of the paper. The proofs are deliberately written with full details, especially in the functional-analytic part. The aim is not only to obtain the local and global branches for the present equation, but also to make transparent a procedure that can be reused in related constrained geometric problems with symmetries.

\section{Travelling--rotating profiles}
\label{S2}

In this paper we study travelling--rotating solutions of the Schrödinger map
equation
\[
\mathbf{T}_{t}=\mathbf{T}\times \mathbf{T}_{ss}, \qquad \mathbf{T}: \mathbb{R}\times \mathbb{R} \longrightarrow \mathbb{R}^{3}, \qquad \|\mathbf T(t,s)\|\equiv 1,
\]
of the form
\[
\mathbf{T}(t,s)=\mathscr{R}^{\Omega t}\,\mathbf{T}_{0}(s-at),
\]
where \(a,\Omega\in\mathbb R\) and
\[
\mathscr{R}^{\Omega t}:=
\begin{pmatrix}
	\cos(\Omega t) & -\sin(\Omega t) & 0 \\
	\sin(\Omega t) &  \cos(\Omega t) & 0 \\
	0              &  0              & 1
\end{pmatrix}.
\]
Writing \(\xi=s-at\) and
\[
\mathbf{T}_{0}(\xi)=(x_{1}(\xi),x_{2}(\xi),x_{3}(\xi)), \qquad \xi\in (\a,\o)\subset\mathbb R,
\]
a straightforward computation shows that \(\mathbf{T}_{0}\) satisfies
\begin{equation*}
	\left\{
	\begin{array}{lll}
		x_{2} \ddot x_{3} - x_{3} \ddot x_{2} + a \dot x_{1} + \Omega x_{2} = 0, \\[2pt]
		x_{3} \ddot x_{1} - x_{1} \ddot x_{3} + a \dot x_{2} - \Omega x_{1} = 0, \\[2pt]
		x_{1} \ddot x_{2} - x_{2} \ddot x_{1} + a \dot x_{3} = 0,
	\end{array}
	\right.
\end{equation*}
or, equivalently,
\begin{equation}
	\label{Eq2.1}
	\mathbf{x}\times \ddot{\mathbf{x}}-\Omega e_{3}\times \mathbf{x}+a\,\dot{\mathbf{x}}=0,
\end{equation}
where \(e_{3}=(0,0,1)\) and
\(\mathbf{x}=(x_{1},x_{2},x_{3})\). Dots denote differentiation with respect to \(\xi\).

Although the Schrödinger map equation corresponds to the unit sphere \(\|\mathbf x\|\equiv1\), it is convenient to keep a general radius \(R>0\) in the profile analysis. This entails no essential change. Indeed, if \(\mathbf x=R\mathbf y\), then \(\mathbf x\) solves \eqref{Eq2.1} with parameters \((a,\Omega)\) if and only if \(\mathbf y\) solves the same type of equation on
the unit sphere with parameters
\[
\widetilde a=\frac{a}{R},\qquad\widetilde\Omega=\frac{\Omega}{R}.
\]
Thus the radius can be normalized to \(1\) at the price of rescaling the two
parameters. We keep \(R\) explicitly because it is useful in the bifurcation
formulas, where the critical values will depend on \(R\).

We shall also assume throughout the rest of the paper that
\[
a>0,\qquad\Omega>0.
\]
This is not a loss of generality for the non-degenerate travelling--rotating
case \(a\Omega\neq0\). First, reversing the orientation of the profile variable
changes the sign of \(a\). More precisely, if \(\mathbf x(\xi)\) solves
\eqref{Eq2.1} with parameters \((a,\Omega)\), then $\widetilde{\mathbf x}(\xi):=\mathbf x(-\xi)$
solves the same equation with parameters \((-a,\Omega)\). Second, the
transformation $\widetilde{\mathbf x}(\xi):=-\mathbf x(\xi)$
changes simultaneously the signs of \(a\) and \(\Omega\): it transforms
solutions with parameters \((a,\Omega)\) into solutions with parameters
\((-a,-\Omega)\). Combining these two elementary symmetries, every case with
\(a\neq0\) and \(\Omega\neq0\) can be reduced to the case \(a>0\), \(\Omega>0\). The degenerate regimes \(a=0\) and \(\Omega=0\) are discussed separately in Appendix~\ref{ApB}.

We start our analysis with a simple result.

\begin{proposition}
	Let $\mathbf{x}\in \mathcal{C}^{2}((\alpha,\omega), \mathbb{R}^{3})$ be a solution of \eqref{Eq2.1}. 
	Then there exists $R\geq 0$ such that
	\[
	\|\mathbf{x}(\xi)\|=R \quad \text{for every }\xi\in (\alpha,\omega),
	\]
	where $\|\cdot\|$ denotes the Euclidean norm in $\mathbb{R}^{3}$.
\end{proposition}

\begin{proof}
	If $\mathbf{x}\equiv 0$, the result follows with $R=0$. Suppose $\mathbf{x}\not\equiv 0$.
	Taking the scalar product of equation \eqref{Eq2.1} with $\mathbf{x}$ yields
	\[
	(\mathbf{x}\times \ddot{\mathbf{x}})\cdot \mathbf{x} - \Omega \, (e_{3}\times \mathbf{x})\cdot \mathbf{x} + a \, \dot{\mathbf{x}}\cdot \mathbf{x}=0.
	\]
	The vectors $\mathbf{x}\times \ddot{\mathbf{x}}$ and $e_{3}\times \mathbf{x}$ are orthogonal 
	to $\mathbf{x}$, and hence
	\[
	(\mathbf{x}\times \ddot{\mathbf{x}})\cdot \mathbf{x}=0, \qquad (e_{3}\times \mathbf{x})\cdot \mathbf{x}=0.
	\]
	Therefore,
	\[
	0 = a \, \dot{\mathbf{x}}\cdot \mathbf{x}= \frac{a}{2} \frac{d}{d\xi}\|\mathbf{x}(\xi)\|^{2}.
	\]
	Since $a\neq 0$, we deduce that $\|\mathbf{x}(\xi)\|^{2}$ is constant in $(\alpha,\omega)$, 
	which proves the claim.
\end{proof}

We proceed to describe two additional constants of motion.

\begin{proposition}
	Let $\mathbf{x}\in \mathcal{C}^{2}((\alpha,\omega), \mathbb{R}^{3})$ be a solution of \eqref{Eq2.1}. 
	Then there exist constants $C, E \in \mathbb{R}$ such that:
	\begin{enumerate}
		\item[{\rm(i)}] $(\mathbf{x}\times\dot{\mathbf{x}})_{3}+a \, x_{3}=C$ for each $\xi\in (\alpha,\omega)$.
		\item[{\rm(ii)}] $\displaystyle \frac{a}{2} \|\dot{\mathbf{x}}(\xi)\|^{2}- \Omega \, (e_{3}\times \mathbf{x})\cdot \dot{\mathbf{x}}=E$ for each $\xi\in (\alpha,\omega)$.
	\end{enumerate}
\end{proposition}

\begin{proof}
	We start with item (i). The identity $\dot{\mathbf{x}} \times \dot{\mathbf{x}} =0$ implies
	\begin{equation}
		\label{Eq2.2}
		\frac{d}{d\xi}(\mathbf{x}\times \dot{\mathbf{x}}) 
		= \dot{\mathbf{x}}\times \dot{\mathbf{x}} + \mathbf{x}\times \ddot{\mathbf{x}} 
		= \mathbf{x} \times \ddot{\mathbf{x}}.
	\end{equation}
	Taking the third component of equation \eqref{Eq2.1} and using that $(e_{3}\times \mathbf{x})_{3}=0$, we obtain
	\[
	(\mathbf{x}\times \ddot{\mathbf{x}})_{3}+ a \, \dot x_{3}=0.
	\]
	By \eqref{Eq2.2}, this is equivalent to
	\[
	\frac{d}{d\xi}\big((\mathbf{x}\times \dot{\mathbf{x}})_{3}+ a \, x_{3}\big)=0,
	\]
	which proves (i). For item (ii), define
	\[
	E: (\alpha,\omega) \longrightarrow \mathbb{R}, \qquad 
	E(\xi):= \frac{a}{2} \|\dot{\mathbf{x}}(\xi)\|^{2}- \Omega \, (e_{3}\times \mathbf{x})\cdot \dot{\mathbf{x}}.
	\]
	Then $E\in \mathcal{C}^{1}((\alpha,\omega))$ and
	\begin{align*}
		\dot{E}(\xi) 
		= a \, \dot{\mathbf{x}}\cdot \ddot{\mathbf{x}}
		- \Omega \, (e_{3}\times \mathbf{x})\cdot \ddot{\mathbf{x}}
		- \Omega \, (e_{3}\times \dot{\mathbf{x}})\cdot \dot{\mathbf{x}} = a \, \dot{\mathbf{x}}\cdot \ddot{\mathbf{x}}
		- \Omega \, (e_{3}\times \mathbf{x})\cdot \ddot{\mathbf{x}},
	\end{align*}
	since $(e_{3}\times \dot{\mathbf{x}})\cdot \dot{\mathbf{x}}=0$. On the other hand, taking the scalar product of 
	equation \eqref{Eq2.1} with $\ddot{\mathbf{x}}$ we obtain
	\[
	(\mathbf{x}\times \ddot{\mathbf{x}})\cdot \ddot{\mathbf{x}}- \Omega \, (e_{3}\times \mathbf{x})\cdot \ddot{\mathbf{x}}+ a \, \dot{\mathbf{x}}\cdot \ddot{\mathbf{x}}=0.
	\]
	Here $(\mathbf{x}\times \ddot{\mathbf{x}})\cdot \ddot{\mathbf{x}}=0$ by orthogonality, so
	\[
	a \, \dot{\mathbf{x}}\cdot \ddot{\mathbf{x}}- \Omega \, (e_{3}\times \mathbf{x})\cdot \ddot{\mathbf{x}}=0,
	\]
	which implies $\dot{E}(\xi)=0$ and concludes the proof of (ii).
\end{proof}

The three conserved quantities obtained above can be written in coordinates as
\begin{equation}
	\label{Eq2.3}
	\left\{
	\begin{array}{l}
		R^{2}=x_{1}^{2}+x_{2}^{2}+x_{3}^{2}, \\[2pt]
		C=x_{1}\dot{x}_{2}-x_{2}\dot{x}_{1}+a x_{3}, \\[2pt]
		E=\tfrac{a}{2}\big(\dot{x}_{1}^{2}+\dot{x}_{2}^{2}+\dot{x}_{3}^{2}\big)
		+\Omega \big(x_{2}\dot{x}_{1}-x_{1}\dot{x}_{2}\big).
	\end{array}
	\right.
\end{equation}
Moreover, differentiating the first identity in \eqref{Eq2.3} we obtain
\begin{equation}
	\label{Eq2.4}
	x_{1}\dot{x}_{1}+x_{2}\dot{x}_{2}+x_{3}\dot{x}_{3}=0.
\end{equation}
A direct computation shows that the following algebraic identity holds for all 
$x_{1},x_{2},\dot{x}_{1},\dot{x}_{2}\in\mathbb{R}$:
\[
(x_{1}^{2}+x_{2}^{2})(\dot{x}_{1}^{2}+\dot{x}_{2}^{2})= (x_{1}\dot{x}_{1}+x_{2}\dot{x}_{2})^{2}+(x_{1}\dot{x}_{2}-x_{2}\dot{x}_{1})^{2}.
\]
Using this, together with \eqref{Eq2.3} and \eqref{Eq2.4}, we deduce
\begin{align*}
	\dot{x}_{1}^{2}+\dot{x}_{2}^{2} = \frac{(x_{1}\dot{x}_{1}+x_{2}\dot{x}_{2})^{2}+(x_{1}\dot{x}_{2}-x_{2}\dot{x}_{1})^{2}}
	{x_{1}^{2}+x_{2}^{2}} = \frac{(x_{3}\dot{x}_{3})^{2} + (C-a x_{3})^{2}}{R^{2}-x_{3}^{2}},
\end{align*}
since $x_{1}\dot{x}_{1}+x_{2}\dot{x}_{2}=-x_{3}\dot{x}_{3}$ by \eqref{Eq2.4} and 
$x_{1}\dot{x}_{2}-x_{2}\dot{x}_{1}=C-a x_{3}$ by \eqref{Eq2.3}. Hence, the third equation in \eqref{Eq2.3} can be rewritten as
\begin{align*}
	E 
	&= \frac{a}{2}\big(\dot{x}_{1}^{2}+\dot{x}_{2}^{2}+\dot{x}_{3}^{2}\big)
	+\Omega \big(x_{2}\dot{x}_{1}-x_{1}\dot{x}_{2}\big) \\
	&= \frac{a}{2}\left(\frac{(x_{3}\dot{x}_{3})^{2}+(C-a x_{3})^{2}}{R^{2}-x_{3}^{2}}+\dot{x}_{3}^{2}\right)
	-\Omega (C-a x_{3}) \\
	&= \frac{a}{2}\,\frac{(C-a x_{3})^{2}+R^{2}\dot{x}_{3}^{2}}{R^{2}-x_{3}^{2}}
	-\Omega(C-a x_{3}).
\end{align*}
Rearranging this identity we obtain an autonomous equation for $x_{3}$.

\begin{proposition}
	If $\mathbf{x}\in \mathcal{C}^{2}((\alpha,\omega), \mathbb{R}^{3})$ is a solution of \eqref{Eq2.1}, 
	then $x_{3}$ satisfies the ODE
	\begin{equation}
		\label{Eq2.5}
		\left\{
		\begin{array}{ll}
			\dot{u}^{2}=a_{3} u^{3}- a_{2}u^{2}+a_{1}u+a_{0},  & \xi\in (\alpha,\omega), \\[2pt]
			-R \leq u \leq R,
		\end{array}
		\right.
	\end{equation}
	where
	\begin{align}
		\label{Eq2.6}
		& a_{0}:=\frac{2ER^{2}-aC^{2}+2\Omega C R^{2}}{aR^{2}}, 
		\qquad 
		a_{1}:=\frac{2aC-2\Omega R^{2}}{R^{2}}, \nonumber \\[2pt] 
		& a_{2}:=\frac{2E+a^{3}+2\Omega C}{aR^{2}}, 
		\qquad 
		a_{3}:=\frac{2\Omega}{R^{2}}.
	\end{align}
\end{proposition}

The ODE \eqref{Eq2.5} can be solved explicitly in terms of the roots of the cubic polynomial
\begin{equation}
	\label{Eq2.7}
P(x):= a_{3} x^{3}- a_{2}x^{2}+a_{1}x+a_{0}, \qquad x\in \mathbb{R},
\end{equation}
where $a_{0},a_{1},a_{2}$ and $a_{3}$ are given in \eqref{Eq2.6}. 
The next result classifies all the solutions of the ODE \eqref{Eq2.5}.

\begin{proposition}
	\label{Pr2.4}
	The following statements describe all $\mc{C}^{2}$-solutions of the ODE \eqref{Eq2.5}:
	\vspace{2pt}
	\begin{enumerate}
		\item[{\rm (a)}] The only constant solutions of \eqref{Eq2.5} are
		\[
		u(\xi) = e, \quad \xi\in\mathbb{R},
		\]
		where $e$ is a root of the polynomial $P$ satisfying the bound $-R\leq e \leq R$.
		\vspace{2pt}
		
		\item[{\rm (b)}] If the polynomial $P$ has three real roots $e_{1}<e_{2}<e_{3}$, then
		every non-constant solution of \eqref{Eq2.5} with values in $[-R,R]$ is of the form
		\[
		u(\xi) = e_{1} + (e_{2}-e_{1}) \, \text{\rm sn}^{2}\!\big(\omega (\xi- \xi_{0}) \; \big| \; k\big),
		\quad \xi\in\mathbb{R},
		\]
		where 
		\begin{equation*}
			k:=\frac{e_{2}-e_{1}}{e_{3}-e_{1}}\in (0,1),  \qquad  \omega:=\sqrt{\frac{\Omega}{2R^{2}}(e_{3}-e_{1})},
		\end{equation*}
		and $\mathrm{sn}$ is the Jacobi elliptic function. 
		Therefore, in this case, all non-constant solutions in $[-R,R]$ are periodic with period
		\begin{equation*}
		\mathcal{P}= \frac{2 K(k)}{\omega},
		\end{equation*}
		where $K$ is the complete elliptic integral of the first kind.
		\vspace{2pt}
		
		\item[{\rm (c)}] If the polynomial $P$ has two real roots $e_{1}<e_{2}=R$, with $e_{2}$ of
		multiplicity $2$, then every non-constant solution of \eqref{Eq2.5} with values in $[-R,R]$
		is of the form
		\[
		u(\xi) = e_{1} + (R-e_{1})\tanh^{2}\big(\omega (\xi - \xi_{0})\big), \quad \xi\in\mathbb{R},
		\]
		where
		\begin{equation*}
			\omega:=\sqrt{\frac{\Omega}{2R^{2}}(R-e_{1})}.
		\end{equation*}
		In this case, the solutions are defined on the whole real line and satisfy
		\[
		\lim_{|\xi|\to+\infty}u(\xi) = R.
		\]
		In particular, there are no periodic non-constant solutions in this case.
		\vspace{2pt}
		
		\item[{\rm (d)}] If the polynomial $P$ does not satisfy the assumptions of either item {\rm (b)} or item {\rm (c)}, 
		then equation \eqref{Eq2.5} does not admit any non-constant solution with values in $[-R,R]$. 
	\end{enumerate}
\end{proposition}

\begin{proof}
	Item (a) is immediate, since constant solutions of \eqref{Eq2.5} correspond exactly to roots
	of $P$ lying in $[-R,R]$.
	
	Assume now that there exists a non-constant solution of equation \eqref{Eq2.5}. Then there exist 
	an interval $(a,b)\subset\mathbb{R}$ and a non-constant function 
	$u\in\mathcal{C}^{2}((a,b),\mathbb{R})$ such that
	\[
	\big(u'(\xi)\big)^{2} = P(u(\xi))>0 \quad \text{for every } \xi\in(a,b).
	\]
	In particular, $u((a,b))$ is contained in the set $\{x\in\mathbb{R} : P(x)>0\}$. 
	Since, by the geometric constraint, $-R\leq u(\xi)\leq R$ for all $\xi\in(a,b)$, we must have 
	$P(x_{0})>0$ for some $x_{0}\in[-R,R]$. On the other hand,
	\begin{equation}
		\label{Eq2.8}
	P(-R) = -\frac{(C+aR)^{2}}{R^{2}}\leq 0, \qquad	P(R)  = -\frac{(C-aR)^{2}}{R^{2}}\leq 0.
	\end{equation}
	By continuity of $P$, there exist $r_{1},r_{2}\in[-R,R]$ with $r_{1}<r_{2}$ such that 
	$P(r_{1})=P(r_{2})=0$ and $P(x)>0$ for every $x\in(r_{1},r_{2})$. Consequently, if there 
	exists a non-constant solution of \eqref{Eq2.5}, there exists a connected component $(r_1,r_2)$ of $\{P>0\}$ contained in $[-R,R]$, whose endpoints are roots of $P$.
	
	\medskip
	\noindent\emph{Proof of item {\rm (b)}.}
	Assume that the polynomial $P$ has three real roots $e_{1}<e_{2}<e_{3}$. Since for 
	$\Omega>0$ the leading coefficient of $P$ is positive, we have $P(x)\to -\infty$ as 
	$x\to -\infty$ and $P(x)\to +\infty$ as $x\to +\infty$. Hence the sign of $P$ alternates:
	\begin{align*}
		& P(x)<0 \text{ on }(-\infty,e_{1}),\quad
		P(x)>0 \text{ on }(e_{1},e_{2}), \\
		& P(x)<0 \text{ on }(e_{2},e_{3}),\quad
		P(x)>0 \text{ on }(e_{3},+\infty).
	\end{align*}
	Since $u((a,b))$ contains an interval $(r_{1},r_{2})$ where $P>0$ and 
	$u((a,b))\subset[-R,R]$, this interval must be contained in $(e_{1},e_{2})$. 
	From $P(-R)\le 0$ and $P(R)\le 0$ we deduce that $e_{1},e_{2}\in[-R,R]$ and $e_{3}\ge R$.
	Therefore $r_{1}=e_{1}$ and $r_{2}=e_{2}$, and
	\[
	u((a,b))\subset [e_{1},e_{2}], \qquad (u-e_{1})(u-e_{2})(u-e_{3}) > 0 \text{ on }(a,b).
	\]
	
	Since we are in the case $\Omega>0$, we can choose the positive branch of the square root and rewrite
	equation \eqref{Eq2.5} as
	\begin{equation}\label{Eq2.9}
		\dot{u} = \sqrt{\frac{2\Omega}{R^{2}}\, (u-e_{1})(u-e_{2})(u-e_{3})}.
	\end{equation}
	Integrating, we obtain
	\[
	\frac{\sqrt{2\Omega}}{R}\,(\xi-\xi_{0})= \int_{e_{1}}^{u}\frac{dz}{\sqrt{(z-e_{1})(z-e_{2})(z-e_{3})}},
	\]
	where $\xi_{0}$ is chosen so that $u(\xi_{0}) = e_{1}$.  
	The change of variables $z = e_{1} + (e_{2}-e_{1})t^{2}$, $0\leq t \leq 1$,
	yields
	\[
	\frac{\sqrt{2\Omega}}{R}\,(\xi-\xi_{0})= \frac{2}{\sqrt{e_{3}-e_{1}}}\int_{0}^{\sqrt{\frac{u-e_{1}}{e_{2}-e_{1}}}} 
	\frac{dt}{\sqrt{(1-t^{2})(1-k t^{2})}},
	\]
	where
	\begin{equation}
		\label{Eq2.10}
	k := \frac{e_{2}-e_{1}}{e_{3}-e_{1}} \in (0,1).
	\end{equation}
	By the definition of the Jacobi elliptic function $\mathrm{sn}$ (see Appendix \ref{ApA}), we obtain
	\[
	\sqrt{\frac{u-e_{1}}{e_{2}-e_{1}}}= \mathrm{sn}\!\left(\sqrt{\frac{\Omega}{2R^{2}}(e_{3}-e_{1})}\,(\xi-\xi_{0})\,\Big|\, k\right).
	\]
	Equivalently,
	\[
	u(\xi)= e_{1} + (e_{2}-e_{1})\,\mathrm{sn}^{2}\!\left(\omega(\xi-\xi_{0}) \,\Big|\, k\right), \quad  \omega:=\sqrt{\frac{\Omega}{2R^{2}}(e_{3}-e_{1})}.
	\]
	Since $\mathrm{sn}(\cdot\,|\,k)$ has real period $4K(k)$, the function 
	$\mathrm{sn}^{2}(\cdot\,|\,k)$ has period $2K(k)$, and thus $u$ has period
	\[
	\mathcal{P}=\frac{2K(k)}{\omega}.
	\]
	This proves item (b).
	
	We have used the positive branch in \eqref{Eq2.9}. If instead we take the negative branch,
	\[
	\dot{u} = -\sqrt{\frac{2\Omega}{R^{2}}(u-e_{1})(u-e_{2})(u-e_{3})},
	\]
	the integration leads to
	\[
	-\frac{\sqrt{2\Omega}}{R}(\xi-\xi_{0})= \int_{e_{1}}^{u}\frac{dz}{\sqrt{(z-e_{1})(z-e_{2})(z-e_{3})}},
	\]
	which is equivalent to the previous relation after the change 
	$\xi \mapsto 2\xi_{0}-\xi$ (i.e., a shift of the origin of the independent variable). 
	Thus the choice of sign only reverses the orientation in $\xi$ and does not produce new
	solutions; these are already encoded in the expression above with a suitable choice of 
	$\xi_{0}$.
	
	\medskip
	\noindent\emph{Proof of item {\rm (c)}.}
	For item (c) we consider the degenerate case in which $P$ has two real roots 
	$e_{1}<e_{2}$, with $e_{2}$ of multiplicity $2$. As before, the existence of a non-constant 
	solution with values in $[-R,R]$ forces the existence of an interval $(r_{1},r_{2})\subset[-R,R]$
	where $P>0$, bounded by two roots of $P$ in $[-R,R]$. The sign structure of a cubic with a 
	double root, together with $P(\pm R)\le 0$, implies that this is only possible if $e_{2}$ is the largest root and $e_{2}=R$.
	In particular, $e_{1}\in[-R,R]$ and
	\[
	P(u) = \frac{2\Omega}{R^{2}} (u-e_{1})(u-e_{2})^{2},\qquad e_{2}=R.
	\]
	This configuration can be viewed as the limit $e_{3}\to e_{2}$ in item (b). 
	Indeed, when $e_{3}\to e_{2}$, the modulus $k$ in \eqref{Eq2.10}
	tends to $1$, and the identity $\mathrm{sn}(\tau\,|\,1)=\tanh(\tau)$ yields
	\[
	u(\xi) = e_{1} + (e_{2}-e_{1})\,\tanh^{2}\!\left(\sqrt{\frac{\Omega}{2R^{2}}(e_{2}-e_{1})}\,(\xi-\xi_{0})\right),
	\]
	which is precisely the formula stated in (c) with \(\omega:=\sqrt{\frac{\Omega}{2R^{2}}(e_{2}-e_{1})}\).
	Since $\tanh^{2}(\tau)\to 1$ as $|\tau|\to\infty$, it follows that \(\lim_{|\xi|\to\infty} u(\xi) = e_{2}\),
	and non-constant solutions are not periodic in this case. This proves item (c).
	
	\medskip
	\noindent\emph{Proof of item {\rm (d)}.}
	Finally, suppose that the polynomial $P$ does not satisfy the assumptions of either (b) or (c). 
	In particular, $P$ does not admit three distinct real roots $e_{1}<e_{2}<e_{3}$ with 
	$e_{1},e_{2}\in[-R,R]$ and $e_{3}>R$, nor the degenerate configuration with a double root 
	at $R$. Assume by contradiction that there exists a non-constant solution $u$ of 
	\eqref{Eq2.5} with values in $[-R,R]$. Then, as shown at the beginning of the proof, there 
	must exist two real numbers $r_{1}<r_{2}$ in $[-R,R]$ such that $P(r_{1})=P(r_{2})=0$ and 
	$P(x)>0$ for all $x\in(r_{1},r_{2})$. This forces $P$ to have either three real roots with 
	the geometry described in item (b), or the degenerate configuration of item (c). Hence $P$ 
	would satisfy one of the previously considered conditions, a contradiction. We conclude that, 
	if $P$ does not satisfy (b) or (c), equation \eqref{Eq2.5} admits no non-constant solution 
	with values in $[-R,R]$.
\end{proof}

We now introduce azimuthal coordinates on the sphere for non-polar solutions of \eqref{Eq2.1}, that is, solutions $\mathbf{x}$ such that $|x_3|<R$. Since every non-trivial solution
$\mathbf{x}$ of \eqref{Eq2.1} has constant norm $\|\mathbf{x}\|\equiv R>0$, we may write
\[
\mathbf{x}=(x_{1},x_{2},x_{3})=\big(R\sqrt{1-z^{2}}\cos\phi,\; R\sqrt{1-z^{2}}\sin\phi,\; Rz\big),
\]
where $z=z(\xi)\in(-1,1)$ and $\phi=\phi(\xi)\in\mathbb{R}$ are functions of $\xi$ and 
$R=\|\mathbf{x}\|$. A direct computation gives
\[
\dot z = \frac{1}{R}\,\dot x_{3},
\]
and, setting $\rho(\xi):=\sqrt{1-z(\xi)^{2}}$,
\begin{align*}
	\dot x_{1}
	&= R\big(\dot\rho\cos\phi - \rho\sin\phi\,\dot\phi\big)
	= R\left(-\frac{z}{\rho}\dot z\cos\phi - \rho\sin\phi\,\dot\phi\right),\\[2pt]
	\dot x_{2}
	&= R\big(\dot\rho\sin\phi + \rho\cos\phi\,\dot\phi\big)
	= R\left(-\frac{z}{\rho}\dot z\sin\phi + \rho\cos\phi\,\dot\phi\right),
\end{align*}
where we have used $\dot\rho = -\dfrac{z}{\rho}\dot z$. From these expressions we obtain
\[
x_{1}\dot x_{2} - x_{2}\dot x_{1}= R^{2}\rho^{2}\dot\phi= R^{2}(1-z^{2})\dot\phi.
\]
Therefore, by the second identity in \eqref{Eq2.3},
\[
C = x_{1}\dot x_{2}-x_{2}\dot x_{1}+a x_{3}= R^{2}(1-z^{2})\dot\phi + aR z,
\]
and hence
\begin{equation}\label{Eq2.11}
	\dot\phi(\xi)= \frac{C-aR z(\xi)}{R^{2}(1-z(\xi)^{2})}= \frac{C-a x_{3}(\xi)}{R^{2}-x_{3}(\xi)^{2}}.
\end{equation}

Recall that $u(\xi):=x_{3}(\xi)$ satisfies the scalar ODE \eqref{Eq2.5}, $\dot u^{2} = P(u)$, $-R\le u\le R$,
where $P$ is the cubic polynomial introduced in \eqref{Eq2.7}. In terms of $z=u/R$ this becomes
\begin{equation*}
	\dot z^{2}= \frac{1}{R^{2}}\,P(Rz), \qquad -1\le z\le 1.
\end{equation*}

We thus arrive at the following reduction of \eqref{Eq2.1} to azimuthal coordinates.

\begin{proposition}
	Let $\mathbf{x}\in\mathcal{C}^{2}((\alpha,\omega),\mathbb{R}^{3})$ be a non-trivial and non-polar solution
	of \eqref{Eq2.1}, and let $R>0$, $C,E\in\mathbb{R}$ be the constants given by \eqref{Eq2.3}. 
	Writing $\mathbf{x}$ in azimuthal coordinates as
	\[
	\mathbf{x}(\xi)=\big(R\sqrt{1-z(\xi)^{2}}\cos\phi(\xi),\; R\sqrt{1-z(\xi)^{2}}\sin\phi(\xi),\; Rz(\xi)\big),
	\]
	the functions $z,\phi$ satisfy
	\begin{equation*}
	\left\{
	\begin{array}{ll}
		\dot z^{2} = \dfrac{1}{R^{2}}\,P(Rz), & -1\le z \le 1, \\[6pt]
		\dot\phi = \dfrac{C-aRz}{R^{2}(1-z^{2})},
	\end{array}
	\right.
	\end{equation*}
	where $P$ is the cubic polynomial defined in \eqref{Eq2.7}.
\end{proposition}

\begin{theorem}[Classification of non-polar solutions]
	\label{TP}
	The following statements describe all the profiles of the $\mc{C}^{2}$-solutions of equation 
	\begin{equation}
		\label{Eq2.12}
	\mathbf{x}\times \ddot{\mathbf{x}} - \Omega \, e_{3}\times \mathbf{x} + a\, \dot{\mathbf{x}}=0, \quad \|x\|\equiv R.
\end{equation}
	such that $|x_{3}|<R$.
	\vspace{2pt}
	\begin{enumerate}
		
		\item[{\rm (a)}] {\rm \textbf{Constant-latitude solutions.}}
		The non-polar constant-latitude solutions are precisely the curves
		\[
		\mathbf{x}(\xi)
		=
		\left(\sqrt{R^{2}-e^{2}}\cos\!\left(\phi_{0}+\nu\xi\right),\;\sqrt{R^{2}-e^{2}}\sin\!\left(\phi_{0}+\nu\xi\right),\; e\right),
		\]
		where $e\in(-R,R)$, $\phi_{0}\in\mathbb R$ and $e\nu^{2}-a\nu+\Omega=0$. Moreover $\nu=\tfrac{C-ae}{R^{2}-e^{2}}$. In particular, $\mathbf{x}$ parametrizes a horizontal circle
		on the sphere $\|\mathbf{x}\|=R$ at height $x_{3}\equiv e$, with constant angular
		velocity $\nu$.
		\vspace{2pt}
		
		\item[{\rm (b)}] {\rm\textbf{Periodic and quasi-periodic rotating solutions.}}
		Assume that the polynomial $P$ has three real roots $e_1<e_2<e_3$ such that $-R<e_{1}<e_{2}<R$. Then every
		non-constant solution in the $x_{3}$-component of \eqref{Eq2.12} is of the form
		\begin{equation}
			\label{Eq2.13}
			\mathbf{x}(\xi)=\big(R\sqrt{1-z(\xi)^{2}}\cos\phi(\xi),\; R\sqrt{1-z(\xi)^{2}}\sin\phi(\xi),\; Rz(\xi)\big),
		\end{equation}
		where
		\[
		z(\xi)= \frac{1}{R}\left(e_{1}+(e_{2}-e_{1}) \;\mathrm{sn}^{2}\; (\omega(\xi-\xi_{0})\;\big|\;k)\right),
		\]
		with
		\[
		\omega = \sqrt{\frac{\Omega}{2R^{2}}(e_{3}-e_{1})},
		\qquad
		k:=\frac{e_{2}-e_{1}}{e_{3}-e_{1}}\in(0,1),
		\]
		and
		\begin{align}
			\label{Eq2.14}
			\phi(\xi)= \; & \phi_{0} + \frac{C-aR}{2R(R-e_{1})\,\omega} \; \Pi\!\left(\omega(\xi-\xi_{0})\;\Big|\;\frac{e_{2}-e_{1}}{R-e_{1}},k\right) \nonumber \\ 
			& + \frac{C+aR}{2R(R+e_{1})\,\omega} \;\Pi\!\left(\omega(\xi-\xi_{0})\;\Big|\;-\,\frac{e_{2}-e_{1}}{R+e_{1}},k\right),
		\end{align}
		for some $\xi_{0},\phi_{0}\in\mathbb{R}$. Here $\mathrm{sn}$ denotes the Jacobi elliptic
		function and $\Pi$ the incomplete elliptic integral of the third kind in Jacobi form,
		\[
		\Pi(u\,|\,n,k):=\int_{0}^{u}\frac{d\tau}{1-n\,\mathrm{sn}^{2}(\tau\,|\,k)}.
		\]
		Let
		\[
		\mathcal{T}:= \frac{2\,K(k)}{\omega}= \frac{2\,K\!\left(\frac{e_{2}-e_{1}}{e_{3}-e_{1}}\right)}{\sqrt{\frac{\Omega}{2R^{2}}(e_{3}-e_{1})}},
		\]
		where $K(k)$ is the complete elliptic integral of the first kind. Then $z$ (and hence $x_{3}=Rz$) has minimal period $\mathcal{T}$.
		
		\begin{itemize}
			\item[{\rm (i)}] The solution \eqref{Eq2.13} is periodic if, and only if, there exist $p,q\in \mathbb{Z}$, $q\ge 1$, $\gcd(p,q)=1$, such that
			\begin{align}
				\label{Eq2.15}
				& \frac{C-aR}{R(R-e_{1})}\;\Pi\!\left(\frac{e_{2}-e_{1}}{R-e_{1}} \;\Big|\; k\right)+ \frac{C+aR}{R(R+e_{1})}\;\Pi\!\left(-\,\frac{e_{2}-e_{1}}{R+e_{1}} \;\Big|\; k\right)  \nonumber \\
				& \qquad\qquad= 2\pi\,\omega\,\frac{p}{q},
			\end{align}
			where, in \eqref{Eq2.15}, $\Pi(\eta\,|\,k)$ denotes the \emph{complete} elliptic integral of the third kind,
			\[
			\Pi(\eta\,|\,k) := \frac{1}{2}\int_{0}^{2K(k)}\frac{d\tau}{1-\eta\,\mathrm{sn}^{2}(\tau\,|\,k)}.
			\]
			In this case, the minimal period of $\mathbf{x}$ is 
			\[
			\mathcal{P} = q\,\mathcal{T}= \frac{2q\,K\!\left(\frac{e_{2}-e_{1}}{e_{3}-e_{1}}\right)}{\sqrt{\frac{\Omega}{2R^{2}}(e_{3}-e_{1})}}.
			\]
			\item[{\rm (ii)}] If \eqref{Eq2.15} does not hold for any $p,q\in \mathbb{Z}$, $q\ge 1$, $\gcd(p,q)=1$, then the solution
			\eqref{Eq2.13} is quasi-periodic: the component $x_{3}$ is periodic of period $\mathcal{T}$,
			while the azimuthal angle $\phi$ winds on the circle with an incommensurable rotation number,
			and therefore $\mathbf{x}$ does not close up.
		\end{itemize}
		\vspace{2pt}
		
		\item[{\rm (c)}] {\rm\textbf{Soliton-type (homoclinic/separatrix-type) solutions.}}
		Assume that the polynomial $P$ has two real roots $-R<e_{1}<e_{2}=R$, with $e_{2}$ of multiplicity $2$.
		Then necessarily $C=aR$, and every non-constant solution in the $x_{3}$-component of \eqref{Eq2.12} is of the form
		\begin{equation*}
			\mathbf{x}(\xi)=\big(R\sqrt{1-z(\xi)^{2}}\cos\phi(\xi),\; R\sqrt{1-z(\xi)^{2}}\sin\phi(\xi),\; Rz(\xi)\big),
		\end{equation*}
		where
		\[
		z(\xi)= \frac{1}{R}\left(e_{1}+(R-e_{1})\; \tanh^{2}\!\big(\omega (\xi-\xi_{0})\big)\right),
		\]
		with
		\[
		\omega = \sqrt{\frac{\Omega}{2R^{2}}(R-e_{1})},\qquad\delta = \sqrt{\frac{R-e_{1}}{R+e_{1}}},
		\]
		and the azimuthal angle is given explicitly by
		\[
		\phi(\xi)= \phi_{0}+ \frac{a}{2R}\left(\xi-\xi_{0}+ \frac{\delta}{\omega}\,\arctan\!\big(\delta\,\tanh(\omega(\xi-\xi_{0}))\big)\right),
		\]
		for some $\xi_{0},\phi_{0}\in\mathbb{R}$. In particular, $z(\xi)\to  1$ as $|\xi|\to\infty$, and $\mathbf{x}(\xi)$ converges to the
		north pole $(0,0,R)$ as $|\xi|\to\infty$. These solutions are not periodic and have the
		qualitative behaviour of soliton trajectories on the sphere.
	\end{enumerate}
\end{theorem}

Geometrically, each non-polar solution $\mathbf{x}$ of \eqref{Eq2.1} is a curve on the sphere $\|\mathbf{x}\|=R$, and the previous classification shows that the admissible trajectories 
are either horizontal circles (constant latitude), band–limited oscillatory curves between two 
parallels, or soliton–type spirals asymptotic to the north pole. In terms of the Schr\"odinger 
map (for $R=1$), the corresponding solution
\[
\mathbf{T}(t,s)=\mathscr{R}^{\Omega t}\,\mathbf{T}_{0}(s-at),\qquad \mathbf{T}_{0}(\xi)=\mathbf{x}(\xi),
\]
is obtained by rigidly rotating in time the profile $\mathbf{T}_{0}$ on the unit sphere 
and simultaneously translating it with constant speed $a$ along the parameter $s$. Thus our 
solutions are rotating–traveling waves on $\mathbb{S}^{2}$, whose spatial profile is precisely 
one of the spherical curves described in Theorem~\ref{TP}.
\begin{proof}[Proof of Theorem \ref{TP}]
	Recall that, for any solution $\mathbf{x}$ of \eqref{Eq2.1} with $\|\mathbf{x}\|\equiv R>0$,
	the third component $u(\xi):=x_{3}(\xi)$ satisfies the scalar ODE \eqref{Eq2.5},
	\[
	\dot u^{\,2}=P(u),\qquad -R\le u\le R,
	\]
	and the azimuthal angle $\phi$ satisfies \eqref{Eq2.11}.
	Moreover, Proposition~\ref{Pr2.4} gives the complete classification of solutions
	of \eqref{Eq2.5} according to the real roots of $P$.
	
	\medskip\noindent
	\emph{Proof of item {\rm (a)}.}
	Suppose first that \(\mathbf{x}\) is a non-polar solution with constant vertical
	component $x_{3}(\xi)\equiv e$, $e\in(-R,R)$.
	Since \(\|\mathbf{x}\|\equiv R\), its horizontal projection has constant radius $\rho:=\sqrt{R^{2}-e^{2}}>0$.
	Therefore, we may write
	\[
	\mathbf{x}(\xi)=\big(\rho\cos\phi(\xi),\rho\sin\phi(\xi),e\big).
	\]
	Hence, by \eqref{Eq2.11},
	\[
	\dot\phi=\frac{C-ae}{\rho^{2}}=\frac{C-ae}{R^{2}-e^{2}}=:\nu.
	\]
	Thus $\phi(\xi)=\nu\xi+\phi_{0}$,
	and consequently
	\[
	\mathbf{x}(\xi)=\big(\sqrt{R^{2}-e^{2}}\cos(\nu\xi+\phi_{0}),\sqrt{R^{2}-e^{2}}\sin(\nu\xi+\phi_{0}),e\big).
	\]
	It remains to determine which values of \(e\) and \(\nu\) give genuine solutions of
	the profile equation. For such a curve one has $\dot{\mathbf{x}}=\nu e_{3}\times\mathbf{x}$ and $\ddot{\mathbf{x}}=-\nu^{2}(\mathbf{x}-e e_{3})$. Therefore
	\[
	\mathbf{x}\times\ddot{\mathbf{x}}=-e\nu^{2}e_{3}\times\mathbf{x}.
	\]
	Substitution into equation \eqref{Eq2.12}
	gives
	\[
	(-e\nu^{2}-\Omega+a\nu)e_{3}\times\mathbf{x}=0.
	\]
	Since \(e\in(-R,R)\), the horizontal projection is nonzero, and hence $e_{3}\times\mathbf{x}\neq0$.
	Thus
	\[
	e\nu^{2}-a\nu+\Omega=0.
	\]
	Conversely, if \(e\in(-R,R)\), \(\nu\in\mathbb R\), and $e\nu^{2}-a\nu+\Omega=0$,
	then the curve
	\[
	\mathbf{x}(\xi)=\big(\sqrt{R^{2}-e^{2}}\cos(\nu\xi+\phi_{0}),\sqrt{R^{2}-e^{2}}\sin(\nu\xi+\phi_{0}),e\big)
	\]
	satisfies $\dot{\mathbf{x}}=\nu e_{3}\times\mathbf{x}$, $\ddot{\mathbf{x}}=-\nu^{2}(\mathbf{x}-e e_{3})$, and the same computation shows that
	\[
	\mathbf{x}\times\ddot{\mathbf{x}}-\Omega e_{3}\times\mathbf{x}+a\dot{\mathbf{x}}=0.
	\]
	This proves the characterization of the non-polar constant-latitude solutions.
	
	\medskip\noindent
	\emph{Proof of item {\rm (b)}.}
	Assume now that $P$ has three real roots $e_{1}<e_{2}<e_{3}$. By
	Proposition~\ref{Pr2.4}{\rm(b)}, any non-constant solution of \eqref{Eq2.5} with values in $[-R,R]$
	is of the form
	\[
	u(\xi)= e_{1}+(e_{2}-e_{1})\,\mathrm{sn}^{2}\!\big(\omega(\xi-\xi_{0})\,\big|\,k\big),
	\]
	where
	\[
	\omega = \sqrt{\frac{\Omega}{2R^{2}}(e_{3}-e_{1})}, \qquad k = \frac{e_{2}-e_{1}}{e_{3}-e_{1}}.
	\]
	This gives the expression for $z=u/R$ in the statement.
	
	To compute $\phi$, we rewrite \eqref{Eq2.11} by partial fractions:
	\[
	\frac{C-aRz}{1-z^{2}}= \frac{C-aR}{2}\,\frac{1}{1-z}+\frac{C+aR}{2}\,\frac{1}{1+z},
	\]
	so that
	\begin{equation}\label{Eq2.16}
		\dot{\phi}(\xi)= \frac{C-aR}{2R^{2}}\frac{1}{1-z(\xi)}+ \frac{C+aR}{2R^{2}}\frac{1}{1+z(\xi)}.
	\end{equation}
	Using
	\[
	z(\xi) = \frac{e_{1}}{R}+\frac{e_{2}-e_{1}}{R}\,\mathrm{sn}^{2}\!\big(\omega(\xi-\xi_{0})\,\big|\,k\big),
	\]
	we obtain
	\begin{align*}
		1-z(\xi)
		&= \frac{R-e_{1}}{R}\Big(1-\frac{e_{2}-e_{1}}{R-e_{1}}\mathrm{sn}^{2}(\omega(\xi-\xi_{0})\,|\,k)\Big),\\[2pt]1+z(\xi)
		&= \frac{R+e_{1}}{R}\Big(1+\frac{e_{2}-e_{1}}{R+e_{1}}\mathrm{sn}^{2}(\omega(\xi-\xi_{0})\,|\,k)\Big).
	\end{align*}
	Inserting these expressions into \eqref{Eq2.16}, we get
	\begin{align*}
		\dot\phi(\xi)
		&= \frac{C-aR}{2R(R-e_{1})}
		\frac{1}{1-\frac{e_{2}-e_{1}}{R-e_{1}}
			\mathrm{sn}^{2}(\omega(\xi-\xi_{0})\,|\,k)} \\[2pt]
		&\quad
		+\frac{C+aR}{2R(R+e_{1})}
		\frac{1}{1+\frac{e_{2}-e_{1}}{R+e_{1}}
			\mathrm{sn}^{2}(\omega(\xi-\xi_{0})\,|\,k)}.
	\end{align*}
	Integrating in $\xi$ and using the Jacobi variable
	$u=\omega(\xi-\xi_{0})$, $d\xi = du/\omega$, we get
	\begin{align*}
	\phi(\xi)-\phi_{0} 
	= &\; \frac{C-aR}{2R(R-e_{1})\,\omega}\,
	\Pi\!\left(\omega(\xi-\xi_{0})\,\Big|\,
	\frac{e_{2}-e_{1}}{R-e_{1}},k\right) \\
	& + \frac{C+aR}{2R(R+e_{1})\,\omega}\,
	\Pi\!\left(\omega(\xi-\xi_{0})\,\Big|\,-\frac{e_{2}-e_{1}}{R+e_{1}},k\right),
	\end{align*}
	where
	\[
	\Pi(u\,|\,n,k) := \int_{0}^{u}\frac{d\tau}{1-n\,\mathrm{sn}^{2}(\tau\,|\,k)}.
	\]
	This is precisely \eqref{Eq2.14}, and proves the explicit formula for $\phi$.
	
	Since $\mathrm{sn}(\cdot\,|\,k)$ has real period $4K(k)$, the function
	$\mathrm{sn}^{2}(\cdot\,|\,k)$ has period $2K(k)$, and hence $u$ and $z$ have period
	\[
	\mathcal{T} = \frac{2K(k)}{\omega}= \frac{2K\!\left(\frac{e_{2}-e_{1}}{e_{3}-e_{1}}\right)}{\sqrt{\frac{\Omega}{2R^{2}}(e_{3}-e_{1})}}.
	\]
	Thus $x_{3}(\xi)=Rz(\xi)$ is periodic of period $\mathcal{T}$. The curve $\mathbf{x}$ given by \eqref{Eq2.13} is periodic if, and only if,
	\[
	\phi(\xi_{0}+\mathcal{T}) - \phi(\xi_{0}) \in 2\pi\mathbb{Q}.
	\]
	
	Under the change of variable $u=\omega(\xi-\xi_{0})$ and using
	\[
	\int_{0}^{2K(k)}\frac{d\xi}{1-\eta\,\mathrm{sn}^{2}(\xi\,|\,k)}
	= 2\,\Pi(\eta\,|\,k),
	\]
	where now $\Pi(\eta\,|\,k)$ denotes the \emph{complete} elliptic integral of the third kind, we obtain
	\[
	\phi(\xi_{0}+\mathcal{T}) - \phi(\xi_{0})= \frac{1}{\omega}\left[
	\frac{C-aR}{R(R-e_{1})}\,\Pi\!\left(\frac{e_{2}-e_{1}}{R-e_{1}}\,\Big|\,k\right)+\frac{C+aR}{R(R+e_{1})}\,\Pi\!\left(-\frac{e_{2}-e_{1}}{R+e_{1}}\,\Big|\,k\right)\right].
	\]
	The periodicity condition therefore becomes \eqref{Eq2.15}, i.e.
	\[
	\phi(\xi_{0}+\mathcal{T}) - \phi(\xi_{0})= 2\pi\,\frac{p}{q}
	\quad\Longleftrightarrow\quad\eqref{Eq2.15} \text{ holds}.
	\]
    If \eqref{Eq2.15} holds for some $p,q\in\mathbb{Z}$ with $q\neq 0$, then the minimal period
	of $\mathbf{x}$ is $q\mathcal{T}$; otherwise, $\phi$ has an irrational rotation number with respect to $\mathcal{T}$. Hence $\mathbf{x}$ is quasi-periodic. This proves (b).
	
	\medskip\noindent
	\emph{Proof of item {\rm (c)}.}
	Assume now that $P$ has two real roots $e_{1}<e_{2}=R$, with $e_{2}$ of multiplicity $2$.
	By the explicit formula \eqref{Eq2.8} for $P(\pm R)$,
	\[
	P(R)=-\frac{(C-aR)^{2}}{R^{2}},
	\]
	the fact that $R$ is a root of $P$ implies $C=aR$. In particular, the equation for $\phi$
	simplifies considerably: using \eqref{Eq2.11},
	\[
	\dot\phi(\xi)= \frac{C-aRz(\xi)}{R^{2}(1-z(\xi)^{2})}= \frac{aR(1-z(\xi))}{R^{2}(1-z(\xi)^{2})}= \frac{a}{R(1+z(\xi))}.
	\]
	
	On the other hand, by Proposition~\ref{Pr2.4}{\rm(c)}, every non-constant solution of \eqref{Eq2.5} with values in
	$[-R,R]$ is given by
	\[
	u(\xi)= e_{1}+(R-e_{1})\,\tanh^{2}\!\big(\omega(\xi-\xi_{0})\big),
	\qquad\omega = \sqrt{\frac{\Omega}{2R^{2}}(R-e_{1})},
	\]
	and hence
	\[
	z(\xi)=\frac{u(\xi)}{R}= \frac{1}{R}\left(e_{1}+(R-e_{1})\,\tanh^{2}\!\big(\omega(\xi-\xi_{0})\big)\right).
	\]
	Writing $\theta := \omega(\xi-\xi_{0})$, $t:=\tanh\theta$, we have
	\[
	1+z(\xi) = \frac{R+e_{1}+(R-e_{1})t^{2}}{R}.
	\]
	Thus
	\[
	\dot\phi(\xi)= \frac{a}{R(1+z(\xi))}= \frac{a}{R+e_{1}+(R-e_{1})t^{2}}.
	\]
	Since $t=\tanh\theta$ and $d\theta/d\xi = \omega$, we obtain
	\[
	\phi(\xi)-\phi_{0}= \frac{a}{\omega}\int_{0}^{\theta}
	\frac{d\vartheta}{R+e_{1}+(R-e_{1})\tanh^{2}\vartheta}.
	\]
	Set $\alpha := R+e_{1}$ and $\beta := R-e_{1}$,
	so that $\alpha,\beta>0$ and $\alpha+\beta=2R$.  A direct computation (or the change of variables $y=\tanh\vartheta$) shows that, for
	$\alpha,\beta>0$,
	\[
	\int \frac{dx}{\alpha+\beta\tanh^{2}x}= \frac{x}{\alpha+\beta}+ \frac{1}{\alpha+\beta}\sqrt{\frac{\beta}{\alpha}}\arctan\!\left(\sqrt{\frac{\beta}{\alpha}}\tanh x\right)+ \text{\rm const}.
	\]
	Applying this formula with $x=\vartheta$ and evaluating between $0$ and $\theta$, we obtain
	\[
	\phi(\xi)-\phi_{0}= \frac{a}{\omega(\alpha+\beta)}\left[\theta+ \sqrt{\frac{\beta}{\alpha}}\,\arctan\!\left(\sqrt{\frac{\beta}{\alpha}}\tanh\theta\right)\right].
	\]
	Since $\alpha+\beta=2R$ and
	\[
	\sqrt{\frac{\beta}{\alpha}}= \sqrt{\frac{R-e_{1}}{R+e_{1}}}=:\delta,
	\]
	we deduce that
	\begin{align*}
		\phi(\xi) &= \phi_{0}+ \frac{a}{2R\omega}\left[
		\theta+ \delta\,\arctan\!\big(\delta\tanh\theta\big)\right] \\
		&= \phi_{0}+ \frac{a}{2R}\left(\xi-\xi_{0}+ \frac{\delta}{\omega}\,\arctan\!\big(\delta\tanh(\omega(\xi-\xi_{0}))\big)\right),
	\end{align*}
	which is precisely the expression stated in item~(c).
	
	Finally, from the formula for $z$ we see that $z(\xi)\to 1$ and hence $x_{3}(\xi)\to R$ as
	$|\xi|\to\infty$. Since $\phi(\xi)$ grows roughly linearly with $\xi$, the curve
	$\mathbf{x}(\xi)$ converges to the north pole $(0,0,R)$ while winding infinitely many times
	around the axis, and is therefore non-periodic. This proves item (c) and completes the proof
	of the theorem.
\end{proof}

The geometric content of Theorem~\ref{TP} is schematically represented in Figure~\ref{F1}. The picture is intended only as a qualitative illustration of the different types of non-polar profiles described in the theorem.

\begin{figure}[h!]
	\centering
	\includegraphics[scale=0.25]{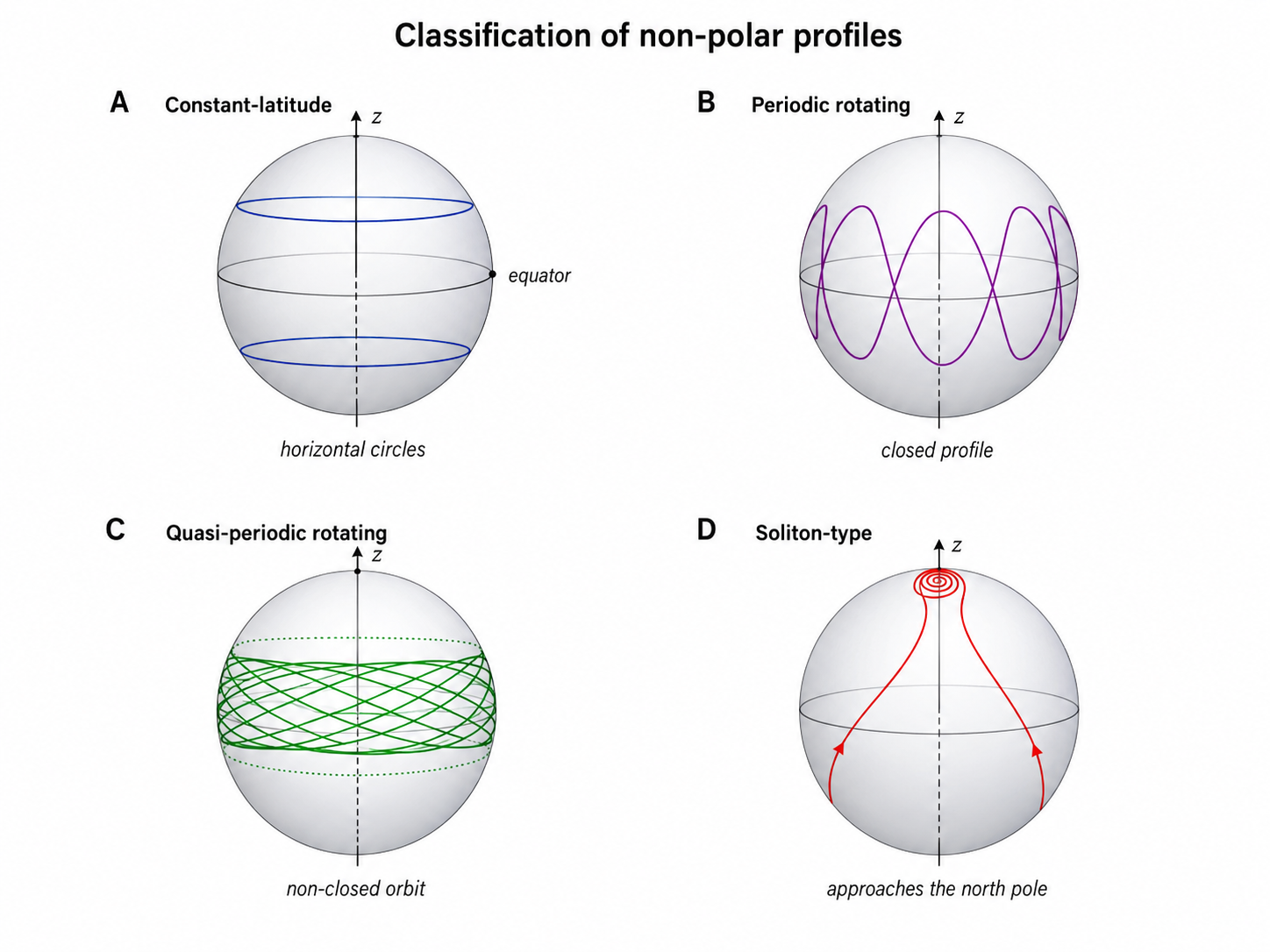}
	\caption{Schematic representation of the non-polar profiles described in Theorem~\ref{TP}.}
		\label{F1}
\end{figure}

\section{Bifurcation analysis around the equatorial branch}
\label{S3}

In the previous section we have obtained a complete explicit description of all non-polar solutions
of the reduced equation
\begin{equation}
	\label{Eq3.1}
\mathbf{x}\times \ddot{\mathbf{x}} - \Omega \, e_{3}\times \mathbf{x} + a\, \dot{\mathbf{x}}=0, \qquad \|\mathbf{x}\|\equiv R,
\end{equation}
In particular, constant solutions of the scalar ODE for $u=x_{3}$ correspond to horizontal
circles on the sphere, and the case $u\equiv 0$ describes rotating waves on the equator
of $\mathbb{S}^{2}_{R}$.

In this section we focus on the family of equatorial solutions and study, by means
of bifurcation theory, how nontrivial oscillatory solutions in the vertical
component \(x_{3}\) bifurcate from this branch as the parameter
\[
\lambda=\frac{a^{2}}{\Omega}
\]
is varied. We shall assume throughout this section that \(a>0\) and
\(\Omega>0\). We start by characterising precisely the equatorial branch and its
associated constants of motion.

\begin{proposition}[Equatorial rotating waves]
	Fix \(R>0\), \(a>0\) and \(\Omega>0\). Set $\omega:=\frac{\Omega}{a}$. Then, for every \(\phi_{0}\in\mathbb R\), the curve
	\begin{equation}
		\label{Eq3.2}
		\mathbf{x}_{\phi_0}(\xi)=\big(R\cos(\omega\xi+\phi_{0}),R\sin(\omega\xi+\phi_{0}),0\big)
	\end{equation}
	is a solution of \eqref{Eq3.1}. Its conserved quantities are
	\begin{equation*}
		C=\frac{\Omega R^{2}}{a}, \qquad E=-\frac{\Omega^{2}R^{2}}{2a}.
	\end{equation*}
	Moreover, the associated vertical polynomial \(P\) satisfies
	\[
	P(u)=u^{2}(a_{3}u-a_{2}), \qquad a_{2}=\frac{a^{4}+\Omega^{2}R^{2}}{a^{2}R^{2}}>0,
	\]
	and therefore \(u=0\) is a root of multiplicity two.
\end{proposition}

\begin{proof}
	The existence of the equatorial branch follows from the constant-latitude
	classification in Theorem~\ref{TP}(a). Indeed, taking \(e=0\), the compatibility
	condition $e\nu^{2}-a\nu+\Omega=0$
	reduces to $-a\nu+\Omega=0$,
	hence
	\[
	\nu=\frac{\Omega}{a}=\omega.
	\]
	This gives precisely \eqref{Eq3.2}. We compute the conserved quantities. Writing $\theta(\xi):=\omega\xi+\phi_0$,
	we have
	\[
	\dot{\mathbf{x}}_{\phi_0}=R\omega(-\sin\theta,\cos\theta,0),
	\]
	and therefore $x_1\dot x_2-x_2\dot x_1=R^2\omega$.
	Since \(x_3\equiv0\), it follows from \eqref{Eq2.3} that
	\[
	C=R^2\omega=\frac{\Omega R^2}{a}.
	\]
	Moreover, $\|\dot{\mathbf{x}}_{\phi_0}\|^2=R^2\omega^2$ and $x_2\dot x_1-x_1\dot x_2=-R^2\omega$.
	Thus
	\[
	E=\frac a2 R^2\omega^2-\Omega R^2\omega=R^2\left(\frac a2\omega^2-\Omega\omega\right).
	\]
	Using \(\omega=\Omega/a\), we obtain
	\[
	E=R^2\left(\frac{\Omega^2}{2a}-\frac{\Omega^2}{a}\right)=-\frac{\Omega^2R^2}{2a}.
	\]
	It remains to identify the corresponding vertical polynomial. Substituting
	\[
	C=\frac{\Omega R^2}{a},\qquad E=-\frac{\Omega^2R^2}{2a}
	\]
	into the coefficients of \(P\), equation \eqref{Eq2.6}, one obtains $a_0=0$, $a_1=0$,
	and
	\[
	a_2=\frac{2E+a^3+2\Omega C}{aR^2}=\frac{a^4+\Omega^2R^2}{a^2R^2}>0.
	\]
	Consequently
	\[
	P(u)=a_3u^3-a_2u^2=u^2(a_3u-a_2),
	\]
	and \(u=0\) is a root of multiplicity two.
\end{proof}

The equatorial branch \eqref{Eq3.2} will be the reference family from
which we study bifurcations. We concentrate on the particular solution
\[
\mathbf{x}(\xi) = \big(R\cos(\omega \xi),\, R \sin(\omega \xi),\,0\big),
\]
corresponding to the choice $\phi_{0}=0$ in \eqref{Eq3.2}. Since
equation \eqref{Eq3.1} is invariant under rotations around the $x_{3}$–axis, different
values of $\phi_{0}$ are related by a rigid rotation and do not change the local
bifurcation picture. Thus, fixing $\phi_{0}=0$ entails no loss of generality.

For the bifurcation analysis it is convenient to work with $2\pi$–periodic curves. We
fix the period to be $2\pi$ and rescale the independent variable by
\[
\eta := \omega \xi, \quad\text{so that}\quad \mathbf{x}_{0}(\eta) = \big(R\cos\eta,\,R\sin\eta,\,0\big)
\]
is $2\pi$–periodic. Writing $\mathbf{x}(\xi)=\mathbf{y}(\eta)$ with $\eta=\omega\xi$, we have
\[
\dot{\mathbf{x}}(\xi) = \omega\,\mathbf{y}'(\eta), \qquad \ddot{\mathbf{x}}(\xi) = \omega^{2}\,\mathbf{y}''(\eta),
\]
where primes denote derivatives with respect to $\eta$. Substituting into \eqref{Eq3.1} yields
\[
\mathbf{y}\times \omega^{2}\mathbf{y}'' - \Omega\,e_{3}\times\mathbf{y} + a\,\omega\,\mathbf{y}' = 0.
\]
Dividing by $\omega^{2}$ and using the relation $\omega=\Omega/a$, we obtain
\[
\mathbf{y}\times \mathbf{y}''- \frac{\Omega}{\omega^{2}}\,e_{3}\times\mathbf{y}+ \frac{a}{\omega}\,\mathbf{y}'= \mathbf{y}\times \mathbf{y}'' - \lambda\, e_{3}\times\mathbf{y}+ \lambda\,\mathbf{y}' = 0,
\]
where we have introduced the bifurcation parameter
\[
\lambda := \frac{a^{2}}{\Omega}.
\]
Renaming $\mathbf{y}$ back to $\mathbf{x}$, we arrive at the rescaled problem
\begin{equation}
	\label{Eq3.3}
	\mathbf{x}\times \mathbf{x}''(\eta)- \lambda\, e_{3}\times\mathbf{x}(\eta) + \lambda\,\mathbf{x}'(\eta) = 0, \qquad \|\mathbf{x}(\eta)\|\equiv R,\quad \mathbf{x}(\eta+2\pi)=\mathbf{x}(\eta),
\end{equation}
with parameter $\lambda = a^{2}/\Omega$. In these variables the equatorial branch takes
the particularly simple form
\[
\mathbf{x}_{0}(\eta) = \big(R\cos\eta,\,R\sin\eta,\,0\big),\qquad \eta\in\mathbb{R},
\]
which will serve as our reference solution for the bifurcation analysis.

We now write a general solution as a perturbation of $\mathbf{x}_{0}$:
\[
\mathbf{x}(\eta) = \mathbf{x}_{0}(\eta) + \mathbf{u}(\eta),
\]
so that $\mathbf{u}\equiv 0$ corresponds to the equatorial branch. Substituting this
ansatz into \eqref{Eq3.3} gives
\[
\big(\mathbf{x}_{0}+\mathbf{u}\big)\times\big(\mathbf{x}_{0}''+\mathbf{u}''\big)- \lambda\,e_{3}\times\big(\mathbf{x}_{0}+\mathbf{u}\big)+ \lambda\,\big(\mathbf{x}_{0}'+\mathbf{u}'\big)=0.
\]
Since $\mathbf{x}_{0}$ itself solves \eqref{Eq3.3},
we are left with the perturbation equation
\begin{equation}
	\label{Eq3.4}
	\mathbf{x}_{0}\times \mathbf{u}'' + \mathbf{u}\times \mathbf{x}_{0}'' + \mathbf{u}\times \mathbf{u}''- \lambda\, e_{3}\times \mathbf{u}+ \lambda\,\mathbf{u}' = 0,\qquad \mathbf{u}(\eta+2\pi)=\mathbf{u}(\eta).
\end{equation}
The geometric constraint $\|\mathbf{x}(\eta)\|=R$ becomes
\[
\mathbf{x}_{0}(\eta)\cdot \mathbf{u}(\eta)+ \tfrac{1}{2}\|\mathbf{u}(\eta)\|^{2}=0.
\]
In particular, at the linear level admissible perturbations satisfy $\mathbf{x}_{0}(\eta)\cdot \mathbf{u}(\eta)=0$, so $\mathbf{u}(\eta)$ takes values in the tangent space
$T_{\mathbf{x}_{0}(\eta)}\mathbb{S}^{2}_{R}$ along the equator. The nonlinear equation
\eqref{Eq3.4}, together with this constraint, will be the starting point for our
bifurcation analysis in the next subsection.

\subsection{Nonlinear operator in the augmented setting}

We now recast the bifurcation problem around the equatorial rotating wave as a nonlinear
operator equation on a suitable Banach space, incorporating the spherical constraint into
the operator itself.

Recall that, after the rescaling $\eta=\omega\xi$ with $\omega=\Omega/a$, the reduced
equation \eqref{Eq3.1} becomes the $2\pi$–periodic problem
\begin{equation}
	\label{Eq3.5}
	\mathbf{x}(\eta)\times \mathbf{x}''(\eta)- \lambda\, e_{3}\times\mathbf{x}(\eta)+ \lambda\,\mathbf{x}'(\eta) = 0,\qquad \|\mathbf{x}(\eta)\|\equiv R,\quad \mathbf{x}(\eta+2\pi)=\mathbf{x}(\eta),
\end{equation}
with bifurcation parameter
\[
\lambda := \frac{a^{2}}{\Omega}>0.
\]
In these variables, the equatorial rotating wave reads
\[
\mathbf{x}_{0}(\eta) = \big(R\cos\eta,\,R\sin\eta,\,0\big),\qquad \eta\in\mathbb{T}:=\mathbb{R}/2\pi\mathbb{Z},
\]
and is a $2\pi$–periodic solution of \eqref{Eq3.5}. We describe any nearby solution as a perturbation of $\mathbf{x}_{0}$,
\[
\mathbf{x}(\eta) = \mathbf{x}_{0}(\eta) + \mathbf{u}(\eta),
\]
where $\mathbf{u}:\mathbb{T}\to\mathbb{R}^{3}$ is $2\pi$–periodic. Substituting into
\eqref{Eq3.5} and using that $\mathbf{x}_{0}$ solves the equation, we obtain
the perturbation equation
\begin{equation}
	\label{Eq3.6}
	\mathbf{x}_{0}\times \mathbf{u}'' + \mathbf{u}\times \mathbf{x}_{0}'' + \mathbf{u}\times \mathbf{u}''- \lambda\, e_{3}\times \mathbf{u}+ \lambda\,\mathbf{u}' = 0.
\end{equation}
The spherical constraint $\|\mathbf{x}(\eta)\|=R$ becomes
\begin{equation}
	\label{Eq3.7}
	\mathbf{x}_{0}(\eta)\cdot\mathbf{u}(\eta)+ \tfrac{1}{2}\|\mathbf{u}(\eta)\|^{2}=0\qquad\text{for all }\eta\in\mathbb{T}.
\end{equation}

Instead of restricting a priori to perturbations tangent to the sphere, we encode both the
differential equation \eqref{Eq3.6} and the constraint
\eqref{Eq3.7} in a single augmented nonlinear operator. 

\begin{definition}[Augmented nonlinear operator]
	For $\lambda>0$ and $\mathbf{u}\in H^{2}(\mathbb{T},\mathbb{R}^{3})$ we define
	the differential part
	\[
	\mc{F}(\lambda,\mathbf{u})(\eta):= \mathbf{x}_{0}(\eta)\times \mathbf{u}''(\eta)+ \mathbf{u}(\eta)\times \mathbf{x}_{0}''(\eta)+ \mathbf{u}(\eta)\times \mathbf{u}''(\eta)- \lambda\, e_{3}\times \mathbf{u}(\eta)+ \lambda\,\mathbf{u}'(\eta),
	\]
	and the constraint part
	\[
	\mc{C}(\mathbf{u})(\eta):= \mathbf{x}_{0}(\eta)\cdot\mathbf{u}(\eta)+ \tfrac{1}{2}\|\mathbf{u}(\eta)\|^{2}.
	\]
	The augmented operator is defined by
	\[
	\mc{G}:(0,+\infty)\times H^{2}(\mathbb{T},\mathbb{R}^{3})
	\longrightarrow  L^{2}(\mathbb{T},\mathbb{R}^{3})\times H^{2}(\mathbb{T}), \qquad
	\mc{G}(\lambda,\mathbf{u}):= \big(\mc{F}(\lambda,\mathbf{u}),\,\mc{C}(\mathbf{u})\big).
	\]
\end{definition}

The equation $\mc{G}(\lambda,\mathbf{u})=(0,0)$ is equivalent to the fact that $\mathbf{x}=\mathbf{x}_{0}+\mathbf{u}$ is a $2\pi$–periodic solution of
\eqref{Eq3.5} lying on the sphere of radius $R$.

We first verify that the operator $\mc{G}$ is well defined.

\begin{lemma}
The operator
	\[
	\mc{G}:(0,+\infty)\times H^{2}(\mathbb{T},\mathbb{R}^{3})\longrightarrow  L^{2}(\mathbb{T},\mathbb{R}^{3})\times H^{2}(\mathbb{T})
	\]
	is well defined and is continuous.
\end{lemma}

\begin{proof}
	Let $\mathbf{u}\in H^{2}(\mathbb{T},\mathbb{R}^{3})$. Since $\mathbb{T}$ is one–dimensional
	and compact, the Sobolev embedding theorem yields the continuous embedding $H^{2}(\mathbb{T})\hookrightarrow \mc{C}^{1}(\mathbb{T})$. Applied componentwise, this implies:
	\begin{itemize}
		\item[(i)] $\mathbf{u}$ and $\mathbf{u}'$ are continuous and bounded on $\mathbb{T}$, hence
		$\mathbf{u},\mathbf{u}'\in L^{\infty}(\mathbb{T},\mathbb{R}^{3})$;
		\item[(ii)] $\mathbf{u}''\in L^{2}(\mathbb{T},\mathbb{R}^{3})$ by definition of $H^{2}$.
	\end{itemize}
	The reference curve $\mathbf{x}_{0}(\eta) = (R\cos\eta,R\sin\eta,0)$ is smooth and periodic,
	so both $\mathbf{x}_{0}$ and $\mathbf{x}_{0}''$ are bounded functions on $\mathbb{T}$, hence
	$\mathbf{x}_{0},\mathbf{x}_{0}''\in L^{\infty}(\mathbb{T},\mathbb{R}^{3})$.	We now inspect each term in the definition of $\mc{F}(\lambda,\mathbf{u})$.
	
	\medskip
	\begin{enumerate}
	\item[{\rm(a)}] The term $\mathbf{x}_{0}\times \mathbf{u}''$.
	Since $\mathbf{x}_{0}\in L^{\infty}$ and $\mathbf{u}''\in L^{2}$, the pointwise estimate
	\[
	\big|\mathbf{x}_{0}(\eta)\times \mathbf{u}''(\eta)\big|\leq \|\mathbf{x}_{0}\|_{L^{\infty}}\,|\mathbf{u}''(\eta)|
	\]
	implies
	\[
	\|\mathbf{x}_{0}\times \mathbf{u}''\|_{L^{2}}\leq \|\mathbf{x}_{0}\|_{L^{\infty}}\,\|\mathbf{u}''\|_{L^{2}}<\infty.
	\]
	Hence $\mathbf{x}_{0}\times \mathbf{u}''\in L^{2}(\mathbb{T},\mathbb{R}^{3})$.
	
	\medskip
	\item[{\rm (b)}] The term $\mathbf{u}\times \mathbf{x}_{0}''$.
	Here $\mathbf{u}\in L^{\infty}$ and $\mathbf{x}_{0}''\in L^{2}$ (indeed
	$\mathbf{x}_{0}''$ is smooth, so it also belongs to $L^{2}$). We obtain
	\[
	\|\mathbf{u}\times \mathbf{x}_{0}''\|_{L^{2}}\leq \|\mathbf{u}\|_{L^{\infty}}\,\|\mathbf{x}_{0}''\|_{L^{2}}<\infty,
	\]
	so $\mathbf{u}\times \mathbf{x}_{0}''\in L^{2}(\mathbb{T},\mathbb{R}^{3})$.
	
	\medskip
	\item[{\rm(c)}] The term $\mathbf{u}\times \mathbf{u}''$.
	Since $\mathbf{u}\in L^{\infty}$ and $\mathbf{u}''\in L^{2}$, we have
	\[
	\|\mathbf{u}\times \mathbf{u}''\|_{L^{2}}\leq \|\mathbf{u}\|_{L^{\infty}}\,
	\|\mathbf{u}''\|_{L^{2}}<\infty,
	\]
	so $\mathbf{u}\times \mathbf{u}''\in L^{2}(\mathbb{T},\mathbb{R}^{3})$.
	
	\medskip
	\item[\rm{(d)}] The term $e_{3}\times \mathbf{u}$.
	The vector $e_{3}$ is constant and $\mathbf{u}\in L^{2}$, so
	\[
	\|e_{3}\times \mathbf{u}\|_{L^{2}}\leq |e_{3}|\,\|\mathbf{u}\|_{L^{2}}= \|\mathbf{u}\|_{L^{2}}<\infty.
	\]
	The same estimate, multiplied by the constant $\lambda>0$, shows that
	$-\lambda\,e_{3}\times\mathbf{u}\in L^{2}(\mathbb{T},\mathbb{R}^{3})$.
	
	\medskip
	\item[\rm{(e)}] The term $\lambda\,\mathbf{u}'$.
	By definition of $H^{2}$, we have $\mathbf{u}'\in H^{1}(\mathbb{T},\mathbb{R}^{3})$, so in
	particular $\mathbf{u}'\in L^{2}(\mathbb{T},\mathbb{R}^{3})$. Therefore
	$\lambda\,\mathbf{u}'\in L^{2}(\mathbb{T},\mathbb{R}^{3})$.
\end{enumerate}

	\medskip
	
	Combining (a)–(e), we conclude that every term in the expression of
	$\mc{F}(\lambda,\mathbf{u})$ lies in $L^{2}(\mathbb{T},\mathbb{R}^{3})$. Hence, $\mc{F}$ is well-defined. Moreover the same estimates can be used to prove its continuity.
	
	We now consider the constraint operator $\mc{C}$. Since $\mathbb{T}$ is one--dimensional, we have
	the Sobolev embedding $H^{2}(\mathbb{T})\hookrightarrow L^{\infty}(\mathbb{T})$ and the algebra property
	$H^{2}(\mathbb{T})\cdot H^{2}(\mathbb{T})\subset H^{2}(\mathbb{T})$, with a bound
	\[
	\|fg\|_{H^{2}}\le C\|f\|_{H^{2}}\|g\|_{H^{2}}.
	\]
	Write $\mathbf{u}=(u_{1},u_{2},u_{3})$ and $\mathbf{x}_{0}=(x_{0,1},x_{0,2},x_{0,3})$. Then
	\[
	\mathbf{x}_{0}\cdot\mathbf{u}=\sum_{j=1}^{3}x_{0,j}u_{j}\in H^{2}(\mathbb{T}),
	\qquad\|\mathbf{x}_{0}\cdot\mathbf{u}\|_{H^{2}}\le C\|\mathbf{x}_{0}\|_{\mc{C}^{2}}\|\mathbf{u}\|_{H^{2}}.
	\]
	Moreover,
	\[
	\|\mathbf{u}\|^{2}=u_{1}^{2}+u_{2}^{2}+u_{3}^{2}\in H^{2}(\mathbb{T}),\qquad\|\|\mathbf{u}\|^{2}\|_{H^{2}}\le C\|\mathbf{u}\|_{H^{2}}^{2},
	\]
	by the algebra property. Therefore $\mc{C}(\mathbf{u})\in H^{2}(\mathbb{T})$ for every
	$\mathbf{u}\in H^{2}(\mathbb{T};\R^{3})$, and $\mc{C}:H^{2}(\mathbb{T},\R^{3})\to H^{2}(\mathbb{T})$ is well defined. The same inequalities can be used to prove its continuity. Hence, we conclude that $\mc{G}=(\mc{F},\mc{C})$ is well-defined and continuous.
\end{proof}

The next result computes the first Fréchet derivative of $\mc{G}$.

\begin{lemma}[First Fréchet derivative of \(\mc{G}\)]
	The map
	\[
	\mc{G}:(0,+\infty)\times H^{2}(\mathbb{T},\mathbb{R}^{3})
	\longrightarrow  L^{2}(\mathbb{T},\mathbb{R}^{3})\times H^{2}(\mathbb{T})
	\]
	is of class \(\mc{C}^{\infty}\). For \((\lambda,\mathbf{u})\in(0,+\infty)\times H^{2}(\mathbb{T},\mathbb{R}^{3})\)
	and an increment \((\mu,\mathbf{v})\in\mathbb{R}\times H^{2}(\mathbb{T},\mathbb{R}^{3})\), the Fréchet
	derivative is given by
	\[
	D\mc{G}(\lambda,\mathbf{u})[\mu,\mathbf{v}]= \big(D_{\lambda}\mc{F}(\lambda,\mathbf{u})[\mu]+ D_{\mathbf{u}}\mc{F}(\lambda,\mathbf{u})[\mathbf{v}],\;
	D_{\mathbf{u}}\mc{C}(\mathbf{u})[\mathbf{v}]\big),
	\]
	where
	\begin{align*}
		D_{\lambda}\mc{F}(\lambda,\mathbf{u})[\mu]
		&= \mu\big(- e_{3}\times \mathbf{u} + \mathbf{u}'\big),\\[4pt]
		D_{\mathbf{u}}\mc{F}(\lambda,\mathbf{u})[\mathbf{v}]
		&= \mathbf{x}_{0}\times \mathbf{v}''+ \mathbf{v}\times \mathbf{x}_{0}''+ \mathbf{v}\times \mathbf{u}'' + \mathbf{u}\times \mathbf{v}''- \lambda\, e_{3}\times \mathbf{v}+ \lambda\,\mathbf{v}',\\[4pt] D_{\mathbf{u}}\mc{C}(\mathbf{u})[\mathbf{v}]
		&= \mathbf{x}_{0}\cdot\mathbf{v} + \mathbf{u}\cdot\mathbf{v}.
	\end{align*}
	In particular, at any point \((\lambda,0)\) the derivative simplifies to
	\[
	D\mc{G}(\lambda,0)[\mu,\mathbf{v}]= \big( L_{\lambda}\mathbf{v},\; \mathbf{x}_{0}\cdot\mathbf{v}\big),
	\]
	where
	\begin{equation*}
		L_{\lambda}\mathbf{v}:= \mathbf{x}_{0}\times \mathbf{v}''+ \mathbf{v}\times \mathbf{x}_{0}''- \lambda\, e_{3}\times \mathbf{v}+ \lambda\,\mathbf{v}'.
	\end{equation*}
	Equivalently,
	\[
	D_{\lambda}\mc{G}(\lambda,0) = (0,0),\qquad
	D_{\mathbf{u}}\mc{G}(\lambda,0)[\mathbf{v}]= \big(L_{\lambda}\mathbf{v},\,\mathbf{x}_{0}\cdot\mathbf{v}\big).
	\]
\end{lemma}

\begin{proof}
	Firstly, let us give the formal derivatives. The dependence on \(\lambda\) is affine and only in the last two terms of \(\mc{F}\), so for
	\(\mu\in\mathbb{R}\),
	\[
	D_{\lambda}\mc{F}(\lambda,\mathbf{u})[\mu]= \mu\big(- e_{3}\times \mathbf{u} + \mathbf{u}'\big).
	\]
	The constraint \(\mc{C}\) does not depend on \(\lambda\), so \(D_{\lambda}\mc{C}(\mathbf{u})=0\).

		To compute the derivative with respect to \(\mathbf{u}\), we use linearity of differentiation and the cross product, and the fact that
		\(\mathbf{u}\mapsto \mathbf{u}\times\mathbf{u}''\) is quadratic. For
		\(\mathbf{u},\mathbf{v}\in H^{2}(\mathbb{T},\mathbb{R}^{3})\),
		\begin{align*}
			\mc{F}(\lambda,\mathbf{u}+\varepsilon\mathbf{v})
			&= \mathbf{x}_{0}\times (\mathbf{u}+\varepsilon\mathbf{v})''
			+ (\mathbf{u}+\varepsilon\mathbf{v})\times \mathbf{x}_{0}''
			+ (\mathbf{u}+\varepsilon\mathbf{v})\times (\mathbf{u}+\varepsilon\mathbf{v})'' \\
			&\quad - \lambda\, e_{3}\times (\mathbf{u}+\varepsilon\mathbf{v})+ \lambda\,(\mathbf{u}+\varepsilon\mathbf{v})'.
		\end{align*}
		Expanding the linear terms gives
		\[
		\mathbf{x}_{0}\times (\mathbf{u}+\varepsilon\mathbf{v})''= \mathbf{x}_{0}\times \mathbf{u}'' + \varepsilon\,\mathbf{x}_{0}\times \mathbf{v}'',
		\]
		\[
		(\mathbf{u}+\varepsilon\mathbf{v})\times \mathbf{x}_{0}''= \mathbf{u}\times \mathbf{x}_{0}'' + \varepsilon\,\mathbf{v}\times \mathbf{x}_{0}'',
		\]
		\[
		-\lambda\,e_{3}\times (\mathbf{u}+\varepsilon\mathbf{v})= -\lambda\,e_{3}\times\mathbf{u} - \varepsilon\,\lambda\,e_{3}\times\mathbf{v},
		\]
		\[
		\lambda(\mathbf{u}+\varepsilon\mathbf{v})'= \lambda\mathbf{u}' + \varepsilon\,\lambda\mathbf{v}'.
		\]
		For the quadratic term we use bilinearity:
		\[
		(\mathbf{u}+\varepsilon\mathbf{v})\times(\mathbf{u}+\varepsilon\mathbf{v})''
		= \mathbf{u}\times\mathbf{u}''+ \varepsilon\big(\mathbf{v}\times\mathbf{u}'' + \mathbf{u}\times\mathbf{v}''\big)+ \varepsilon^{2}\,\mathbf{v}\times\mathbf{v}''.
		\]
		Dividing by \(\varepsilon\) and letting \(\varepsilon\to 0\) yields
		\[
		D_{\mathbf{u}}\mc{F}(\lambda,\mathbf{u})[\mathbf{v}]
		= \mathbf{x}_{0}\times \mathbf{v}''+ \mathbf{v}\times \mathbf{x}_{0}''+ \mathbf{v}\times \mathbf{u}'' + \mathbf{u}\times \mathbf{v}''- \lambda\, e_{3}\times \mathbf{v}+ \lambda\,\mathbf{v}'.
		\]
		For the constraint functional $\mc{C}(\mathbf{u}) = \mathbf{x}_{0}\cdot\mathbf{u} + \tfrac12\|\mathbf{u}\|^{2}$,
		the first term is linear, so its derivative is \(\mathbf{x}_{0}\cdot\mathbf{v}\). For the
		quadratic term,
		\[
		\frac{d}{d\varepsilon}
		\left(\tfrac{1}{2}\|\mathbf{u}+\varepsilon\mathbf{v}\|^{2}\right)\Big|_{\varepsilon=0}
		= \mathbf{u}\cdot\mathbf{v},
		\]
		and hence
		\[
		D_{\mathbf{u}}\mc{C}(\mathbf{u})[\mathbf{v}]
		= \mathbf{x}_{0}\cdot\mathbf{v} + \mathbf{u}\cdot\mathbf{v}.
		\]
		Fix \((\lambda,\mathbf{u})\). We now show that
		\[
		D\mc{G}(\lambda,\mathbf{u}):
		\mathbb{R}\times H^{2}(\mathbb{T},\mathbb{R}^{3})
		\longrightarrow L^{2}(\mathbb{T},\mathbb{R}^{3})\times H^{2}(\mathbb{T})
		\]
		is a bounded linear operator. For the derivative with respect to \(\lambda\),
		\[
		\|D_{\lambda}\mc{F}(\lambda,\mathbf{u})[\mu]\|_{L^{2}}
		\le |\mu|\big(\|e_{3}\times\mathbf{u}\|_{L^{2}} + \|\mathbf{u}'\|_{L^{2}}\big)
		\le C\,|\mu|\,\|\mathbf{u}\|_{H^{2}}.
		\]
		For the derivative with respect to \(\mathbf{u}\), using the explicit formula,
		\[
		\|D_{\mathbf{u}}\mc{F}(\lambda,\mathbf{u})[\mathbf{v}]\|_{L^{2}}
		\le C\Big(
		\|\mathbf{v}''\|_{L^{2}}+ \|\mathbf{v}\|_{L^{2}}+ \|\mathbf{v}\times\mathbf{u}''\|_{L^{2}}+ \|\mathbf{u}\times\mathbf{v}''\|_{L^{2}}+ \|\mathbf{v}\|_{L^{2}}+ \|\mathbf{v}'\|_{L^{2}}
		\Big).
		\]
		Using the embedding $H^{2}\hookrightarrow L^{\infty}$, the mixed terms are estimated by
		\[
		\|\mathbf{v}\times\mathbf{u}''\|_{L^{2}}
		\le \|\mathbf{v}\|_{L^{\infty}}\|\mathbf{u}''\|_{L^{2}}
		\le C\|\mathbf{v}\|_{H^{2}}\|\mathbf{u}\|_{H^{2}},
		\]
		\[
		\|\mathbf{u}\times\mathbf{v}''\|_{L^{2}}
		\le \|\mathbf{u}\|_{L^{\infty}}\|\mathbf{v}''\|_{L^{2}}
		\le C\|\mathbf{u}\|_{H^{2}}\|\mathbf{v}\|_{H^{2}}.
		\]
		The remaining terms are directly controlled by \(\|\mathbf{v}\|_{H^{2}}\). Hence there exists a
		constant \(C(\lambda,\|\mathbf{u}\|_{H^{2}})>0\) such that
		\[
		\|D_{\mathbf{u}}\mc{F}(\lambda,\mathbf{u})[\mathbf{v}]\|_{L^{2}}
		\le C(\lambda,\|\mathbf{u}\|_{H^{2}})\,\|\mathbf{v}\|_{H^{2}}.
		\]
		Thus \(D_{\mathbf{u}}\mc{F}(\lambda,\mathbf{u})\) is a bounded linear operator from
		\(H^{2}\) to \(L^{2}\).
		
		Now, we do the same for the component \(\mc{C}\). For the derivative with respect to \(\mathbf{u}\),
		\[
		D_{\mathbf{u}}\mc{C}(\mathbf{u})[\mathbf{v}]
		= \mathbf{x}_{0}\cdot\mathbf{v} + \mathbf{u}\cdot\mathbf{v}.
		\]
		By the algebra property of $H^{2}(\mathbb{T})$, the first term satisfies
		\[
		\|\mathbf{x}_{0}\cdot\mathbf{v}\|_{H^{2}}
		\le C \, \|\mathbf{x}_{0}\|_{\mc{C}^{2}}\|\mathbf{v}\|_{H^{2}}
		\le C\,\|\mathbf{v}\|_{H^{2}},
		\]
		since \(\mathbf{x}_{0}\) is smooth and bounded. The second term satisfies
		\[
		\|\mathbf{u}\cdot\mathbf{v}\|_{H^{2}}
		\le C\|\mathbf{u}\|_{H^{2}}\|\mathbf{v}\|_{H^{2}}.
		\]
		Therefore
		\[
		\|D_{\mathbf{u}}\mc{C}(\mathbf{u})[\mathbf{v}]\|_{H^{2}}
		\le C(1+\|\mathbf{u}\|_{H^{2}})\,\|\mathbf{v}\|_{H^{2}},
		\]
		so \(D_{\mathbf{u}}\mc{C}(\mathbf{u})\) is also a bounded linear operator from \(H^{2}\) to
		\(H^{2}\). Combining these estimates, we see that $D\mc{G}(\lambda,\mathbf{u})
		: \mathbb{R}\times H^{2}\to L^{2}\times H^{2}$ is linear and bounded. 
		
		On the other hand, we prove that, indeed, $D\mc{G}$ is the Fréchet derivative of $\mc{G}$. 	Fix \((\lambda,\mathbf{u})\in\R\times H^{2}(\mathbb{T},\R^{3})\) and consider \((\mu,\mathbf{v})\in \R\times H^{2}(\mathbb{T},\R^{3})\).
		We expand \(\mc{F}(\lambda+\mu,\mathbf{u}+\mathbf{v})\) and subtract the linear terms.
		A direct computation shows
		\begin{equation}
			\label{Eq3.8}
			\mc{F}(\lambda+\mu,\mathbf{u}+\mathbf{v})
			-\mc{F}(\lambda,\mathbf{u})
			-D_{\lambda}\mc{F}(\lambda,\mathbf{u})[\mu]
			-D_{\mathbf{u}}\mc{F}(\lambda,\mathbf{u})[\mathbf{v}]
			=\mathbf{v}\times \mathbf{v}''
			+\mu\big(-e_{3}\times \mathbf{v}+\mathbf{v}'\big).
		\end{equation}
		We estimate the remainder in \(L^{2}\). Using \(H^{2}\hookrightarrow L^{\infty}\),
		\[
		\|\mathbf{v}\times\mathbf{v}''\|_{L^{2}}
		\le \|\mathbf{v}\|_{L^{\infty}}\|\mathbf{v}''\|_{L^{2}}
		\le c\,\|\mathbf{v}\|_{H^{2}}^{2},
		\]
		and
		\[
		\|\mu(-e_{3}\times \mathbf{v}+\mathbf{v}')\|_{L^{2}}
		\le |\mu|\,\big(\|\mathbf{v}\|_{L^{2}}+\|\mathbf{v}'\|_{L^{2}}\big)
		\le c\,|\mu|\,\|\mathbf{v}\|_{H^{2}}.
		\]
		Hence the \(L^{2}\)-norm of the right-hand side of \eqref{Eq3.8} is
		\(o\big(|\mu|+\|\mathbf{v}\|_{H^{2}}\big)\) as \((\mu,\mathbf{v})\to(0,0)\).
		
		For the constraint \(\mc{C}\), one has the exact expansion
		\begin{equation*}
			\mc{C}(\mathbf{u}+\mathbf{v})-\mc{C}(\mathbf{u})-D_{\mathbf{u}}\mc{C}(\mathbf{u})[\mathbf{v}]
			= \tfrac12\|\mathbf{v}\|^{2}.
		\end{equation*}
		Since \(H^{2}(\mathbb{T})\) is a Banach algebra,
		\[
		\Big\|\tfrac12\|\mathbf{v}\|^{2}\Big\|_{H^{2}}
		\le c\,\|\mathbf{v}\|_{H^{2}}^{2}=o\big(\|\mathbf{v}\|_{H^{2}}\big)
		\quad\text{as }\mathbf{v}\to 0.
		\]
		Together with \eqref{Eq3.8}, this proves that \(\mc{G}\) is Fréchet differentiable
		at \((\lambda,\mathbf{u})\) with derivative given in the statement.
		
		To conclude that \(\mc{G}\) is of class \(\mc{C}^{1}\), it remains to check that
		\((\lambda,\mathbf{u})\mapsto D\mc{G}(\lambda,\mathbf{u})\) is continuous in the operator norm.
		Let \(\mathbf{u}_{1},\mathbf{u}_{2}\in H^{2}\) and fix \(\mathbf{v}\in H^{2}\) with
		\(\|\mathbf{v}\|_{H^{2}}=1\). Then
		\[
		D_{\mathbf{u}}\mc{F}(\lambda,\mathbf{u}_{1})[\mathbf{v}]- D_{\mathbf{u}}\mc{F}(\lambda,\mathbf{u}_{2})[\mathbf{v}]= \mathbf{v}\times(\mathbf{u}_{1}''-\mathbf{u}_{2}'')+ (\mathbf{u}_{1}-\mathbf{u}_{2})\times\mathbf{v}''.
		\]
		Hence
		\[
		\big\|D_{\mathbf{u}}\mc{F}(\lambda,\mathbf{u}_{1})- D_{\mathbf{u}}\mc{F}(\lambda,\mathbf{u}_{2})\big\|_{\mathcal{L}(H^{2},L^{2})}
		\le C\,\|\mathbf{u}_{1}-\mathbf{u}_{2}\|_{H^{2}},
		\]
		for some constant \(C\) (on bounded sets in \(H^{2}\)). A similar estimate holds for
		\(D_{\mathbf{u}}\mc{C}\):
		\[
		\big\|D_{\mathbf{u}}\mc{C}(\mathbf{u}_{1})- D_{\mathbf{u}}\mc{C}(\mathbf{u}_{2})\big\|_{\mathcal{L}(H^{2},H^{2})}
		\le C\,\|\mathbf{u}_{1}-\mathbf{u}_{2}\|_{H^{2}}.
		\]
		The dependence on \(\lambda\) is affine, hence continuous. Therefore
		$(\lambda,\mathbf{u})\mapsto D\mc{G}(\lambda,\mathbf{u})$
		is continuous as a map into
		\(\mathcal{L}\big(\mathbb{R}\times H^{2},\,L^{2}\times H^{2}\big)\), and \(\mc{G}\) is of class
		\(\mc{C}^{1}\) on \((0,+\infty)\times H^{2}(\mathbb{T},\mathbb{R}^{3})\). The $\mc{C}^{\infty}$-regularity of $\mc{G}$ is proven by following the same steps.
	\end{proof}

\subsection{Study of the linearization} In this subsection we compute the generalized spectrum of the linearization. The main result is the following.

\begin{proposition}
	\label{Pr3.2}
	Let $\lambda>0$ and consider the linear operator $L_{\lambda}:=D_{\mathbf{u}}\mc{F}(\lambda,0)$ defined by
	\[
	L_{\lambda}:H^{2}(\mathbb{T},\mathbb{R}^{3})\longrightarrow L^{2}(\mathbb{T},\mathbb{R}^{3}),
	\qquad L_{\lambda}\mathbf{v}:= \mathbf{x}_{0}\times \mathbf{v}'' + \mathbf{v}\times \mathbf{x}_{0}'' - \lambda\, e_{3}\times \mathbf{v} + \lambda\,\mathbf{v}'.
	\]
	Let
	\[
	X:=\Big\{\mathbf{v}\in H^{2}(\mathbb{T},\mathbb{R}^{3}) : 
	\mathbf{x}_{0}(\eta)\cdot \mathbf{v}(\eta)=0\ \text{for all } \eta\in\mathbb{T}\Big\}.
	\]
	Then:
	\begin{enumerate}
		\item[{\rm(i)}] The space $X$ is invariant under $L_{\lambda}$, and
		\[
		\ker D_{\mathbf{u}}\mathcal{G}(\lambda,0)=\ker\big(L_{\lambda}\vert_{X}\big).
		\]
		
		\item[{\rm(ii)}] Define $\lambda_{k} := R\sqrt{k^{2}-1}$, $k\in\mathbb{N}$, $k\ge 2$. If $\lambda\neq \lambda_{k}$ for all $k\ge 2$, then
		\[
		\ker\big(L_{\lambda}\vert_{X}\big)=\mathrm{span}\{\mathbf{x}'_{0}\}.
		\]
		
		\item[{\rm(iii)}] For each $k\ge 2$ one has
		\[
		\ker\big(L_{\lambda_{k}}\vert_{X}\big)
		=\mathrm{span}\{\mathbf{x}'_{0},\,\mathbf{v}_{k}^{(1)},\,\mathbf{v}_{k}^{(2)}\},
		\]
		and in particular $\dim\ker(L_{\lambda_k}\vert_X)=3$, where
		\begin{align*}
			\mathbf{v}_{k}^{(1)}(\eta)
			&= \beta_{k}^{(1)}(\eta)\,E_{2}(\eta) + \gamma_{k}^{(1)}(\eta)\,E_{3}, \\
			\mathbf{v}_{k}^{(2)}(\eta)
			&= \beta_{k}^{(2)}(\eta)\,E_{2}(\eta) + \gamma_{k}^{(2)}(\eta)\,E_{3},
		\end{align*}
		with the moving orthonormal frame
		\[
		E_{1}(\eta):=\frac{\mathbf{x}_{0}(\eta)}{R},\quad
		E_{2}(\eta):=\frac{\mathbf{x}_{0}'(\eta)}{R},\quad
		E_{3}:=e_{3},
		\]
		and the scalar coefficients
		\[
		\beta_{k}^{(1)}(\eta)=\cos(k\eta),\qquad
		\gamma_{k}^{(1)}(\eta)=\frac{k}{\sqrt{k^{2}-1}}\sin(k\eta),
		\]
		\[
		\beta_{k}^{(2)}(\eta)=\sin(k\eta),\qquad
		\gamma_{k}^{(2)}(\eta)=-\frac{k}{\sqrt{k^{2}-1}}\cos(k\eta).
		\]
	\end{enumerate}
\end{proposition}

\begin{proof}
	We introduce the moving orthonormal frame along the equator
	\begin{align*}
	\left\{
	\begin{array}{l}
	E_{1}(\eta):=\dfrac{\mathbf{x}_{0}(\eta)}{R}=(\cos\eta,\sin\eta,0), \\
	\\
	E_{2}(\eta):=\dfrac{\mathbf{x}_{0}'(\eta)}{R}=(-\sin\eta,\cos\eta,0), \\
	\\
	E_{3}:=e_{3}=(0,0,1).
	\end{array}
	\right.
	\end{align*}
	A direct computation gives 
	\begin{equation}
		\label{Eq3.9}
	E_{1}'=E_{2}, \quad  E_{2}'=-E_{1}, \quad E_{3}'=0,
    \end{equation}
 and $\mathbf{x}_{0}=R E_{1}$, $\mathbf{x}_{0}'=R E_{2}$, $\mathbf{x}_{0}''=-R E_{1}$. Moreover, the cross products in this frame satisfy
	\begin{equation}
		\label{Eq3.10}
	E_{1}\times E_{2}=E_{3},\qquad
	E_{1}\times E_{3}=-E_{2},\qquad
	E_{3}\times E_{1}=E_{2},\qquad
	E_{3}\times E_{2}=-E_{1}.
	\end{equation}
	These identities will be useful for the forthcoming computations. Any $\mathbf{v}\in H^{2}(\mathbb{T},\mathbb{R}^{3})$ can be written uniquely as
	\[
	\mathbf{v}(\eta)=\alpha(\eta)E_{1}(\eta)
	+\beta(\eta)E_{2}(\eta)
	+\gamma(\eta)E_{3},
	\]
	with $\alpha,\beta,\gamma\in H^{2}(\mathbb{T},\mathbb{R})$. The constraint
	$\mathbf{x}_{0}\cdot\mathbf{v}=0$ is equivalent to
	\[
	\mathbf{x}_{0}\cdot\mathbf{v}
	= R\,E_{1}\cdot(\alpha E_{1}+\beta E_{2}+\gamma E_{3})
	= R\,\alpha = 0,
	\]
	so on $X$ we have $\alpha\equiv 0$, and therefore
	\[
	X
	= \big\{\mathbf{v}=\beta E_{2}+\gamma E_{3} : \beta,\gamma\in H^{2}(\mathbb{T},\mathbb{R})\big\}.
	\]
	For a general $\mathbf{v}=\alpha E_{1}+\beta E_{2}+\gamma E_{3}$, its first derivative is given by
	\begin{align}
		\label{Eq3.11}
	\mathbf{v}'  & = \alpha' E_{1}+\alpha E_{1}'
	+ \beta' E_{2}+\beta E_{2}'
	+ \gamma' E_{3}+\gamma E_{3}' \nonumber \\
	&= (\alpha'-\beta)E_{1} + (\alpha+\beta')E_{2} + \gamma' E_{3},
\end{align} 
	where we have used identities \eqref{Eq3.9}. Setting $A:=\alpha'-\beta$ and $B:=\alpha+\beta'$, we obtain $\mathbf{v}' = A E_{1} + B E_{2} + \gamma' E_{3}$. Differentiating once more and using again identities \eqref{Eq3.9}, it becomes apparent that
	\begin{align*}
	\mathbf{v}''  & = A' E_{1} + A E_{1}'
	+ B' E_{2} + B E_{2}'
	+ \gamma'' E_{3}  = (A'-B)E_{1} + (A+B')E_{2} + \gamma'' E_{3}.
	\end{align*}
	Since $A'=\alpha''-\beta'$, $B'=\alpha'+\beta''$, we obtain
	\begin{align}
		\label{Eq3.12}
	\mathbf{v}'' &=a_{1}E_{1}+a_{2}E_{2}+a_{3}E_{3} \nonumber \\
	&= (\alpha''-\alpha-2\beta')E_{1}
	+ (2\alpha'-\beta+\beta'')E_{2}
	+ \gamma'' E_{3}.
     \end{align}
	We now compute each term in $L_{\lambda}\mathbf{v}$.
	
	\begin{itemize}
	\item[{\rm(a)}] The term $\mathbf{x}_{0}\times\mathbf{v}''$.
	Since $\mathbf{x}_{0}=R E_{1}$ and \eqref{Eq3.12}, we have, by using identities \eqref{Eq3.10}, that
	\begin{align*}
	\mathbf{x}_{0}\times\mathbf{v}'' = R E_{1}\times(a_{1}E_{1}+a_{2}E_{2}+a_{3}E_{3}) = R\big(a_{2}E_{3} - a_{3}E_{2}\big).
	\end{align*}
	Therefore
	\[
	\mathbf{x}_{0}\times\mathbf{v}''
	= -R\gamma'' E_{2} + R(2\alpha'-\beta+\beta'')E_{3}.
	\]
	
	\item[{\rm(b)}] The term $\mathbf{v}\times\mathbf{x}_{0}''$.
	We have $\mathbf{x}_{0}''=-R E_{1}$, so
	\[
	\mathbf{v}\times\mathbf{x}_{0}''
	= \mathbf{v}\times(-R E_{1})
	= -R\,\mathbf{v}\times E_{1}.
	\]
	Using
	\begin{align*}
	\mathbf{v}\times E_{1} &= (\alpha E_{1}+\beta E_{2}+\gamma E_{3})\times E_{1} = \beta(E_{2}\times E_{1})+\gamma(E_{3}\times E_{1}) \\
	& = -\beta E_{3} + \gamma E_{2},
	\end{align*}
	we obtain
	\[
	\mathbf{v}\times\mathbf{x}_{0}''= -R(\gamma E_{2} - \beta E_{3})= -R\gamma E_{2} + R\beta E_{3}.
	\]
	
	\item[{\rm(c)}] The term $-\,\lambda e_{3}\times\mathbf{v}$.
	Since $e_{3}=E_{3}$,
	\begin{align*}
	e_{3}\times\mathbf{v} &= E_{3}\times(\alpha E_{1}+\beta E_{2}+\gamma E_{3})
	= \alpha(E_{3}\times E_{1})+\beta(E_{3}\times E_{2}) \\
	&= \alpha E_{2} - \beta E_{1}.
\end{align*}
	Thus $-\lambda e_{3}\times\mathbf{v}
	= \lambda\beta E_{1} - \lambda\alpha E_{2}$.
	
	\smallskip
	\item[{\rm(d)}] The term $\lambda\mathbf{v}'$. By \eqref{Eq3.11}, 
	\[
	\lambda\mathbf{v}'= \lambda(\alpha'-\beta)E_{1}
	+ \lambda(\alpha+\beta')E_{2}+ \lambda\gamma' E_{3}.
	\]
\end{itemize}
Summing up the contributions in the $E_{1},E_{2},E_{3}$ directions we obtain
\[
L_{\lambda}\mathbf{v}= \lambda\alpha' E_{1}+ \big[-R(\gamma''+\gamma) + \lambda\beta'\big]E_{2}+ \big[R(2\alpha'+\beta'') + \lambda\gamma'\big]E_{3}.
\]
On the subspace $X$ we have $\alpha\equiv 0$ and therefore $\alpha'\equiv 0$. Inserting
	$\alpha\equiv 0$ in the above expression, we obtain
	\[
	L_{\lambda}(\beta E_{2}+\gamma E_{3})
	= \big[-R(\gamma''+\gamma) + \lambda\beta'\big]E_{2}
	+ \big[R\beta'' + \lambda\gamma'\big]E_{3}.
	\]
	Hence $L_{\lambda}\mathbf{v}$ has zero $E_{1}$–component whenever
	$\mathbf{v}=\beta E_{2}+\gamma E_{3}\in X$, so $L_{\lambda}(X)\subset X$. This proves that
	$X$ is invariant under $L_{\lambda}$.
	Therefore
	\[
	\ker D_{\mathbf{u}}\mathcal{G}(\lambda,0)
	= \big\{\mathbf{v}\in H^{2}(\mathbb{T},\mathbb{R}^{3}) :
	L_{\lambda}\mathbf{v}=0,\ \mathbf{x}_{0}\cdot\mathbf{v}=0\big\}
	= \ker\big(L_{\lambda}\vert_{X}\big),
	\]
	which proves item (i).
	Note that $L_{\lambda}\mathbf{v}=0$ if and only if $(\beta,\gamma)$ solves
	\begin{equation}
		\label{Eq3.13}
		\left\{
		\begin{aligned}
			\lambda\beta' - R(\gamma''+\gamma) &= 0,\\
			R\beta'' + \lambda\gamma' &= 0.
		\end{aligned}
		\right.
	\end{equation}
	We expand in Fourier series,
	$$
	\b(\eta)=\sum_{k\in\Z}B_{k} e^{ik\eta}, \qquad \g(\eta)=\sum_{k\in\Z}G_{k}e^{ik\eta}.
	$$
	Substituting into \eqref{Eq3.13}, we obtain, for each $k\in\Z$,
	\[
	\lambda(ikB_{k}) - R(-k^{2}G_{k}+G_{k}) = 0,
	\qquad R(-k^{2}B_{k}) + \lambda(ikG_{k}) = 0,
	\]
	that is,
	\begin{equation}
		\label{Eq3.14}
	\begin{pmatrix}
		i\lambda k & R(k^{2}-1)\\[2pt]
		-Rk^{2} & i\lambda k
	\end{pmatrix}
	\begin{pmatrix} B_{k} \\ G_{k} \end{pmatrix}
	= \begin{pmatrix} 0 \\ 0 \end{pmatrix}.
	\end{equation}
	There exists a non-trivial solution $(B_{k},G_{k})\neq(0,0)$ of this system if and only if the determinant vanishes:
	\[
	(i\lambda k)^{2} - (-Rk^{2})\,R(k^{2}-1) = 0.
	\]
	Since we restrict to $\lambda>0$, this gives the discrete set
	\[
	\lambda = \lambda_{k} := R\sqrt{k^{2}-1}.
	\]
	The case $k=1$ would give $\lambda=0$ and is therefore excluded. For $k=0$ the system reduces to
	\[
	\begin{pmatrix}
		0 & -R\\[2pt]
		0 & 0
	\end{pmatrix}
	\begin{pmatrix} B_{0} \\ G_{0} \end{pmatrix}
	= \begin{pmatrix} 0 \\ 0 \end{pmatrix}.
	\]
	Hence $G_{0}=0$ and, therefore, we obtain $B_{0}\in\R$ as a free parameter. Thus, for $\lambda>0$, if $\lambda\neq\lambda_{k}$ for all integers
	$k\ge 2$,
	$$
	\ker(L_{\lambda}\vert_{X})=\text{span}\{E_{2}(\eta)\}=\text{span}\{\mathbf{x}'_{0}(\eta)\}.
	$$
	For $k\ge 2$, we solve the algebraic system \eqref{Eq3.14}	for $(B_{k},G_{k})$.
	Both equations are equivalent when $\lambda=\lambda_{k}$, so the solution space is
	one–dimensional over $\mathbb{C}$. Choosing $B_{k}=1$ and using the first equation, we obtain
	\[
	G_{k} = -\,i\,\frac{\lambda_{k}k}{R(k^{2}-1)}= -\,i\,\frac{k}{\sqrt{k^{2}-1}}.
	\]
	Thus a complex solution is
	\[
	\beta(\eta)=e^{ik\eta},\qquad
	\gamma(\eta)=-\,i\,\frac{k}{\sqrt{k^{2}-1}}\,e^{ik\eta}.
	\]
	Taking real and imaginary parts, we obtain two linearly independent real solutions:
	\[
	\beta_{k}^{(1)}(\eta)=\cos(k\eta),\qquad\gamma_{k}^{(1)}(\eta)=\frac{k}{\sqrt{k^{2}-1}}\sin(k\eta),
	\]
	\[
	\beta_{k}^{(2)}(\eta)=\sin(k\eta),\qquad \gamma_{k}^{(2)}(\eta)=-\frac{k}{\sqrt{k^{2}-1}}\cos(k\eta).
	\]
	By construction, for $\lambda=\lambda_{k}$ the pairs $(\beta_{k}^{(j)},\gamma_{k}^{(j)})$,
	$j=1,2$, solve \eqref{Eq3.13} and are $2\pi$–periodic. Therefore the vector fields
	\[
	\mathbf{v}_{k}^{(j)}(\eta)
	= \beta_{k}^{(j)}(\eta)\,E_{2}(\eta) + \gamma_{k}^{(j)}(\eta)\,E_{3},\qquad j=1,2,
	\]
	belong to $X$ and satisfy $L_{\lambda_{k}}\mathbf{v}_{k}^{(j)}=0$. They are linearly independent
	over $\mathbb{R}$, so
	\[
	\ker(L_{\lambda_{k}}\vert_{X})=\text{span}\{\mathbf{x}'_{0}, \mathbf{v}_{k}^{(1)}, \mathbf{v}_{k}^{(2)}\}.
	\]
	This proves item (iii) and completes the proof of the
	proposition.
\end{proof}

In order to reduce the dimension of the kernels of the linearization, we introduce an adequate symmetry on the involved functional spaces. 	Let
\[
S:=\mathrm{diag}(1,-1,1)\in\mathbb R^{3\times 3},\qquad(\mathcal R\mathbf u)(\eta):=S\,\mathbf u(-\eta),\qquad \eta\in\mathbb T,
\]
and define the closed subspace (fixed-point space)
\[
Y_{\mathrm{rev}}:=\mathrm{Fix}(\mathcal R)=\{\mathbf u\in H^{2}(\mathbb T,\mathbb R^{3}) : \mathcal R\mathbf u=\mathbf u\}.
\]
On the codomain define the linear involutions
\[
(\mathcal R_{F}\mathbf f)(\eta):=-S\,\mathbf f(-\eta)\quad\text{for }\mathbf f\in L^{2}(\mathbb T,\mathbb R^{3}),
\qquad (\mathcal R_{C}g)(\eta):=g(-\eta)\quad\text{for }g\in H^{2}(\mathbb T),
\]
and set
\[
Z_{\mathrm{rev}}:=\mathrm{Fix}(\mathcal R_{F})\times \mathrm{Fix}(\mathcal R_{C})
\subset L^{2}(\mathbb T,\mathbb R^{3})\times H^{2}(\mathbb T).
\]

\begin{lemma}
	\label{Le3.3}
	The following statements hold:
	\begin{enumerate}
		\item[{\rm(i)}] For every $\lambda>0$ and every $\mathbf u\in H^{2}(\mathbb T,\mathbb R^{3})$ one has the symmetry relations
		\[
		\mathcal F(\lambda,\mathcal R\mathbf u)=\mathcal R_{F}\,\mathcal F(\lambda,\mathbf u),
		\qquad\mathcal C(\mathcal R\mathbf u)=\mathcal R_{C}\,\mathcal C(\mathbf u).
		\]
		In particular, $\mathcal G\big((0,\infty)\times Y_{\mathrm{rev}}\big)\subset Z_{\mathrm{rev}}$,
		so the restricted map
		\[
		\mathcal G_{\mathrm{rev}}:=\mathcal G\big|_{(0,\infty)\times Y_{\mathrm{rev}}}:
		(0,\infty)\times Y_{\mathrm{rev}}\longrightarrow Z_{\mathrm{rev}}
		\]
		is well-defined.
		
		\item[{\rm(ii)}] For every $\lambda>0$,
		\begin{align*}
		\ker D_{\mathbf u}\mathcal G_{\mathrm{rev}}(\lambda,0)
		 =\ker\Big(D_{\mathbf u}\mathcal G(\lambda,0)\big|_{Y_{\mathrm{rev}}}\Big)  =
		\Big\{\mathbf v\in Y_{\mathrm{rev}}:\ L_{\lambda}\mathbf v=0,\ \mathbf x_{0}\cdot \mathbf v=0\Big\}.
		\end{align*}
		Moreover, $\ker D_{\mathbf u}\mathcal G_{\mathrm{rev}}(\lambda,0)=\{0\}$ if $\lambda\neq \lambda_{k}:=R\sqrt{k^{2}-1}$ for all $k\ge2$,
		whereas for $\lambda=\lambda_{k}$ with $k\ge2$ one has
		\[
		\ker D_{\mathbf u}\mathcal G_{\mathrm{rev}}(\lambda_{k},0)=\mathrm{span}\{\mathbf v_{k}^{(2)}\}.
		\]
	\end{enumerate}
\end{lemma}

\begin{proof}
	Recall that $\mathbf x_{0}(\eta)=(R\cos\eta,R\sin\eta,0)$. A direct check gives
	\begin{equation}
		\label{Eq3.15}
		\mathbf x_{0}(-\eta)=S\,\mathbf x_{0}(\eta),\qquad
		\mathbf x_{0}'(-\eta)=-S\,\mathbf x_{0}'(\eta),\qquad
		\mathbf x_{0}''(-\eta)=S\,\mathbf x_{0}''(\eta).
	\end{equation}
	Moreover \(Se_{3}=e_{3}\) and \(\det S=-1\). Hence, for every
	\(a,b\in\mathbb R^{3}\),
	\begin{equation}
		\label{Eq3.16}
		(Sa)\times(Sb)=\det(S)\,S(a\times b)=-S(a\times b),
	\end{equation}
	and, in particular, $e_{3}\times(Sa)=-S(e_{3}\times a)$.
	Let \(\mathbf u\in H^{2}(\mathbb T,\mathbb R^{3})\), and set
	\[
	\widetilde{\mathbf u}:=\mathcal R\mathbf u,\qquad	\widetilde{\mathbf u}(\eta)=S\mathbf u(-\eta).
	\]
	Then
	\[
	\widetilde{\mathbf u}'(\eta)=-S\mathbf u'(-\eta),
	\qquad\widetilde{\mathbf u}''(\eta)=S\mathbf u''(-\eta).
	\]
	Using \eqref{Eq3.15} and \eqref{Eq3.16}, term by term in the definition of \(\mathcal F\), we obtain
	\[
	\begin{aligned}
		\mathcal F(\lambda,\widetilde{\mathbf u})(\eta)
		&=
		\mathbf x_0(\eta)\times\widetilde{\mathbf u}''(\eta)+\widetilde{\mathbf u}(\eta)\times\mathbf x_0''(\eta)+\widetilde{\mathbf u}(\eta)\times\widetilde{\mathbf u}''(\eta)-\lambda e_3\times\widetilde{\mathbf u}(\eta)+\lambda\widetilde{\mathbf u}'(\eta)\\
		&=
		\mathbf x_0(\eta)\times S\mathbf u''(-\eta)+S\mathbf u(-\eta)\times\mathbf x_0''(\eta)+S\mathbf u(-\eta)\times S\mathbf u''(-\eta)\\
		&\quad
		-\lambda e_3\times S\mathbf u(-\eta)-\lambda S\mathbf u'(-\eta).
	\end{aligned}
	\]
	Since \(\mathbf x_0(\eta)=S\mathbf x_0(-\eta)\) and
	\(\mathbf x_0''(\eta)=S\mathbf x_0''(-\eta)\), this becomes
	\[
	\begin{aligned}
		\mathcal F(\lambda,\widetilde{\mathbf u})(\eta)
		&=
		(S\mathbf x_0(-\eta))\times(S\mathbf u''(-\eta))
		+(S\mathbf u(-\eta))\times(S\mathbf x_0''(-\eta))\\
		&\quad
		+(S\mathbf u(-\eta))\times(S\mathbf u''(-\eta))-\lambda e_3\times S\mathbf u(-\eta)-\lambda S\mathbf u'(-\eta)\\
		&=
		-S\big(\mathbf x_0(-\eta)\times\mathbf u''(-\eta)\big)-S\big(\mathbf u(-\eta)\times\mathbf x_0''(-\eta)\big)\\
		&\quad
		-S\big(\mathbf u(-\eta)\times\mathbf u''(-\eta)\big)+\lambda S\big(e_3\times\mathbf u(-\eta)\big)-\lambda S\mathbf u'(-\eta)\\
		&=
		-S\Big[
		\mathbf x_0(-\eta)\times\mathbf u''(-\eta)
		+\mathbf u(-\eta)\times\mathbf x_0''(-\eta)+\mathbf u(-\eta)\times\mathbf u''(-\eta)\\
		&\qquad\qquad
		-\lambda e_3\times\mathbf u(-\eta)+\lambda\mathbf u'(-\eta)
		\Big]=-S\,\mathcal F(\lambda,\mathbf u)(-\eta).
	\end{aligned}
	\]
	Thus $\mathcal F(\lambda,\mathcal R\mathbf u)=\mathcal R_F\mathcal F(\lambda,\mathbf u)$. For the constraint, using again \eqref{Eq3.15} and the orthogonality of \(S\), we get
	\[
	\begin{aligned}
		\mathcal C(\widetilde{\mathbf u})(\eta)
		&=
		\mathbf x_0(\eta)\cdot S\mathbf u(-\eta)+\frac12\|S\mathbf u(-\eta)\|^2=(S\mathbf x_0(\eta))\cdot \mathbf u(-\eta)+\frac12\|\mathbf u(-\eta)\|^2\\
		&=
		\mathbf x_0(-\eta)\cdot\mathbf u(-\eta)+\frac12\|\mathbf u(-\eta)\|^2=\mathcal C(\mathbf u)(-\eta).
	\end{aligned}
	\]
	Hence $\mathcal C(\mathcal R\mathbf u)=\mathcal R_C\mathcal C(\mathbf u)$.
	Therefore, if \(\mathbf u\in Y_{\mathrm{rev}}=\operatorname{Fix}(\mathcal R)\), then $\mathcal F(\lambda,\mathbf u)\in\operatorname{Fix}(\mathcal R_F)$ and $\mathcal C(\mathbf u)\in\operatorname{Fix}(\mathcal R_C)$.
	Equivalently,
	\[
	\mathcal G\big((0,\infty)\times Y_{\mathrm{rev}}\big)\subset Z_{\mathrm{rev}},
	\]
	and the restricted map $\mathcal G_{\mathrm{rev}}:(0,\infty)\times Y_{\mathrm{rev}}\to Z_{\mathrm{rev}}$
	is well defined. This proves item {\rm(i)}.
	
	We now prove item {\rm(ii)}. Since $D_{\mathbf u}\mathcal G(\lambda,0)[\mathbf v]=\big(L_\lambda\mathbf v,\mathbf x_0\cdot\mathbf v\big)$, we have
	\[
	\ker D_{\mathbf u}\mathcal G_{\mathrm{rev}}(\lambda,0)=\Big\{\mathbf v\in Y_{\mathrm{rev}}:L_\lambda\mathbf v=0,\;\mathbf x_0\cdot\mathbf v=0\Big\}.
	\]
	The condition \(\mathbf x_0\cdot\mathbf v=0\) is precisely the condition
	\(\mathbf v\in X\). Therefore
	\[
	\ker D_{\mathbf u}\mathcal G_{\mathrm{rev}}(\lambda,0)=\ker(L_\lambda|_X)\cap Y_{\mathrm{rev}}=\ker\big(L_\lambda|_{X\cap Y_{\mathrm{rev}}}\big).
	\]
	By Proposition~\ref{Pr3.2}, if $\lambda\neq\lambda_k:=R\sqrt{k^2-1}$ for every $k\ge2$,
	then $\ker(L_\lambda|_X)=\operatorname{span}\{\mathbf x_0'\}$.
	Since \(\mathbf x_0'=R E_2\), and $E_2(\eta)=(-\sin\eta,\cos\eta,0)$,
	we have
	\[
	(\mathcal R E_2)(\eta)=S E_2(-\eta)=S(\sin\eta,\cos\eta,0)=(\sin\eta,-\cos\eta,0)=-E_2(\eta).
	\]
	Thus $\mathcal R\mathbf x_0'=-\mathbf x_0'$,
	and hence \(\mathbf x_0'\notin Y_{\mathrm{rev}}\). Consequently,
	\[
	\ker D_{\mathbf u}\mathcal G_{\mathrm{rev}}(\lambda,0)=\{0\}
	\]
	whenever \(\lambda\neq\lambda_k\) for all \(k\ge2\). Now let \(\lambda=\lambda_k\) for some \(k\ge2\). Again by
	Proposition~\ref{Pr3.2},
	\[
	\ker(L_{\lambda_k}|_X)=\operatorname{span}\{\mathbf x_0',\mathbf v_k^{(1)},\mathbf v_k^{(2)}\}.
	\]
	We have already seen that \(\mathbf x_0'\notin Y_{\mathrm{rev}}\). It remains to determine
	which of \(\mathbf v_k^{(1)}\) and \(\mathbf v_k^{(2)}\) belongs to \(Y_{\mathrm{rev}}\).
	Recall that
	\[
	\mathbf v_k^{(1)}=\cos(k\eta)E_2(\eta)+\frac{k}{\sqrt{k^2-1}}\sin(k\eta)E_3,
	\]
	and
	\[
	\mathbf v_k^{(2)}=\sin(k\eta)E_2(\eta)-\frac{k}{\sqrt{k^2-1}}\cos(k\eta)E_3.
	\]
	Since $\mathcal R E_2=-E_2$, $\mathcal R E_3=E_3$,
	and since \(\cos(k\eta)\) is even whereas \(\sin(k\eta)\) is odd, we obtain
	\[
	\mathcal R\mathbf v_k^{(1)}=-\mathbf v_k^{(1)},\qquad \mathcal R\mathbf v_k^{(2)}=\mathbf v_k^{(2)}.
	\]
	Therefore $\mathbf v_k^{(1)}\notin Y_{\mathrm{rev}}$, $\mathbf v_k^{(2)}\in Y_{\mathrm{rev}}$.
	It follows that
	\[
	\ker D_{\mathbf u}\mathcal G_{\mathrm{rev}}(\lambda_k,0)=\operatorname{span}\{\mathbf v_k^{(2)}\}.
	\]
	This proves item {\rm(ii)} and completes the proof.
\end{proof}

\subsection{The tangential augmented operator}

The augmented operator \(\mathcal G=(\mathcal F,\mathcal C)\) encodes the full
differential equation together with the spherical constraint. However, once the
constraint is imposed, the differential equation has only two independent
components, since the vector field
\[
\mathcal E_\lambda(\mathbf x):=\mathbf x\times \mathbf x''-\lambda e_3\times \mathbf x+\lambda \mathbf x'
\]
is tangent to the sphere whenever \(\|\mathbf x\|\equiv R\). Thus, for the
Fredholm analysis, it is convenient to remove the redundant normal component of
\(\mathcal F\) and keep only its components along \(E_2\) and \(E_3\).

We define the tangential augmented operator
\[
\widehat{\mathcal G}:(0,+\infty)\times H^{2}(\mathbb T,\mathbb R^3)\longrightarrow L^2(\mathbb T)\times L^2(\mathbb T)\times H^2(\mathbb T)
\]
\[
\widehat{\mathcal G}(\lambda,\mathbf u):=\left(\mathcal F(\lambda,\mathbf u)\cdot E_2,\,\mathcal F(\lambda,\mathbf u)\cdot E_3,\,\mathcal C(\mathbf u)\right).
\]
Equivalently, \(\widehat{\mathcal G}\) is obtained from the full augmented
operator \(\mathcal G\) by projecting the differential component onto the two
directions \(E_2,E_3\). In the reversible setting, we set
\[
\widehat Z_{\mathrm{rev}}:= L^2_{\mathrm{even}}(\mathbb T) \times L^2_{\mathrm{odd}}(\mathbb T) \times H^2_{\mathrm{even}}(\mathbb T).
\]
Then the symmetry relations in Lemma~\ref{Le3.3}
imply that $\widehat{\mathcal G}\big((0,+\infty)\times Y_{\mathrm{rev}}\big)\subset\widehat Z_{\mathrm{rev}}$. We therefore obtain a well-defined restricted operator
\[
\widehat{\mathcal G}_{\mathrm{rev}}:(0,+\infty)\times Y_{\mathrm{rev}}\longrightarrow\widehat Z_{\mathrm{rev}}.
\]
Since \(\widehat{\mathcal G}\) is obtained from \(\mathcal G\) by composition with
a bounded linear projection, the smoothness of \(\mathcal G\) immediately gives $\widehat{\mathcal G}\in \mathcal C^\infty$.
Moreover,
\[
D_{\mathbf u}\widehat{\mathcal G}(\lambda,0)[\mathbf v]=\left((L_\lambda\mathbf v)\cdot E_2,\,(L_\lambda\mathbf v)\cdot E_3,\,\mathbf x_0\cdot \mathbf v\right).
\]

\begin{lemma}[Local equivalence of the zero sets]
	\label{Le3.4}
	There exists an \(H^2\)-neighbourhood \(\mathcal U\) of \(0\) in
	\(H^2(\mathbb T,\mathbb R^3)\) such that, for every $(\lambda,\mathbf u)\in (0,+\infty)\times \mathcal U$,
	one has
	\[
	\widehat{\mathcal G}(\lambda,\mathbf u)=0 \quad\Longleftrightarrow\quad \mathcal G(\lambda,\mathbf u)=0.
	\]
\end{lemma}

\begin{proof}
	The implication $\mathcal G(\lambda,\mathbf u)=0\Rightarrow\widehat{\mathcal G}(\lambda,\mathbf u)=0$
	is immediate, since \(\widehat{\mathcal G}\) is obtained from \(\mathcal G\) by
	keeping the \(E_2\)- and \(E_3\)-components of the differential part and the
	same constraint component. Conversely, assume that $\widehat{\mathcal G}(\lambda,\mathbf u)=0$ and set $\mathbf x:=\mathbf x_0+\mathbf u$. Then the constraint component gives $\mathcal C(\mathbf u)=0$, and therefore $\|\mathbf x(\eta)\|^2=R^2$ for every $\eta\in\mathbb T$. In particular, $\mathbf x\cdot\mathbf x'=0$.
	Let
	\[
	\mathcal E_\lambda(\mathbf x):=\mathbf x\times\mathbf x''-\lambda e_3\times\mathbf x+\lambda\mathbf x'.
	\]
	Since \(\mathbf x\) lies on the sphere, we have $\mathcal E_\lambda(\mathbf x)\cdot\mathbf x=0$.
	On the other hand, the first two components of
	\(\widehat{\mathcal G}(\lambda,\mathbf u)=0\) give $\mathcal E_\lambda(\mathbf x)\cdot E_2=0$ and $\mathcal E_\lambda(\mathbf x)\cdot E_3=0$.
	It remains to choose the neighbourhood \(\mathcal U\). Since $\mathbf x_0(\eta)=R E_1(\eta)$,
	we have $\mathbf x_0(\eta)\cdot E_1(\eta)=R$.
	By the Sobolev embedding \(H^2(\mathbb T)\hookrightarrow \mc{C}^0(\mathbb T)\), we may choose
	an \(H^2\)-neighbourhood \(\mathcal U\) of \(0\) such that
	\[
	\|\mathbf u\|_{\mc{C}^{0}}<\frac R2\quad\text{for every }\mathbf u\in\mathcal U.
	\]
	Then, for every \(\eta\in\mathbb T\),
	\[
	\mathbf x(\eta)\cdot E_1(\eta)=R+\mathbf u(\eta)\cdot E_1(\eta)\ge R-\|\mathbf u\|_{C^0}>\frac R2.
	\]
	Thus the three vectors $\mathbf x(\eta)$, $E_2(\eta)$, $E_3$ form a basis of \(\mathbb R^3\) for every \(\eta\in\mathbb T\). Since
	\(\mathcal E_\lambda(\mathbf x)(\eta)\) is orthogonal to each of these three
	vectors, it follows that $\mathcal E_\lambda(\mathbf x)=0$.
	Therefore the differential part of \(\mathcal G(\lambda,\mathbf u)\) vanishes.
	Since we already know that \(\mathcal C(\mathbf u)=0\), we conclude that $\mathcal G(\lambda,\mathbf u)=0$.
	This proves the reverse implication and hence the local equivalence of the two
	zero sets.
\end{proof}

\begin{lemma}
	\label{Le3.5}
	For every \(\lambda>0\), $\ker D_{\mathbf u}\widehat{\mathcal G}_{\mathrm{rev}}(\lambda,0)
	=\ker D_{\mathbf u}\mathcal G_{\mathrm{rev}}(\lambda,0)$. In particular, if \(\lambda\neq\lambda_k:=R\sqrt{k^2-1}\) for all \(k\ge2\), $\ker D_{\mathbf u}\widehat{\mathcal G}_{\mathrm{rev}}(\lambda,0)=\{0\}$,
	whereas
	\[
	\ker D_{\mathbf u}\widehat{\mathcal G}_{\mathrm{rev}}(\lambda_k,0)=\mathrm{span}\{\mathbf v_k^{(2)}\}.
	\]
\end{lemma}

\begin{proof}
	Let \(\mathbf v\in Y_{\mathrm{rev}}\). Write $\mathbf v=\alpha E_1+\beta E_2+\gamma E_3$.
	Then $\mathbf x_0\cdot\mathbf v=R\alpha$. If $D_{\mathbf u}\mathcal G_{\mathrm{rev}}(\lambda,0)[\mathbf v]=0$, then clearly
	\[
	D_{\mathbf u}\widehat{\mathcal G}_{\mathrm{rev}}(\lambda,0)[\mathbf v]=0.
	\]
	Conversely, assume that $D_{\mathbf u}\widehat{\mathcal G}_{\mathrm{rev}}(\lambda,0)[\mathbf v]=0$.
	Then $(L_\lambda\mathbf v)\cdot E_2=0$, $(L_\lambda\mathbf v)\cdot E_3=0$ and $\mathbf x_0\cdot\mathbf v=0$.
	The last condition gives \(\alpha=0\). From the computation of
	\(L_\lambda\) in the moving frame,
	\[
	L_\lambda\mathbf v=\lambda\alpha' E_1+\big[-R(\gamma''+\gamma)+\lambda\beta'\big]E_2+\big[R(2\alpha'+\beta'')+\lambda\gamma'\big]E_3.
	\]
	Since \(\alpha=0\), the \(E_1\)-component also vanishes. Therefore $L_\lambda\mathbf v=0$, $\mathbf x_0\cdot\mathbf v=0$,
	and hence
	\[
	D_{\mathbf u}\mathcal G_{\mathrm{rev}}(\lambda,0)[\mathbf v]=0.
	\]
	The description of the kernel follows from Lemma~\ref{Le3.3}.
\end{proof}

\begin{lemma}[Fredholm property]
	\label{Le3.6}
	For every \(\lambda>0\), the operator
	\[D_{\mathbf u}\widehat{\mathcal G}_{\mathrm{rev}}(\lambda,0):Y_{\mathrm{rev}}\longrightarrow\widehat Z_{\mathrm{rev}}\]
	is Fredholm of index zero.
\end{lemma}

\begin{proof}
	Let $\mathbf v=\alpha E_1+\beta E_2+\gamma E_3\in Y_{\mathrm{rev}}$.
	The reversibility condition is equivalent to
	\[
	\alpha(-\eta)=\alpha(\eta), \qquad \beta(-\eta)=-\beta(\eta), \qquad \gamma(-\eta)=\gamma(\eta).
	\]
	Thus $\alpha\in H^2_{\mathrm{even}}(\mathbb T)$, $\beta\in H^2_{\mathrm{odd}}(\mathbb T)$ and $\gamma\in H^2_{\mathrm{even}}(\mathbb T)$. Using the expression for \(L_\lambda\mathbf v\) in the moving frame, we obtain
	\[
	D_{\mathbf u}\widehat{\mathcal G}_{\mathrm{rev}}(\lambda,0)[\mathbf v]=\left(\lambda\beta'-R(\gamma''+\gamma), \; R(2\alpha'+\beta'')+\lambda\gamma', \; R\alpha \right).
	\]
	Consider the bounded isomorphism of the codomain
	\[
	\mathcal S: L^2_{\mathrm{even}}(\mathbb{T})\times L^2_{\mathrm{odd}}(\mathbb{T})\times H^2_{\mathrm{even}}(\mathbb{T}) \longrightarrow L^2_{\mathrm{even}}(\mathbb{T})\times L^2_{\mathrm{odd}}(\mathbb{T})\times H^2_{\mathrm{even}}(\mathbb{T})
	\]
	defined by
	\[
	\mathcal S(h_2,h_3,g):=(h_2,h_3-2g',g).
	\]
	Its inverse is
	\[
	\mathcal S^{-1}(h_2,h_3,g)=(h_2,h_3+2g',g).
	\]
	Since \(g\in H^2_{\mathrm{even}}\), we have \(g'\in H^1_{\mathrm{odd}}\subset L^2_{\mathrm{odd}}\), so
	\(\mathcal S\) is well-defined and bounded.
	Composing with \(\mathcal S\), we get
	\[
	\mathcal S\,D_{\mathbf u}\widehat{\mathcal G}_{\mathrm{rev}}(\lambda,0)
	(\alpha,\beta,\gamma)=\left(\lambda\beta'-R(\gamma''+\gamma),\; R\beta''+\lambda\gamma',\; R\alpha\right).
	\]
	Therefore $D_{\mathbf u}\widehat{\mathcal G}_{\mathrm{rev}}(\lambda,0)$
	is equivalent, up to an isomorphism of the codomain, to the block diagonal operator
	\[
	\mathcal M_\lambda\oplus R\,\mathrm{Id},
	\]
	where
	\[
	\mathcal M_\lambda: H^2_{\mathrm{odd}}(\mathbb T)\times H^2_{\mathrm{even}}(\mathbb T)
	\to L^2_{\mathrm{even}}(\mathbb T)\times L^2_{\mathrm{odd}}(\mathbb T), \quad
	\mathcal M_\lambda(\beta,\gamma)=\left(\lambda\beta'-R(\gamma''+\gamma),\; R\beta''+\lambda\gamma'\right).
	\]
	It remains to prove that \(\mathcal M_\lambda\) is Fredholm of index zero. We write
	\[
	\beta(\eta)=\sum_{k\geq1} b_k\sin(k\eta), \qquad \gamma(\eta)=a_0+\sum_{k\geq1} a_k\cos(k\eta).
	\]
	Then
	\[
	\gamma''(\eta)+\gamma(\eta)=a_0+\sum_{k\geq1}(1-k^2)a_k\cos(k\eta).
	\]
	Thus, if $\mathcal M_\lambda(\beta,\gamma)=(h_1,h_2)$, with
	\[
	h_1(\eta)=p_0+\sum_{k\geq1}p_k\cos(k\eta), \qquad h_2(\eta)=\sum_{k\geq1}q_k\sin(k\eta),
	\]
	then $p_0=-Ra_0$, and, for every \(k\geq1\),
	\[
	\binom{p_k}{q_k}
	=M_k(\lambda)
	\binom{b_k}{a_k}, \qquad
	M_k(\lambda):=
	\begin{pmatrix}
		\lambda k & R(k^2-1)\\
		-Rk^2 & -\lambda k
	\end{pmatrix}.
	\]
	Then
	\[
	\mathcal M_\lambda(\beta,\gamma)(\eta)=\binom{-Ra_0}{0}+\sum_{k\ge1}
	\begin{pmatrix}
		\cos(k\eta) & 0\\
		0 & \sin(k\eta)
	\end{pmatrix}
	M_k(\lambda)
	\binom{b_k}{a_k}.
	\]
	Moreover, $\det M_k(\lambda)=k^2\bigl(R^2(k^2-1)-\lambda^2\bigr)$.
	Hence the only possible singular modes are those satisfying $\lambda=R\sqrt{k^2-1}$. For fixed \(\lambda>0\), the set
	\[
	\mathcal N_\lambda:=\{k\geq1:\det M_k(\lambda)=0\}
	\]
	is finite, in fact it has at most one element. For \(k\notin\mathcal N_\lambda\), the matrix \(M_k(\lambda)\) is invertible. 
	
	It follows that the kernel of \(\mathcal M_\lambda\) is finite-dimensional. Indeed, all
	non-singular modes vanish for an element of the kernel, and only the finitely many modes
	in \(\mathcal N_\lambda\) may contribute:
	\[
	\ker \mathcal M_\lambda=\bigoplus_{k\in\mathcal N_\lambda}\ker M_k(\lambda).
	\]
	Since each singular \(2\times2\) matrix has rank one, each \(\ker M_k(\lambda)\) is
	one-dimensional.
	
	We now prove that the range of \(\mathcal M_\lambda\) is closed and has finite codimension. For each singular mode \(k\in\mathcal N_\lambda\), the matrix \(M_k(\lambda)\) has rank one.
	Hence its range is a one-dimensional subspace of \(\mathbb R^2\). We choose a non-zero
	linear functional $\ell_k\in(\mathbb R^2)^*$
	such that
	\[
	\operatorname{Ran}M_k(\lambda)=\ker \ell_k.
	\]
	We define the corresponding continuous linear functional on
	\(L^2_{\mathrm{even}}(\mathbb T)\times L^2_{\mathrm{odd}}(\mathbb T)\) by
	\[
	\mathscr L_k(h_1,h_2):=\ell_k(p_k,q_k), \quad (h_1, h_2)\in L^2_{\mathrm{even}}(\mathbb T)\times L^2_{\mathrm{odd}}(\mathbb T),
	\]
	where \(p_k,q_k\) are the \(k\)-th Fourier coefficients of \(h_1,h_2\). This is continuous
	because the maps $h_1\mapsto p_k$, $h_2\mapsto q_k$ are continuous linear functionals on \(L^2(\mathbb T)\). We claim that
	\[
	\operatorname{Ran}\mathcal M_\lambda=\bigcap_{k\in\mathcal N_\lambda}\ker \mathscr L_k.
	\]
	First, let \((h_1,h_2)\in\operatorname{Ran}\mathcal M_\lambda\). Then there exists
	\((\beta,\gamma)\in H^2_{\mathrm{odd}}\times H^2_{\mathrm{even}}\) such that $\mathcal M_\lambda(\beta,\gamma)=(h_1,h_2)$. Taking Fourier coefficients in a singular mode \(k\in\mathcal N_\lambda\), we get
	\[
	\binom{p_k}{q_k}=M_k(\lambda)\binom{b_k}{a_k}.
	\]
	Thus $(p_k,q_k)\in\operatorname{Ran}M_k(\lambda)=\ker\ell_k$, and therefore $\mathscr L_k(h_1,h_2)=0$.
	Since this holds for every \(k\in\mathcal N_\lambda\), we have
	\[
	(h_1,h_2)\in\bigcap_{k\in\mathcal N_\lambda}\ker \mathscr L_k.
	\]
	Conversely, suppose that $(h_1,h_2)\in
	\bigcap_{k\in\mathcal N_\lambda}\ker \mathscr L_k$.
	We construct \((\beta,\gamma)\in H^2_{\mathrm{odd}}\times H^2_{\mathrm{even}}\) such that $\mathcal M_\lambda(\beta,\gamma)=(h_1,h_2)$.
	Set	$a_0=-\frac{p_0}{R}$.
	For every non-singular mode \(k\notin\mathcal N_\lambda\), define
	\[\binom{b_k}{a_k}=
	M_k(\lambda)^{-1}
	\binom{p_k}{q_k}.\]
	For each singular mode \(k\in\mathcal N_\lambda\), the condition $\mathscr L_k(h_1,h_2)=0$
	means precisely that $(p_k,q_k)\in\ker\ell_k=\operatorname{Ran}M_k(\lambda)$.
	Therefore we may choose a vector \((b_k,a_k)\in\mathbb R^2\) satisfying
	\[
	M_k(\lambda)\binom{b_k}{a_k}=\binom{p_k}{q_k}.
	\]
	There are only finitely many singular modes, so these choices contribute only finitely
	many Fourier coefficients. We now prove that the functions defined by these coefficients belong to \(H^2\). Since
	\[
	\det M_k(\lambda)=k^2\bigl(R^2(k^2-1)-\lambda^2\bigr)\sim R^2k^4\qquad (k\to\infty),
	\]
	there exist \(k_0\in\mathbb N\) and \(c>0\) such that $|\det M_k(\lambda)|\geq c k^4$ for all $k\geq k_0$.
	Furthermore,
	\[
	M_k(\lambda)^{-1}=\frac{1}{\det M_k(\lambda)}
	\begin{pmatrix}
		-\lambda k & -R(k^2-1)\\
		Rk^2 & \lambda k
	\end{pmatrix}.
	\]
	Since the entries of the adjugate matrix are \(O(k^2)\), we obtain
	\[
	\|M_k(\lambda)^{-1}\|\leq \frac{C}{k^2}
	\]
	for all sufficiently large \(k\notin\mathcal N_\lambda\). Enlarging \(C\) if necessary, the
	estimate is valid for all non-singular modes \(k\notin\mathcal N_\lambda\), apart from
	finitely many low modes, which are controlled by increasing the constant. Consequently, on the complement of the singular modes,
	\[
	|b_k|^2+|a_k|^2\leq\frac{C}{k^4}\bigl(|p_k|^2+|q_k|^2\bigr).
	\]
	Together with \(p_0=-Ra_0\), this gives
	\[
	|a_0|^2+\sum_{k\notin\mathcal N_\lambda} k^4\bigl(|b_k|^2+|a_k|^2\bigr)\leq C_\lambda\left(|p_0|^2+\sum_{k\geq1}\bigl(|p_k|^2+|q_k|^2\bigr)\right).
	\]
	The singular modes are finite in number, hence they do not affect membership in \(H^2\).
	Therefore
	\[
	\beta(\eta):=\sum_{k\geq1}b_k\sin(k\eta)\in H^2_{\mathrm{odd}}(\mathbb T),
	\]
	and
	\[
	\gamma(\eta):=a_0+\sum_{k\geq1}a_k\cos(k\eta)\in H^2_{\mathrm{even}}(\mathbb T).
	\]
	By construction, their Fourier coefficients satisfy $p_0=-Ra_0$
	and
	\[
	\binom{p_k}{q_k}=M_k(\lambda)\binom{b_k}{a_k}\qquad\text{for every }k\geq1.
	\]
	Thus $\mathcal M_\lambda(\beta,\gamma)=(h_1,h_2)$.
	Hence $(h_1,h_2)\in\operatorname{Ran}\mathcal M_\lambda$.
	This proves the reverse inclusion
	\[
	\bigcap_{k\in\mathcal N_\lambda}\ker \mathscr L_k
	\subset\operatorname{Ran}\mathcal M_\lambda.
	\]
	Consequently,
	\[
	\operatorname{Ran}\mathcal M_\lambda=\bigcap_{k\in\mathcal N_\lambda}\ker \mathscr L_k.
	\]
	Since \(\mathcal N_\lambda\) is finite and each \(\mathscr L_k\) is continuous, the range is
	closed. Moreover,
	\[
	\operatorname{codim}\operatorname{Ran}\mathcal M_\lambda=\#\mathcal N_\lambda,
	\]
	because each singular mode imposes exactly one independent scalar compatibility
	condition. Hence
	\[
	\dim\ker\mathcal M_\lambda=\operatorname{codim}\operatorname{Ran}\mathcal M_\lambda=\#\mathcal N_\lambda.
	\]
	Consequently, $\operatorname{ind}\mathcal M_\lambda=0$. Thus \(\mathcal M_\lambda\) is Fredholm of index zero.
	Since \(R\,\mathrm{Id}\) on	\(H^2_{\mathrm{even}}\) is an isomorphism onto \(H^2_{\mathrm{even}}\), it follows that $D_{\mathbf u}\widehat{\mathcal G}_{\mathrm{rev}}(\lambda,0)$ is Fredholm of index zero.
\end{proof}

\subsection{Application of the Crandall--Rabinowitz theorem}

We now complete the verification of the hypotheses of the Crandall--Rabinowitz
bifurcation theorem \cite{CR1971} (see also \cite{LG}) for the projected reversible operator $\widehat{\mathcal G}_{\rm rev}:
(0,+\infty)\times Y_{\rm{rev}}\to\widehat Z_{\rm rev}$. The Fredholm property and the description of the kernel have been obtained in the previous subsection. It remains to verify the transversality condition at
each critical value \(\lambda_k=R\sqrt{k^2-1}\), \(k\ge2\).

\begin{lemma}[Transversality condition]
	\label{Le3.7}
	For each \(k\ge 2\), one has
	\[
	D_{\lambda\mathbf u}\widehat{\mathcal G}_{\mathrm{rev}}(\lambda_k,0)
	\big[\mathbf v_k^{(2)}\big]\notin\operatorname{Ran}\big(D_{\mathbf u}\widehat{\mathcal G}_{\mathrm{rev}}(\lambda_k,0)\big),
	\]
	where $\lambda_k=R\sqrt{k^2-1}$.
\end{lemma}

\begin{proof}
Since $D_\lambda \mathcal F(\lambda,\mathbf u)[\mu]
	=
	\mu\big(-e_3\times\mathbf u+\mathbf u'\big)$,
	we obtain, differentiating with respect to \(\mathbf u\) at \(\mathbf u=0\),
	\[
	D_{\lambda\mathbf u}\mathcal F(\lambda,0)[\mathbf v]=-e_3\times\mathbf v+\mathbf v'.
	\]
	Moreover, \(\mathcal C\) does not depend on \(\lambda\). Hence
	\[
	D_{\lambda\mathbf u}\widehat{\mathcal G}_{\mathrm{rev}}(\lambda,0)[\mathbf v]
	=\left((-e_3\times\mathbf v+\mathbf v')\cdot E_2,\,(-e_3\times\mathbf v+\mathbf v')\cdot E_3,\,0\right).
	\]
	Fix \(k\ge2\), and set $T_k:=D_{\mathbf u}\widehat{\mathcal G}_{\mathrm{rev}}(\lambda_k,0)$.
	We shall construct a bounded linear functional on \(\widehat Z_{\mathrm{rev}}\) which
	annihilates \(\operatorname{Ran}(T_k)\) but not $D_{\lambda\mathbf u}\widehat{\mathcal G}_{\mathrm{rev}}(\lambda_k,0)\big[\mathbf v_k^{(2)}\big]$. Write an element of \(Y_{\mathrm{rev}}\) as $\mathbf v=\alpha E_1+\beta E_2+\gamma E_3$, with $\alpha\in H^2_{\mathrm{even}}(\mathbb T)$, $\beta\in H^2_{\mathrm{odd}}(\mathbb T)$, $\gamma\in H^2_{\mathrm{even}}(\mathbb T)$.
	From the proof of Lemma~\ref{Le3.6}, after the codomain isomorphism
	\[
	\mathcal S(h_2,h_3,g)=(h_2,h_3-2g',g),
	\]
	the linear operator \(T_k\) is equivalent to
	\[
	(\alpha,\beta,\gamma)\longmapsto\left(\mathcal M_{\lambda_k}(\beta,\gamma),\,R\alpha\right),
	\]
	where
	\[
	\mathcal M_{\lambda_k}(\beta,\gamma)=\left(\lambda_k\beta'-R(\gamma''+\gamma),\; R\beta''+\lambda_k\gamma'\right).
	\]
	Since the third component is onto, the cokernel is entirely produced by the
	singular \(k\)-th Fourier mode of \(\mathcal M_{\lambda_k}\). In the \(k\)-th mode, write $\beta(\eta)=b\sin(k\eta)$, $\gamma(\eta)=a\cos(k\eta)$.
	Then
	\[
	\mathcal M_{\lambda_k}(\beta,\gamma)=\left(\big(\lambda_k k b+R(k^2-1)a\big)\cos(k\eta),\;\big(-Rk^2 b-\lambda_k k a\big)\sin(k\eta)\right).
	\]
	Thus the relevant real \(2\times2\) matrix is
	\[
	M_k(\lambda_k)=
	\begin{pmatrix}
		\lambda_k k & R(k^2-1)\\
		-Rk^2 & -\lambda_k k
	\end{pmatrix}.
	\]
	At \(\lambda=\lambda_k=R\sqrt{k^2-1}\), this matrix has rank one. A generator
	of the left kernel is $(
	1,\frac{\sqrt{k^2-1}}{k})$.
	Consequently, the functional $\ell:\widehat Z_{\mathrm{rev}}\to\mathbb R$
	defined by
	\[
	\ell(h_2,h_3,g):=\int_{\mathbb T}
	\left(h_2(\eta)\cos(k\eta)
	+\frac{\sqrt{k^2-1}}{k}\,
	h_3(\eta)\sin(k\eta)\right)\,d\eta
	\]
	annihilates the range of $\mathcal S T_k$.
	Equivalently, $\ell\big(\mathcal S T_k\mathbf v\big)=0$ for every $\mathbf v\in Y_{\mathrm{rev}}$.
	Now we evaluate the mixed derivative on the kernel generator $\mathbf v_k^{(2)}(\eta)=\beta(\eta)E_2(\eta)+\gamma(\eta)E_3$,
	where
	\[
	\beta(\eta)=\sin(k\eta),
	\qquad\gamma(\eta)=-\frac{k}{\sqrt{k^2-1}}\cos(k\eta).
	\]
	Using $E_2'=-E_1$, $E_3'=0$, $e_3\times E_2=-E_1$, we get $(\mathbf v_k^{(2)})'=\beta' E_2-\beta E_1+\gamma' E_3$
	and $-e_3\times\mathbf v_k^{(2)}=\beta E_1$.
	Therefore the \(E_1\)-terms cancel:
	\[
	-e_3\times\mathbf v_k^{(2)}
	+(\mathbf v_k^{(2)})'=\beta' E_2+\gamma' E_3.
	\]
	Since
	\[
	\beta'(\eta)=k\cos(k\eta),
	\qquad \gamma'(\eta)=\frac{k^2}{\sqrt{k^2-1}}\sin(k\eta),
	\]
	we obtain
	\[
	D_{\lambda\mathbf u}\widehat{\mathcal G}_{\mathrm{rev}}(\lambda_k,0)
	\big[\mathbf v_k^{(2)}\big]=\left(k\cos(k\eta),\frac{k^2}{\sqrt{k^2-1}}\sin(k\eta),0\right).
	\]
	Because the third component is zero, applying \(\mathcal S\) does not change this
	vector. Hence
	\begin{align*}
		\ell\!\left(\mathcal S D_{\lambda\mathbf u}\widehat{\mathcal G}_{\mathrm{rev}}(\lambda_k,0)\big[\mathbf v_k^{(2)}\big]\right)
		&=
		\int_{\mathbb T}\left(k\cos^2(k\eta)+\frac{\sqrt{k^2-1}}{k}\frac{k^2}{\sqrt{k^2-1}}\sin^2(k\eta)\right)\,d\eta \\
		&=
		\int_{\mathbb T} k\,d\eta=2\pi k\neq0.
	\end{align*}
	Therefore
	\[
	\mathcal S
	D_{\lambda\mathbf u}\widehat{\mathcal G}_{\mathrm{rev}}(\lambda_k,0)
	\big[\mathbf v_k^{(2)}\big]\notin\operatorname{Ran}(\mathcal S T_k).
	\]
	Since \(\mathcal S\) is an isomorphism of the codomain, this is equivalent to
	\[
	D_{\lambda\mathbf u}\widehat{\mathcal G}_{\mathrm{rev}}(\lambda_k,0)\big[\mathbf v_k^{(2)}\big]\notin\operatorname{Ran}(T_k).
	\]
	The proof is complete.
\end{proof}
We now introduce the trivial branch with respect to which the bifurcation occurs. Namely, we denote the set of trivial solutions of \(\widehat{\mathcal G}_{\mathrm{rev}}\) by
\[
\mathcal T:=\{(\lambda,0):\lambda>0\}\subset (0,+\infty)\times Y_{\mathrm{rev}}.
\]
The main local result is the following.

\begin{theorem}[Local bifurcation in the reversible class]
	\label{Th3.1}
	Let \(k\ge2\). Then
	\[
	(\lambda_k,0)=(R\sqrt{k^2-1},0)\in (0,+\infty)\times Y_{\mathrm{rev}}
	\]
	is a bifurcation point of the projected reversible operator $\widehat{\mathcal G}_{\mathrm{rev}}:(0,+\infty)\times Y_{\mathrm{rev}}\to\widehat Z_{\mathrm{rev}}$ from the trivial branch \(\mathcal T\). Let
	\[
	Y^k_{\mathrm{rev}}:=\left\{\mathbf u\in Y_{\mathrm{rev}}:(\mathbf u,\mathbf v_k^{(2)})_{H^2}=0\right\}.
	\]
	Then the following statements hold.
	
	\begin{enumerate}
		\item[{\rm(a)}] {\rm\textbf{Existence.}}
		There exist \(\delta>0\), a neighbourhood of \((\lambda_k,0)\), and two
		\(\mathcal C^\infty\)-maps
		\[
		\Lambda_k:(-\delta,\delta)\to(0,+\infty), \qquad \Lambda_k(0)=\lambda_k,
		\]
		\[
		\Gamma_k:(-\delta,\delta)\to Y^k_{\mathrm{rev}}, \qquad \Gamma_k(0)=0,
		\]
		such that, for every \(t\in(-\delta,\delta)\),
		\[
		\widehat{\mathcal G}_{\mathrm{rev}}\big(\Lambda_k(t),\mathbf u_k(t)\big)=0,\qquad \mathbf u_k(t):=t\big(\mathbf v_k^{(2)}+\Gamma_k(t)\big).
		\]
		In particular, $\mathbf u_k'(0)=\mathbf v_k^{(2)}$.
		\vspace{5pt}
		\item[{\rm(b)}] {\rm\textbf{Local uniqueness.}}
		There exists \(\rho>0\) such that, if $\widehat{\mathcal G}_{\mathrm{rev}}(\lambda,\mathbf u)=0$
		and
		\[
		(\lambda,\mathbf u)\in B_\rho(\lambda_k,0)\subset (0,+\infty)\times Y_{\mathrm{rev}},
		\]
		then either \(\mathbf u=0\), or there exists
		\(t\in(-\delta,\delta)\) such that
		\[
		(\lambda,\mathbf u)=\big(\Lambda_k(t),\mathbf u_k(t)\big).
		\]
	\end{enumerate}
	Moreover, after possibly reducing \(\rho\) and \(\delta\), the above curve is also a
	curve of zeros of the full augmented operator \(\mathcal G_{\mathrm{rev}}\) (and therefore, $\mc{G}$). Equivalently,
	if
	\[
	\mathbf x_k(t):=\mathbf x_0+\mathbf u_k(t),
	\]
	then, for every \(t\in(-\delta,\delta)\),
	\[
	\mathbf x_k(t)\times \mathbf x_k(t)''-\Lambda_k(t)e_3\times \mathbf x_k(t)+\Lambda_k(t)\mathbf x_k(t)'=0, \qquad
	\|\mathbf x_k(t)(\eta)\|\equiv R.
	\]
\end{theorem}

\begin{proof}
	We apply the Crandall--Rabinowitz theorem to $\widehat{\mathcal G}_{\mathrm{rev}}:(0,+\infty)\times Y_{\mathrm{rev}}\to\widehat Z_{\mathrm{rev}}$.
	By Lemma~\ref{Le3.6}, $D_{\mathbf u}\widehat{\mathcal G}_{\mathrm{rev}}(\lambda_k,0)$
	is Fredholm of index zero. By Lemma~\ref{Le3.5},
	\[
	\ker D_{\mathbf u}\widehat{\mathcal G}_{\mathrm{rev}}(\lambda_k,0)=\mathrm{span}\{\mathbf v_k^{(2)}\}.
	\]
	Finally, Lemma~\ref{Le3.7} gives the transversality condition
	\[
	D_{\lambda\mathbf u}\widehat{\mathcal G}_{\mathrm{rev}}(\lambda_k,0)[\mathbf v_k^{(2)}]
	\notin \operatorname{Ran} D_{\mathbf u} \widehat{\mathcal G}_{\mathrm{rev}}(\lambda_k,0).
	\]
	The Crandall--Rabinowitz theorem therefore yields \(\delta>0\) and smooth maps
	\[
	\Lambda_k:(-\delta,\delta)\to(0,+\infty), \qquad \Gamma_k:(-\delta,\delta)\to Y^k_{\mathrm{rev}},
	\]
	with $\Lambda_k(0)=\lambda_k$, $\Gamma_k(0)=0$, such that $\mathbf u_k(t)=t(\mathbf v_k^{(2)}+\Gamma_k(t))$
	parametrizes all non-trivial zeros of
	\(\widehat{\mathcal G}_{\mathrm{rev}}\) in a neighbourhood of
	\((\lambda_k,0)\), up to the trivial branch.
	
	It remains only to relate these zeros to the original full equation. By
	Lemma~\ref{Le3.4}, after possibly reducing the neighbourhood,
	the zero sets of \(\widehat{\mathcal G}\) and \(\mathcal G\) coincide. Therefore the local
	branch obtained above is also a local branch of zeros of
	\(\mathcal G_{\mathrm{rev}}\), and hence of solutions of the full rescaled profile
	equation on the sphere. This proves the theorem.
\end{proof}

\subsection{Conclusion: unfixing the phase and returning to the original variable}

We have obtained, by applying the Crandall--Rabinowitz theorem to the projected
operator \(\widehat{\mathcal G}_{\mathrm{rev}}\), a local one--parameter branch in
the reversible class. By Lemma~\ref{Le3.4}, after possibly
shrinking the parameter interval, this branch consists also of zeros of the full
augmented operator \(\mathcal G_{\mathrm{rev}}\). We now use the \(SO(2)\)-symmetry
of the full equation to recover the phase parameter and obtain the corresponding
two--parameter local family of solutions.

\begin{lemma}[\(SO(2)\)-equivariance]
	\label{lem:SO2-equivariance-fixed-x0}
	Let
	\[
	R_\varphi:=
	\begin{pmatrix}
		\cos\varphi & -\sin\varphi & 0\\
		\sin\varphi & \cos\varphi & 0\\
		0 & 0 & 1
	\end{pmatrix},
	\qquad \varphi\in\mathbb R,
	\]
	and define $(\rho_\varphi \mathbf u)(\eta):=
	R_\varphi\,\mathbf u(\eta-\varphi)$, $\mathbf u\in H^2(\mathbb T,\mathbb R^3)$.
	Then, for every \(\lambda>0\), every \(\mathbf u\in H^2(\mathbb T,\mathbb R^3)\), and every
	\(\varphi\in\mathbb R\), one has
	\[
	\mathcal F(\lambda,\rho_\varphi\mathbf u)(\eta)=R_\varphi\,\mathcal F(\lambda,\mathbf u)(\eta-\varphi),
	\]
	and $\mathcal C(\rho_\varphi\mathbf u)(\eta)=\mathcal C(\mathbf u)(\eta-\varphi)$. In particular, if $\mathcal G(\lambda,\mathbf u)=(0,0)$, then $\mathcal G(\lambda,\rho_\varphi\mathbf u)=(0,0)$.
\end{lemma}

\begin{proof}
	Since \(R_\varphi\) is independent of \(\eta\), differentiation commutes with
	\(R_\varphi\). Hence
	\[
	(\rho_\varphi\mathbf u)'(\eta)=R_\varphi\mathbf u'(\eta-\varphi),\qquad(\rho_\varphi\mathbf u)''(\eta)=R_\varphi\mathbf u''(\eta-\varphi).
	\]
	Moreover \(R_\varphi\) preserves scalar products, norms, and cross products, and
	\(R_\varphi e_3=e_3\). Finally,
	\[
	R_\varphi\mathbf x_0(\eta-\varphi)=\mathbf x_0(\eta),\qquad R_\varphi\mathbf x_0''(\eta-\varphi)=\mathbf x_0''(\eta).
	\]
	Using these identities term by term in the definition of \(\mathcal F\), we obtain
	\[
	\mathcal F(\lambda,\rho_\varphi\mathbf u)(\eta)=R_\varphi\,\mathcal F(\lambda,\mathbf u)(\eta-\varphi).
	\]
	Similarly,
	\begin{align*}
		\mathcal C(\rho_\varphi\mathbf u)(\eta)
		&=
		\mathbf x_0(\eta)\cdot R_\varphi\mathbf u(\eta-\varphi)+\frac12\|R_\varphi\mathbf u(\eta-\varphi)\|^2\\
		&=
		R_\varphi\mathbf x_0(\eta-\varphi)\cdot R_\varphi\mathbf u(\eta-\varphi)+\frac12\|\mathbf u(\eta-\varphi)\|^2\\
		&=
		\mathbf x_0(\eta-\varphi)\cdot\mathbf u(\eta-\varphi)+\frac12\|\mathbf u(\eta-\varphi)\|^2=
		\mathcal C(\mathbf u)(\eta-\varphi).
	\end{align*}
	The final implication follows immediately.
\end{proof}

Fix \(k\ge2\), and let $\lambda_k:=R\sqrt{k^2-1}$.
By Theorem~\ref{Th3.1}, there exist \(\delta>0\), a smooth map $\Lambda_k:(-\delta,\delta)\to(0,+\infty)$, $\Lambda_k(0)=\lambda_k$,
and a smooth map $\Gamma_k:(-\delta,\delta)\to Y_{\mathrm{rev}}^k$, $\Gamma_k(0)=0$,
such that $\mathbf u_k(t)=t\big(\mathbf v_k^{(2)}+\Gamma_k(t)\big)$
satisfies $\widehat{\mathcal G}_{\mathrm{rev}}(\Lambda_k(t),\mathbf u_k(t))=0$. After reducing \(\delta\) if necessary, Lemma~\ref{Le3.4}
implies that $\mathcal G_{\mathrm{rev}}(\Lambda_k(t),\mathbf u_k(t))=0$ for every \(t\in(-\delta,\delta)\). Thus $\mathbf x_t(\eta):=\mathbf x_0(\eta)+\mathbf u_k(t)(\eta)$ is a \(2\pi\)-periodic solution of the full rescaled profile equation
\[
\mathbf x\times\mathbf x''-\Lambda_k(t)e_3\times\mathbf x+\Lambda_k(t)\mathbf x'=0,\qquad\|\mathbf x(\eta)\|\equiv R.
\]
Recall that
\[
\mathbf v_k^{(2)}(\eta)=\sin(k\eta)E_2(\eta)-\frac{k}{\sqrt{k^2-1}}\cos(k\eta)E_3.
\]
In Cartesian coordinates this is
\begin{equation*}
	\mathbf v_k^{(2)}(\eta)=
	\begin{pmatrix}
		-\sin(k\eta)\sin\eta\\[2pt]
		\phantom{-}\sin(k\eta)\cos\eta\\[2pt]
		-\dfrac{k}{\sqrt{k^2-1}}\cos(k\eta)
	\end{pmatrix}.
\end{equation*}
Since \(\Gamma_k(t)=o(1)\) in \(H^2(\mathbb T,\mathbb R^3)\), we obtain
\begin{equation*}
	\mathbf x_t(\eta)=
	\begin{pmatrix}
		R\cos\eta\\[2pt]
		R\sin\eta\\[2pt]
		0
	\end{pmatrix}
	+t
	\begin{pmatrix}
		-\sin(k\eta)\sin\eta\\[2pt]
		\phantom{-}\sin(k\eta)\cos\eta\\[2pt]
		-\dfrac{k}{\sqrt{k^2-1}}\cos(k\eta)
	\end{pmatrix}
	+o(t)
\end{equation*}
in \(H^2(\mathbb T,\mathbb R^3)\). We now restore the phase parameter. For \(\varphi\in\mathbb T\), define
\[
\mathbf u_k(t,\varphi):=\rho_\varphi\mathbf u_k(t),\qquad\mathbf x_{t,\varphi}(\eta):=\mathbf x_0(\eta)+\mathbf u_k(t,\varphi)(\eta).
\]
By Lemma~\ref{lem:SO2-equivariance-fixed-x0}, $\mathcal G(\Lambda_k(t),\mathbf u_k(t,\varphi))=(0,0)$. Therefore \(\mathbf x_{t,\varphi}\) is again a \(2\pi\)-periodic solution of the full rescaled
equation:
\[
\mathbf x_{t,\varphi}\times\mathbf x_{t,\varphi}''-\Lambda_k(t)e_3\times\mathbf x_{t,\varphi}+\Lambda_k(t)\mathbf x_{t,\varphi}'=0,\qquad\|\mathbf x_{t,\varphi}(\eta)\|\equiv R.
\]
Using $R_\varphi E_2(\eta-\varphi)=E_2(\eta)$ and $R_\varphi E_3=E_3$,
we get
\[
(\rho_\varphi\mathbf v_k^{(2)})(\eta)=\sin(k(\eta-\varphi))E_2(\eta)-\frac{k}{\sqrt{k^2-1}}\cos(k(\eta-\varphi))E_3.
\]
Consequently,
\begin{equation}
	\label{Eq3.17}
	\mathbf x_{t,\varphi}(\eta)=
	\begin{pmatrix}
		R\cos\eta\\[2pt]
		R\sin\eta\\[2pt]
		0
	\end{pmatrix}
	+t
	\begin{pmatrix}
		-\sin\!\big(k(\eta-\varphi)\big)\sin\eta\\[2pt]
		\phantom{-}\sin\!\big(k(\eta-\varphi)\big)\cos\eta\\[2pt]
		-\dfrac{k}{\sqrt{k^2-1}}\cos\!\big(k(\eta-\varphi)\big)
	\end{pmatrix}
	+o(t)
\end{equation}
in \(H^2(\mathbb T,\mathbb R^3)\), uniformly with respect to \(\varphi\in\mathbb T\). Here the
remainder is $o(t)=t\,\rho_\varphi\Gamma_k(t)$,
and the uniformity follows from the fact that \(\rho_\varphi\) acts isometrically on
\(H^2(\mathbb T,\mathbb R^3)\).

We finally return to the original travelling variable \(\xi\). Recall that
\[
\eta=\omega\xi, \qquad \omega=\frac{\Omega}{a}, \qquad \lambda=\frac{a^2}{\Omega}.
\]
Along the branch, \(\lambda=\Lambda_k(t)\), and therefore we set
\[
\Omega(t):=\frac{a^2}{\Lambda_k(t)}, \qquad \omega(t):=\frac{\Omega(t)}{a}=\frac{a}{\Lambda_k(t)}.
\]
Define
\[
\mathbf X_{t,\varphi}(\xi):=\mathbf x_{t,\varphi}(\omega(t)\xi).
\]
Then \(\mathbf X_{t,\varphi}\) solves the original profile equation
\[
\mathbf X\times \ddot{\mathbf X}-\Omega(t)e_3\times\mathbf X +a\dot{\mathbf X}=0, \qquad \|\mathbf X(\xi)\|\equiv R.
\]
Moreover, \(\mathbf X_{t,\varphi}\) is \(T_\xi(t)\)-periodic, where
\[
T_\xi(t)=\frac{2\pi}{\omega(t)}=\frac{2\pi\Lambda_k(t)}{a}.
\]
Substituting \(\eta=\omega(t)\xi\) in \eqref{Eq3.17}, we obtain
\begin{equation*}
	\mathbf X_{t,\varphi}(\xi)=
	\begin{pmatrix}
		R\cos(\omega(t)\xi)\\[2pt]
		R\sin(\omega(t)\xi)\\[2pt]
		0
	\end{pmatrix}
	+t
	\begin{pmatrix}
		-\sin\!\big(k(\omega(t)\xi-\varphi)\big)\sin(\omega(t)\xi)\\[2pt]
		\phantom{-}\sin\!\big(k(\omega(t)\xi-\varphi)\big)\cos(\omega(t)\xi)\\[2pt]
	    -\dfrac{k}{\sqrt{k^2-1}}\cos\!\big(k(\omega(t)\xi-\varphi)\big)
	\end{pmatrix}
	+o(t),
\end{equation*}
where the remainder is \(o(t)\) in \(H^2_{\mathrm{loc}}(\mathbb R,\mathbb R^3)\), uniformly in
\(\varphi\in\mathbb T\).

We summarize the conclusion in the following theorem.

\begin{theorem}[Local sheet of travelling--rotating profiles]
	Fix an integer \(k\ge2\), and set
	\[
	\lambda_k:=R\sqrt{k^2-1}, \qquad \Omega_k:=\frac{a^2}{\lambda_k}, \qquad \omega_k:=\frac{\Omega_k}{a}=\frac{a}{\lambda_k}.
	\]
	Then there exist \(\delta>0\), a smooth function
	\[
	\Lambda_k:(-\delta,\delta)\to(0,+\infty), \qquad \Lambda_k(0)=\lambda_k,
	\]
	and a smooth two--parameter family
	\[
	\Psi_k:(-\delta,\delta)\times\mathbb T\longrightarrow H^2_{\mathrm{loc}}(\mathbb R,\mathbb R^3),\qquad(t,\varphi)\longmapsto \mathbf X_{t,\varphi},
	\]
	with the following properties.
	
	\begin{enumerate}
		\item[{\rm(i)}]
		For every \((t,\varphi)\in(-\delta,\delta)\times\mathbb T\), the curve
		\(\mathbf X_{t,\varphi}\) is a \(T_\xi(t)\)-periodic solution of
		\[
		\mathbf X\times \ddot{\mathbf X}-\Omega(t)e_3\times\mathbf X+a\dot{\mathbf X}=0, \qquad \|\mathbf X(\xi)\|\equiv R,
		\]
		where
		\[
		\Omega(t):=\frac{a^2}{\Lambda_k(t)}, \qquad \omega(t):=\frac{\Omega(t)}{a}=\frac{a}{\Lambda_k(t)}, \qquad T_\xi(t):=\frac{2\pi}{\omega(t)}=\frac{2\pi\Lambda_k(t)}{a}.
		\]
		
		\item[{\rm(ii)}]
		At \(t=0\), the whole family reduces to the equatorial rotating profile
		\[
		\mathbf X_{0,\varphi}(\xi)=\big(R\cos(\omega_k\xi),\,R\sin(\omega_k\xi),\,0\big), \qquad \varphi\in\mathbb T.
		\]
		
		\item[{\rm(iii)}]
		As \(t\to0\), one has $\Lambda_k(t)=\lambda_k+o(1)$, and, uniformly with respect to \(\varphi\in\mathbb T\),
		\[
		\mathbf X_{t,\varphi}(\xi)=
		\begin{pmatrix}
			R\cos(\omega(t)\xi)\\[2pt] R\sin(\omega(t)\xi)\\[2pt] 0
		\end{pmatrix}
		+t
		\begin{pmatrix}
			-\sin\!\big(k(\omega(t)\xi-\varphi)\big)\sin(\omega(t)\xi)\\[2pt]\phantom{-}\sin\!\big(k(\omega(t)\xi-\varphi)\big)\cos(\omega(t)\xi)\\[2pt]-\dfrac{k}{\sqrt{k^2-1}}\cos\!\big(k(\omega(t)\xi-\varphi)\big)
		\end{pmatrix}
		+o(t)
		\]
		in \(H^2_{\mathrm{loc}}(\mathbb R,\mathbb R^3)\).
		\vspace{5pt}
		\item[{\rm(iv)}]
		The map $(t,\varphi)\mapsto
		(\Omega(t),\mathbf X_{t,\varphi})$
		defines a smooth two--parameter local family of nontrivial profiles, with the trivial equatorial phase fixed, issuing from
		\[
		\left(\Omega_k,\,\big(R\cos(\omega_k\xi),\,R\sin(\omega_k\xi),\,0\big)\right).
		\]
	\end{enumerate}
\end{theorem}

\begin{remark}
	The equatorial branch is not the only elementary family of \(2\pi\)-periodic
	solutions of the fixed-period profile equation. Indeed, for
	\(\nu\in\mathbb Z\setminus\{0\}\), consider a constant-latitude profile of the
	form
	\[
	\mathbf x(\eta)=\big(\sqrt{R^2-z^2}\cos(\nu\eta+\theta),\sqrt{R^2-z^2}\sin(\nu\eta+\theta),z\big),
	\]
	A direct substitution into $\mathbf x\times\mathbf x''-\lambda e_3\times\mathbf x+\lambda\mathbf x'=0$
	shows that such a profile is a solution if and only if
	\[
	\lambda(\nu-1)=\nu^2 z.
	\]
	The case \(\nu=1\) gives \(z=0\), and hence recovers the equatorial branch.
	For \(\nu\ge2\) one obtains non-equatorial parallels in the northern
	hemisphere, while for \(\nu\le -1\) one obtains parallels in the southern
	hemisphere. As the latitude tends to a pole, these families converge to the
	constant polar profiles. After reconstructing the filament by
	\(\mathbf X_0'=\mathbf x\), the corresponding filaments are circular helices
	around the vertical axis. This point of view is closely related to the perturbative analysis around helices developed in \cite{GarciaVega2024}. We do not pursue this direction here. Our aim in this section is to isolate the bifurcation mechanism from the equatorial branch, which is the branch used in the global continuation argument of the next section. This choice keeps the local analysis and the subsequent global description in a single coherent framework. A similar treatment starting from non-equatorial parallels would require a separate discussion and would lead to a different global problem.
\end{remark}

\section{Global continuation in the regular class}
\label{S4}

In this section we study the global continuation of the local branches obtained
in the previous section. Recall that we are dealing with the equation
\begin{equation}
	\label{Eq4.1}
	\mathbf{x}\times\mathbf{x}''-\lambda e_{3}\times\mathbf{x}+\lambda\mathbf{x}'=0, \qquad \|\mathbf{x}\|\equiv R, \qquad \mathbf{x}(\eta+2\pi)=\mathbf{x}(\eta).
\end{equation}
The augmented operator \(\mathcal G\) is written in terms of perturbations $\mathbf{x}=\mathbf{x}_{0}+\mathbf u$.
Thus, if $\mathcal G(\lambda,\mathbf u)=(0,0)$, then $\mathbf{x}:=\mathbf{x}_{0}+\mathbf u$
is a \(2\pi\)-periodic solution of \eqref{Eq4.1}.

We denote the zero set of the augmented problem by
\begin{equation*}
	\mathcal Z:=\left\{
	(\lambda,\mathbf u)\in(0,+\infty)\times H^{2}(\mathbb T,\mathbb R^{3}):\mathcal G(\lambda,\mathbf u)=(0,0)\right\}.
\end{equation*}
For $(\lambda,\mathbf u)\in\mathcal Z$ we always write $\mathbf{x}:=\mathbf{x}_{0}+\mathbf u$ with $u:=x_{3}$.

\subsection{The rescaled invariants}

We first recall how the invariants of Section~\ref{S2} look in the
rescaled equation \(a=\Omega=\lambda\). From \eqref{Eq2.3}, the conserved quantities
are
\begin{equation*}
	C=(\mathbf{x}\times\mathbf{x}')_{3}+\lambda x_{3}, \qquad E=\frac{\lambda}{2}\|\mathbf{x}'\|^{2}+\lambda(x_{2}x_{1}'-x_{1}x_{2}').
\end{equation*}
Since $x_{2}x_{1}'-x_{1}x_{2}'=-(\mathbf{x}\times\mathbf{x}')_{3}$,
and $C=(\mathbf{x}\times\mathbf{x}')_{3}+\lambda x_{3}$,
we get
\[
E=\frac{\lambda}{2}\|\mathbf{x}'\|^{2}-\lambda(C-\lambda x_{3}).
\]
Consequently the quantity
\begin{equation*}
	\mathcal H:=\frac{E}{\lambda}+C
\end{equation*}
satisfies
\begin{equation}
	\label{Eq4.2}
	\mathcal H=\frac12\|\mathbf{x}'\|^{2}+\lambda x_{3}.
\end{equation}
Thus \(\mathcal H\) is not a new first integral; it is only a convenient
combination of the first integrals \(C\) and \(E\). With the same specialization \(a=\Omega=\lambda\), the polynomial \(P\) in
\eqref{Eq2.7} is
\begin{equation}
	\label{Eq4.3}
	P(s)=a_{3}s^{3}-a_{2}s^{2}+a_{1}s+a_{0},
\end{equation}
where, by \eqref{Eq2.6},
\begin{align}
	\label{Eq4.4}
	a_{0}
	&=
	\frac{2ER^{2}-\lambda C^{2}+2\lambda C R^{2}}
	{\lambda R^{2}},
	&
	a_{1}
	&=\frac{2\lambda C-2\lambda R^{2}}{R^{2}},
	\nonumber\\
	a_{2}
	&=\frac{2E+\lambda^{3}+2\lambda C}{\lambda R^{2}},
	&
	a_{3}
	&=\frac{2\lambda}{R^{2}}.
\end{align}
The vertical component satisfies the scalar equation $(u')^{2}=P(u)$, $-R\le u\le R$.
Whenever \(P\) has three real roots $e_{1}<e_{2}<e_{3}$ and the solution oscillates between \(e_{1}\) and \(e_{2}\), we write
\begin{equation}
	\label{Eq4.5}
	P(s)=\frac{2\lambda}{R^{2}}(s-e_{1})(s-e_{2})(s-e_{3}).
\end{equation}
This is precisely the factorization used in Proposition~\ref{Pr2.4}, specialized
to the rescaled equation. Comparing the coefficient of \(s^{2}\) in \eqref{Eq4.3} and
\eqref{Eq4.5}, we obtain
\[
a_{2}=\frac{2\lambda}{R^{2}}(e_{1}+e_{2}+e_{3}).
\]
Using \eqref{Eq4.4}, this gives
\[
\frac{2E+\lambda^{3}+2\lambda C}{\lambda R^{2}}=\frac{2\lambda}{R^{2}}(e_{1}+e_{2}+e_{3}).
\]
Equivalently,
\begin{equation}
	\label{Eq4.6}
	\mathcal H=\frac{E}{\lambda}+C=\lambda(e_{1}+e_{2}+e_{3})-\frac{\lambda^{2}}{2}.
\end{equation}

We shall also use the following explicit expression for \(\mathbf x''\), obtained
from \eqref{Eq4.1} and the constraint \(\|\mathbf x\|=R\):
\begin{equation}
	\label{Eq4.7}
	\mathbf{x}''=-\frac{\|\mathbf{x}'\|^{2}}{R^{2}}\mathbf{x}-\lambda e_{3}+\frac{\lambda x_{3}}{R^{2}}\mathbf{x}+\frac{\lambda}{R^{2}}\mathbf{x}\times\mathbf{x}'.
\end{equation}
Indeed, since \(\mathbf{x}\cdot\mathbf{x}'=0\) and
\(\mathbf{x}\cdot\mathbf{x}''=-\|\mathbf{x}'\|^{2}\), taking the cross product
of
\[
\mathbf{x}\times\mathbf{x}''=\lambda e_{3}\times\mathbf{x}-\lambda\mathbf{x}'
\]
on the right by \(\mathbf x\), and using $(a\times b)\times c=(a\cdot c)b-(b\cdot c)a$,
gives
\[
R^2\mathbf{x}''+\|\mathbf{x}'\|^2\mathbf{x}=\lambda x_3\mathbf{x}-\lambda R^2e_3+\lambda\mathbf{x}\times\mathbf{x}'.
\]
Dividing by \(R^2\) gives \eqref{Eq4.7}.

\subsection{The regular non-polar class}

We now define the regular class in terms of the polynomial \(P\) introduced
above.

\begin{definition}[Regular non-polar oscillatory zeros]
	\label{Df4.1}
	We define \(\mathcal Z_{\rm reg}\subset\mathcal Z\) as the set of zeros $(\lambda,\mathbf u)\in\mathcal Z$ such that, for $\mathbf{x}=\mathbf{x}_{0}+\mathbf u$, $u=x_{3}$, the following conditions hold:
	\begin{enumerate}
		\item[\rm(i)] the solution does not touch the poles:
		\[
		-R<u(\eta)<R\qquad\text{for every }\eta\in\mathbb T;
		\]
		
		\item[\rm(ii)] the vertical component is not constant:
		\[
		\operatorname{osc}(u):=
		\max_{\eta\in\mathbb T}u(\eta)-\min_{\eta\in\mathbb T}u(\eta)>0;
		\]
		
		\item[\rm(iii)] if
		\[
		e_{1}:=\min_{\eta\in\mathbb T}u(\eta), \qquad e_{2}:=\max_{\eta\in\mathbb T}u(\eta),
		\]
		then the polynomial \(P\) associated with the solution by
		\eqref{Eq4.3}--\eqref{Eq4.4} has a third real root
		\(e_{3}\) such that
		\[
		-R<e_{1}<e_{2}<R, \qquad e_{2}<e_{3},
		\]
		and
		\[
		P(s)=\frac{2\lambda}{R^{2}}(s-e_{1})(s-e_{2})(s-e_{3}).
		\]
	\end{enumerate}
\end{definition}
Note that the equatorial branch $\{(\lambda,0):\lambda>0\}$ is not contained in \(\mathcal Z_{\rm reg}\), since for \(\mathbf u=0\) one has \(x_{3}\equiv0\). Thus the bifurcation points $(\lambda_{k},0)$, $\lambda_{k}=R\sqrt{k^{2}-1}$, belong to the boundary of the regular class.

\begin{lemma}
	The set \(\mathcal Z_{\rm reg}\) is relatively open in \(\mathcal Z\), with respect
	to the topology induced by
	\[
	(0,+\infty)\times H^{2}(\mathbb T,\mathbb R^{3}).
	\]
\end{lemma}

\begin{proof}
	We prove the statement sequentially. Let $(\lambda,\mathbf u)\in\mathcal Z_{\rm reg}$
	and let $\{(\lambda_{n},\mathbf u_{n})\}_{n\in\N}\subset\mathcal Z$
	be such that
	\[
	\lim_{n\to+\infty}(\lambda_{n},\mathbf u_{n})=(\lambda,\mathbf u) \qquad \text{in }(0,+\infty)\times H^{2}(\mathbb T,\mathbb R^{3}).
	\]
	We shall prove that $(\lambda_{n},\mathbf u_{n})\in\mathcal Z_{\rm reg}$ for \(n\) large enough. Set
	\[
	\mathbf{x}_{n}:=\mathbf{x}_{0}+\mathbf u_{n}, \qquad \mathbf{x}:=\mathbf{x}_{0}+\mathbf u, \qquad u_{n}:=(x_{n})_{3}, \qquad u:=x_{3}.
	\]
	Since $H^{2}(\mathbb T)\hookrightarrow \mc C^{1}(\mathbb T)$,
	we have
	\[
	\lim_{n\to+\infty}\mathbf{x}_{n}=\mathbf{x}\qquad\text{in }\mc C^{1}(\mathbb T,\mathbb R^{3}).
	\]
	In particular, $u_{n}\to u$ in $\mc C^{1}(\mathbb T)$. Since \((\lambda,\mathbf u)\in\mathcal Z_{\rm reg}\), we have
	\[
	\min_{\eta\in\mathbb T}(R^{2}-u(\eta)^{2})>0.
	\]
	By uniform convergence, for \(n\) large enough,
	\[
	\min_{\eta\in\mathbb T}(R^{2}-u_{n}(\eta)^{2})>0.
	\]
	Thus the solutions \(\mathbf x_n\) are non-polar for \(n\) large enough. Similarly, since $\operatorname{osc}(u)>0$,
	and the map
	\[
	v\mapsto \operatorname{osc}(v)=\max_{\mathbb T}v-\min_{\mathbb T}v
	\]
	is continuous with respect to the \(\mc{C}^{0}\)-topology, we have $\operatorname{osc}(u_{n})>0$ for \(n\) large enough.
	
	It remains to check the regular root configuration. Let
	\[
	e_{1}:=\min_{\mathbb T}u, \qquad e_{2}:=\max_{\mathbb T}u.
	\]
	Since \((\lambda,\mathbf u)\in\mathcal Z_{\rm reg}\), the polynomial \(P\)
	associated with \((\lambda,\mathbf u)\) has a third real root \(e_{3}\) such that $-R<e_{1}<e_{2}<R$,
	$e_{2}<e_{3}$, and
	\[
	P(s)=\frac{2\lambda}{R^{2}}(s-e_{1})(s-e_{2})(s-e_{3}).
	\]
	In particular, \(e_{1},e_{2},e_{3}\) are simple roots of \(P\).
	
	For each \(n\), let \(P_{n}\) be the polynomial associated with
	\((\lambda_{n},\mathbf u_{n})\). The conserved quantities \(C_{n},E_{n}\)
	depend continuously on \((\lambda_{n},\mathbf x_{n})\) in the \(\mc{C}^{1}\)-topology.
	Hence the coefficients of \(P_{n}\) converge to the coefficients of \(P\).
	Since the roots \(e_{1},e_{2},e_{3}\) of \(P\) are simple, there exist, for \(n\)
	large enough, three real simple roots $e_{1,n}<e_{2,n}<e_{3,n}$
	of \(P_{n}\), depending continuously on \(n\), such that $e_{i,n}\to e_{i}$ for each $i=1,2,3$.
	In particular, for \(n\) large enough,
	\[
	-R<e_{1,n}<e_{2,n}<R, \qquad e_{2,n}<e_{3,n}.
	\]
	We now show that \(u_{n}\) oscillates between \(e_{1,n}\) and \(e_{2,n}\).
	Let
	\[
	m_{n}:=\min_{\mathbb T}u_{n}, \qquad M_{n}:=\max_{\mathbb T}u_{n}.
	\]
	By uniform convergence, $m_{n}\to e_{1}$ and $M_{n}\to e_{2}$. At points where the minimum and maximum are attained one has \(u_{n}'=0\).
	Since $(u_{n}')^{2}=P_{n}(u_{n})$, it follows that $P_{n}(m_{n})=0$, $P_{n}(M_{n})=0$.
	For \(n\) large enough, the only root of \(P_{n}\) close to \(e_{1}\) is
	\(e_{1,n}\), and the only root of \(P_{n}\) close to \(e_{2}\) is \(e_{2,n}\).
	Therefore $m_{n}=e_{1,n}$ and $M_{n}=e_{2,n}$. Thus
	\[
	u_{n}(\mathbb T)=[e_{1,n},e_{2,n}].
	\]
	Consequently, for \(n\) large enough, the solution \((\lambda_{n},\mathbf u_{n})\)
	satisfies all the defining conditions of \(\mathcal Z_{\rm reg}\). Hence $(\lambda_{n},\mathbf u_{n})\in\mathcal Z_{\rm reg}$
	for all sufficiently large \(n\), which proves that \(\mathcal Z_{\rm reg}\) is
	relatively open in \(\mathcal Z\).
\end{proof}

\subsection{Horizontal degree and vertical nodal number}

On \(\mathcal Z_{\rm reg}\) there are two natural discrete invariants. First,
since $-R<x_{3}(\eta)<R$ for every $\eta$, the horizontal projection $(x_{1},x_{2})$
does not vanish. Hence we can define the horizontal degree
\begin{equation*}
	p(\mathbf{x}):=\frac1{2\pi}\int_{0}^{2\pi}\frac{x_{1}x_{2}'-x_{2}x_{1}'}{x_{1}^{2}+x_{2}^{2}}\,d\eta \in\mathbb Z.
\end{equation*}

Second, the vertical component \(u=x_{3}\) is regular oscillatory. We define the
vertical nodal number \(q(\mathbf{x})\) as the number of full oscillations of
\(u\) on \([0,2\pi]\). Equivalently, if \(T_{u}\) denotes the minimal period of
\(u\), then
\[
T_{u}=\frac{2\pi}{q(\mathbf{x})}.
\]

\begin{lemma}
	\label{Le4.2}
	The horizontal degree \(p\) and the vertical nodal number \(q\) are locally
	constant on \(\mathcal Z_{\rm reg}\). Consequently, they are constant on every
	connected component of \(\mathcal Z_{\rm reg}\).
\end{lemma}

\begin{proof}
	Let $(\lambda_{0},\mathbf u_{0})\in\mathcal Z_{\rm reg}$ with $\mathbf x^{0}:=\mathbf x_{0}+\mathbf u_{0}$, $u_{0}:=(x^{0})_{3}$. We first prove the local constancy of the horizontal degree. Since
	\((\lambda_{0},\mathbf u_{0})\in\mathcal Z_{\rm reg}\), we have $-R<u_{0}(\eta)<R$ for every $\eta\in\mathbb T$.
	Equivalently, the horizontal projection
	\[
	y_{0}:\mathbb{T}\to \R^{2}\setminus\{(0,0)\}, \quad y_{0}(\eta):=(x^{0}_{1}(\eta),x^{0}_{2}(\eta)),
	\]
	does not vanish. Since \(\mathbb T\) is compact, there exists \(\delta>0\) such
	that $|y_{0}(\eta)|\ge \delta$ for every $\eta\in\mathbb T$.
	If \((\lambda,\mathbf u)\in\mathcal Z_{\rm reg}\) is sufficiently close to
	\((\lambda_{0},\mathbf u_{0})\) in \((0,+\infty)\times H^{2}\), then, writing $\mathbf x=\mathbf x_{0}+\mathbf u$,
	$y=(x_{1},x_{2})$,
	we have $\|y-y_{0}\|_{\infty}<\delta/2$.
	Hence $|y(\eta)|\ge \delta/2$ for every $\eta\in\mathbb T$.
	The straight-line homotopy
	\[
	H:[0,1]\times \mathbb{T}\longrightarrow \R^{2}, \quad H(s,\eta):=(1-s)y_{0}(\eta)+s\,y(\eta),
	\]
	does not vanish, provided the neighbourhood is chosen sufficiently small.
	Therefore \(y_{0}\) and \(y\) have the same winding number around the origin. Thus $p(\mathbf x)=p(\mathbf x^{0})$. This proves that \(p\) is locally constant.
	
	We now prove the local constancy of the vertical nodal number. Let
	\[
	e_{1}^{0}:=\min_{\mathbb T}u_{0}, \qquad e_{2}^{0}:=\max_{\mathbb T}u_{0}.
	\]
	Since \((\lambda_{0},\mathbf u_{0})\in\mathcal Z_{\rm reg}\), the associated
	polynomial \(P_{0}\) has a third root \(e_{3}^{0}\) such that
	\[
	-R<e_{1}^{0}<e_{2}^{0}<R, \qquad e_{2}^{0}<e_{3}^{0},
	\]
	and
	\[
	P_{0}(s)=\frac{2\lambda_{0}}{R^{2}}(s-e_{1}^{0})(s-e_{2}^{0})(s-e_{3}^{0}).
	\]
	The roots are simple. Hence, for every \((\lambda,\mathbf u)\in\mathcal Z_{\rm reg}\)
	sufficiently close to \((\lambda_{0},\mathbf u_{0})\), the associated polynomial
	\(P\) has roots $e_{1}<e_{2}<e_{3}$
	depending continuously on \((\lambda,\mathbf u)\), with $e_{i}\to e_{i}^{0}$ for every $i=1,2,3$.
	Moreover, by the definition of \(\mathcal Z_{\rm reg}\), the vertical component
	\(u=x_{3}\) oscillates between \(e_{1}\) and \(e_{2}\). For such a regular oscillation,
	\[
	P(u)
	=
	\frac{2\lambda}{R^{2}}(u-e_{1})(u-e_{2})(u-e_{3}).
	\]
	On the interval \(e_{1}\le u\le e_{2}\), this is equivalently
	\[
	(u')^{2}
	=\frac{2\lambda}{R^{2}}(u-e_{1})(e_{2}-u)(e_{3}-u).
	\]
	Therefore the minimal vertical period is
	\[
	T_{u}=\frac{2K\left(\frac{e_{2}-e_{1}}{e_{3}-e_{1}}\right)}{\sqrt{\frac{\lambda}{2R^{2}}(e_{3}-e_{1})}}.
	\]
	Thus, the
	period \(T_{u}\) depends continuously on \((\lambda,\mathbf u)\) in a
	neighbourhood of \((\lambda_{0},\mathbf u_{0})\). Because every element of \(\mathcal Z_{\rm reg}\) is \(2\pi\)-periodic, the
	minimal vertical period satisfies $T_{u}=\frac{2\pi}{q(\mathbf x)}$
	for some \(q(\mathbf x)\in\mathbb N\). Hence
	\[
	q(\mathbf x)=\frac{2\pi}{T_{u}}.
	\]
	The right-hand side is continuous in \((\lambda,\mathbf u)\), while
	\(q(\mathbf x)\) takes values in the discrete set \(\mathbb N\). Therefore
	\(q\) is locally constant. Since \(p\) and \(q\) are locally constant, they are constant on every connected
	component of \(\mathcal Z_{\rm reg}\).
\end{proof}

The local bifurcation theorem was proved in the reversible subspace \(Y_{\rm rev}\).
However, every solution obtained there is, of course, a solution of the full
augmented problem. We shall therefore regard the local Crandall--Rabinowitz
branch as a subset of the full zero set \(\mathcal Z\).

For each \(k\ge2\), the local branch bifurcating from $(\lambda_{k},0)$, $\lambda_{k}=R\sqrt{k^{2}-1}$,
has the form $\lambda=\Lambda_{k}(t)$, $\mathbf u=\mathbf u_{k}(t)$, $\mathbf u_{k}(0)=0$.
Moreover, if $\mathbf x_{k}(t):=\mathbf x_{0}+\mathbf u_{k}(t)$,
then its vertical component satisfies
\begin{equation}
	\label{Eq4.8}
	x_{k,3}(t,\eta)=-t\,\frac{k}{\sqrt{k^{2}-1}}\cos(k\eta)+o(t)\quad \text{in }\mc{C}^{1}(\mathbb T)
\end{equation}
as \(t\to0\). In the full zero set, the corresponding solutions obtained by the
phase symmetry have vertical component $x_{k,3}(t,\eta-\varphi)$ and satisfy the same conclusions below.

\begin{lemma}
	\label{Le4.3}
	For \(t\neq0\) sufficiently small, the nontrivial local branch bifurcating from
	\((\lambda_{k},0)\) belongs to \(\mathcal Z_{\rm reg}\). Moreover, on this local
	branch one has
	\[
	p=1, \qquad q=k.
	\]
	The same holds for the phase orbit of this branch in the full zero set
	\(\mathcal Z\).
\end{lemma}

\begin{proof}
	We first prove the statement for the branch in the reversible subspace
	\(Y_{\rm rev}\). Let
	\[
	(\Lambda_k(t),\mathbf u_k(t)), \qquad \mathbf x_k(t):=\mathbf x_0+\mathbf u_k(t),
	\]
	be the local Crandall--Rabinowitz branch. Its vertical component satisfies \eqref{Eq4.8}
	in \(\mc{C}^1(\mathbb T)\) as \(t\to0\). In fact, since the branch consists of
	solutions of the smooth profile equation, the solutions are smooth and the
	expansion may be read in \(\mc{C}^2(\mathbb T)\) after the usual bootstrap. Hence,
	for \(t\neq0\) sufficiently small, \(x_{k,3}(t,\cdot)\) is a \(\mc{C}^2\)-small
	perturbation of a nonzero multiple of \(\cos(k\eta)\).
	
	The function \(\cos(k\eta)\) has exactly \(k\) non-degenerate maxima and
	\(k\) non-degenerate minima on \([0,2\pi)\). By stability of non-degenerate
	critical points under \(\mc{C}^2\)-small perturbations, the function
	\(x_{k,3}(t,\cdot)\) has exactly \(k\) maxima and \(k\) minima on
	\([0,2\pi)\), for \(t\neq0\) sufficiently small. In particular, $\operatorname{osc}(x_{k,3}(t,\cdot))>0$,
	and the vertical component completes exactly \(k\) full oscillations on one
	\(2\pi\)-period. Thus $q(\mathbf x_k(t))=k$. Moreover,
	\[
	\lim_{t\to 0}\mathbf x_k(t) = \mathbf x_0 \qquad \text{in }H^2(\mathbb T,\mathbb R^3).
	\]
	Since \(H^2(\mathbb T)\hookrightarrow \mc{C}^1(\mathbb T)\), we also
	have
	\[
	\lim_{t\to 0}\mathbf x_k(t)=\mathbf x_0 \qquad \text{in }\mc{C}^1(\mathbb T,\mathbb R^3).
	\]
	In particular,
	\[
	\lim_{t\to 0} x_{k,3}(t,\eta)=0 \qquad \text{uniformly in }\eta.
	\]
	Therefore, for \(t\neq0\) sufficiently small, $-R<x_{k,3}(t,\eta)<R$ for every $\eta\in\mathbb T$.
	Hence the solution does not touch the poles.
	
	It remains to verify the regular cubic condition in
	Definition~\ref{Df4.1}. Set
	\[
	u_t:=x_{k,3}(t,\cdot), \qquad e_1(t):=\min_{\mathbb T}u_t, \qquad e_2(t):=\max_{\mathbb T}u_t.
	\]
	From the previous discussion, for \(t\neq0\) sufficiently small,
	\[
	-R<e_1(t)<e_2(t)<R.
	\]
	Let \(P_t\) be the polynomial associated with the solution
	\(\mathbf x_k(t)\). Since \(u_t\) satisfies $(u_t')^2=P_t(u_t)$, at points where
	the minimum and maximum of \(u_t\) are attained, one has \(u_t'=0\). Hence $P_t(e_1(t))=P_t(e_2(t))=0$.
	The function \(u_t\) is nonconstant and periodic, therefore \(e_1(t)\) and
	\(e_2(t)\) are simple roots of \(P_t\). Indeed, if one of these turning values
	were a multiple root, the corresponding nonconstant trajectory of
	\((u')^2=P_t(u)\) would reach it only asymptotically, contradicting the
	periodic oscillatory motion.
	
	Since \(P_t\) is a cubic polynomial, it has a third real root, which we denote
	by \(e_3(t)\). On the interval \((e_1(t),e_2(t))\), the solution moves
	between the two turning values and therefore $P_t(s)>0$ for $s\in(e_1(t),e_2(t))$.
	The leading coefficient of \(P_t\) is
	\[
	\frac{2\Lambda_k(t)}{R^2}>0.
	\]
	Because \(e_1(t)\) and \(e_2(t)\) are simple roots and \(P_t\) is positive
	between them, the third root must lie to the right of \(e_2(t)\). Hence $e_1(t)<e_2(t)<e_3(t)$
	and
	\[
	P_t(s)=\frac{2\Lambda_k(t)}{R^2}(s-e_1(t))(s-e_2(t))(s-e_3(t)).
	\]
	Thus $(\Lambda_k(t),\mathbf u_k(t))\in\mathcal Z_{\rm reg}$
	for \(t\neq0\) sufficiently small.
	
	Finally, the horizontal projection of \(\mathbf x_k(t)\) is a \(\mc{C}^1\)-small
	perturbation of $(R\cos\eta,R\sin\eta)$,
	which has winding number \(1\). Therefore, for \(t\neq0\) sufficiently small, $p(\mathbf x_k(t))=1$.
	We already proved that $q(\mathbf x_k(t))=k$.
	
	The phase orbit of the local branch is obtained by applying
	\[
	(\rho_\varphi\mathbf u)(\eta)=R_\varphi\mathbf u(\eta-\varphi).
	\]
	This operation preserves the sphere constraint, non-polarity, the root
	configuration of the vertical polynomial, and the horizontal winding number.
	Moreover, the vertical component is merely translated:
	\[
	(\rho_\varphi\mathbf u)_3(\eta)=u_3(\eta-\varphi),
	\]
	so the number of vertical oscillations is unchanged. Hence the same
	conclusions hold for the whole phase orbit in the full zero set
	\(\mathcal Z\).
\end{proof}

We now define \(\mathscr C_{k}^{\rm reg}\) to be any connected component of
\(\mathcal Z_{\rm reg}\) which contains a nontrivial local solution given by
Lemma~\ref{Le4.3}, or one of its phase
translates in the full zero set. By Lemmas~\ref{Le4.2}
and \ref{Le4.3}, we have
\begin{equation}
	\label{Eq4.9}
	p=1,\qquad q=k\qquad\text{on }\mathscr C_{k}^{\rm reg}.
\end{equation}

\subsection{A priori estimates at fixed vertical nodal number}

We next prove the a priori estimate needed for compactness. The estimate is
stated for an arbitrary fixed vertical nodal number \(q\), and will then be
applied with \(q=k\) on \(\mathscr C_{k}^{\rm reg}\).

\begin{proposition}
	\label{Pr4.1}
	Let \(q\in\mathbb N\) and \(\Lambda>0\). Let $(\lambda,\mathbf u)\in\mathcal Z_{\rm reg}$
	with $0<\lambda\le\Lambda$
	and vertical nodal number \(q\). Set $\mathbf{x}:=\mathbf{x}_{0}+\mathbf u$.
	Then there exists a constant $C=C(R,\Lambda,q)>0$
	such that
	\[
	\|\mathbf{x}\|_{W^{2,\infty}(\mathbb T,\mathbb R^{3})}
	\le C.
	\]
	In particular, $\|\mathbf u\|_{H^{2}(\mathbb T,\mathbb R^{3})}\le C$.
\end{proposition}

\begin{proof}
	Let \((\lambda,\mathbf u)\in\mathcal Z_{\rm reg}\). Thus $-R<e_{1}<e_{2}<R$, $e_{2}<e_{3}$,
	and \(u=x_{3}\) oscillates between \(e_{1}\) and \(e_{2}\). Set $A:=e_{3}-e_{1}$, $d:=e_{2}-e_{1}$ and $\mu:=\frac{d}{A}$. Then \(0<\mu<1\). From \((u')^{2}=P(u)\), the vertical component has the standard
	form
	\[
	u(\eta)=e_{1}+d\,\sn^{2}(\alpha(\eta-\eta_{0})\mid\mu),
	\]
	with $\alpha^{2}=\frac{\lambda}{2R^{2}}A$. Since \(\sn^{2}(\cdot)\) has period \(2K(\mu)\), the minimal vertical
	period is $T_{u}=\frac{2K(\mu)}{\alpha}$.
	Because the vertical nodal number is \(q\), $T_{u}=\frac{2\pi}{q}$.
	Therefore
	\[
	\alpha=\frac{qK(\mu)}{\pi},
	\]
	and hence
	\begin{equation}
		\label{Eq4.10}
		\lambda A=\frac{2R^{2}q^{2}}{\pi^{2}}K(\mu)^{2}.
	\end{equation}
	We now prove that \(\lambda A\) is bounded in terms of \(R,\Lambda,q\). Since $d=e_{2}-e_{1}\le2R$
	and \(d=\mu A\), we have
	\[
	A\le\frac{2R}{\mu}.
	\]
	Using \(0<\lambda\le\Lambda\), we get $\lambda A\le\frac{2\Lambda R}{\mu}$.
	Combining this with \eqref{Eq4.10}, we find
	\[
	\mu K(\mu)^{2}\le\frac{\Lambda\pi^{2}}{Rq^{2}}.
	\]
	The function $\mu\mapsto \mu K(\mu)^{2}$
	tends to \(+\infty\) as \(\mu\to1^{-}\). Hence \(\mu\) stays bounded away from
	\(1\), and therefore \(K(\mu)\) is bounded. From
	\eqref{Eq4.10} we obtain
	\begin{equation}
		\label{Eq4.11}
	\lambda(e_{3}-e_{1})\le M(R,\Lambda,q).
	\end{equation}
	Recall that by \eqref{Eq4.6},
	\[
	\mathcal H=\lambda(e_{1}+e_{2}+e_{3})-\frac{\lambda^{2}}{2}.
	\]
	Since \(e_{1},e_{2}\in[-R,R]\), \(0<\lambda\le\Lambda\), by \eqref{Eq4.11}
	it follows that $|\mathcal H|\le H(R,\Lambda,q)$.
	Using \eqref{Eq4.2}, we deduce that
	\[
	\|\mathbf{x}'\|^{2}=2(\mathcal H-\lambda x_{3})\le 2|\mathcal H|+2\Lambda R.
	\]
	Hence $\|\mathbf{x}'\|_{L^{\infty}}\le C(R,\Lambda,q)$.
	Finally, using \eqref{Eq4.7}, we obtain $\|\mathbf{x}''\|_{L^{\infty}}\le C(R,\Lambda,q)$.
	Since \(\|\mathbf{x}\|=R\), this proves the \(W^{2,\infty}\)-bound. The
	\(H^{2}\)-bound for \(\mathbf u=\mathbf{x}-\mathbf{x}_{0}\) follows because
	\(\mathbf{x}_{0}\) is fixed and smooth.
\end{proof}

The previous estimate implies strong compactness for bounded sequences in the
regular component.

\begin{proposition}
	Let \(k\ge2\), and let $\{(\lambda_{n},\mathbf u_{n})\}_{n\in\mathbb N}\subset\mathscr C_{k}^{\rm reg}$ be a sequence such that $\lambda_{*}\le\lambda_{n}\le\Lambda$ for every $n\in\N$ and some constants \(0<\lambda_{*}\le \Lambda\). Then, after passing to a subsequence, there exists $(\lambda,\mathbf u)\in\mathcal Z$
	such that
	\[
	\lim_{n\to+\infty}(\lambda_{n},\mathbf u_{n})=(\lambda,\mathbf u)\qquad \text{in }(0,+\infty)\times H^{2}(\mathbb T,\mathbb R^{3}).
	\]
	If the limit belongs to \(\mathcal Z_{\rm reg}\), then it still satisfies $p=1$, $q=k$.
\end{proposition}

\begin{proof}
	By \eqref{Eq4.9}, every element of
	\(\mathscr C_{k}^{\rm reg}\) has vertical nodal number \(k\). Therefore
	Proposition~\ref{Pr4.1} gives
	\[
	\|\mathbf u_{n}\|_{H^{2}(\mathbb T,\R^{3})}\le C(R,\Lambda,k).
	\]
	Thus, after passing to a subsequence, there exist $\lambda\in[\lambda_{*},\Lambda]$
	and $\mathbf u\in H^{2}(\mathbb T,\mathbb R^{3})$ such that
	\[
	\lambda_{n}\to\lambda,\qquad\mathbf u_{n}\rightharpoonup\mathbf u\quad\text{weakly in }H^{2}(\mathbb T,\mathbb R^{3}).
	\]
	Since $H^{2}(\mathbb T)\hookrightarrow \mc{C}^{1}(\mathbb T)$ compactly, we also have, after passing to a further subsequence,
	\[
	\lim_{n\to+\infty}\mathbf u_{n}=\mathbf u\qquad\text{in }\mc C^{1}(\mathbb T,\mathbb R^{3}).
	\]
	Set $\mathbf x_{n}:=\mathbf x_{0}+\mathbf u_{n}$, $\mathbf x:=\mathbf x_{0}+\mathbf u$.
	Then $\mathbf x_{n}\to\mathbf x$ in $\mc C^{1}(\mathbb T,\mathbb R^{3})$. For every \(n\in\N\), the curve \(\mathbf x_{n}\) satisfies
	\[
	\mathbf{x}_{n}\times\mathbf{x}_{n}''-\lambda_{n}e_{3}\times\mathbf{x}_{n}+\lambda_{n}\mathbf{x}_{n}'=0,\qquad\|\mathbf x_n\|\equiv R.
	\]
	Using \eqref{Eq4.7}, we can write
	\[
	\mathbf{x}_{n}''=-\frac{\|\mathbf{x}_{n}'\|^{2}}{R^{2}}\mathbf{x}_{n}-\lambda_{n}e_{3}+\frac{\lambda_{n}(x_{n})_{3}}{R^{2}}\mathbf{x}_{n}+\frac{\lambda_{n}}{R^{2}}\mathbf{x}_{n}\times\mathbf{x}_{n}'.
	\]
	The right-hand side depends continuously on $(\lambda_{n},\mathbf x_{n},\mathbf x_{n}')$
	in the \(\mc{C}^{0}\)-topology. Since $\lambda_{n}\to\lambda$ and $\mathbf x_{n}\to\mathbf x$ in $\mc{C}^{1}$,
	the right-hand sides converge strongly in \(\mc{C}^{0}\), and hence in \(L^{2}\), to
	\[
	\mathbf F:=-\frac{\|\mathbf{x}'\|^{2}}{R^{2}}\mathbf{x}-\lambda e_{3}+\frac{\lambda x_{3}}{R^{2}}\mathbf{x}+\frac{\lambda}{R^{2}}\mathbf{x}\times\mathbf{x}'.
	\]
	Therefore $\mathbf x_{n}''\to \mathbf F$ strongly in $L^{2}(\mathbb T,\mathbb R^{3})$.
	On the other hand, since $\mathbf u_{n}\rightharpoonup\mathbf u$ weakly in $H^{2}$, we also have $\mathbf x_{n}''\rightharpoonup\mathbf x''$ weakly in $L^{2}$.
	By uniqueness of weak limits, $\mathbf x''=\mathbf F$. Hence $\mathbf x_{n}''\to\mathbf x''$ strongly in $L^{2}$.
	Combining this with the already known convergence in \(\mc{C}^{1}\), and hence in
	\(H^{1}\), we obtain $\mathbf x_{n}\to\mathbf x$ strongly in $H^{2}(\mathbb T,\mathbb R^{3})$.
	Equivalently,
	\[
	\lim_{n\to+\infty}\mathbf u_{n}=\mathbf u\qquad	\text{strongly in }H^{2}(\mathbb T,\mathbb R^{3}).
	\]
	It remains to show that the limit belongs to \(\mathcal Z\). Since $\mathcal G(\lambda_{n},\mathbf u_{n})=(0,0)$ for every \(n\), and since \(\mathcal G\) is continuous with respect to the
	\((0,+\infty)\times H^{2}\)-topology, the strong convergence just proved gives $\mathcal G(\lambda,\mathbf u)=(0,0)$. Thus $(\lambda,\mathbf u)\in\mathcal Z$. Finally, if the limit belongs to \(\mathcal Z_{\rm reg}\), then the identities $p=1$, $q=k$ follow from Lemma~\ref{Le4.2}, because
	\(p\) and \(q\) are locally constant on \(\mathcal Z_{\rm reg}\).
\end{proof}
\subsection{The regular global alternative}
We denote by
\[
\partial_{\rm ext}\mathcal Z_{\rm reg}:=\overline{\mathcal Z_{\rm reg}}^{\, [0,+\infty)\times H^{2}} \setminus \mathcal Z_{\rm reg}
\]
the extended boundary of the regular class, where the closure is taken in
\[
[0,+\infty)\times H^{2}(\mathbb T,\mathbb R^{3}).
\]
The previous proposition shows that, inside the regular class, boundedness of
\(\lambda\) prevents any hidden loss of compactness. Thus a regular continuation
can fail to remain compact only if \(\lambda\to+\infty\) or if the branch
approaches the boundary of the regular class.

\begin{theorem}[Global alternative inside the regular class]
	\label{Th4.1}
	Let \(k\ge2\), and let \(\mathscr C_{k}^{\rm reg}\) be a connected component of
	\(\mathcal Z_{\rm reg}\) containing one of the nontrivial local solutions
	bifurcating from $(\lambda_k,0)$, $\lambda_{k}=R\sqrt{k^{2}-1}$,
	or one of its phase translates.
	Then \(p=1\), \(q=k\) on \(\mathscr C_{k}^{\rm reg}\).
	Moreover, one of the following alternatives holds:
	\begin{enumerate}
		\item[\rm(i)] $\mathscr C_{k}^{\rm reg}$
		is relatively compact in \(\mathcal Z_{\rm reg}\cup \{(\l_k,0)\}\);
		\item[\rm(ii)] there exists a sequence $\{(\lambda_n,\mathbf u_n)\}_{n\in\N}\subset\mathscr C_{k}^{\rm reg}$
		such that $\lambda_n\to+\infty$;
		
		\item[\rm(iii)] there exists a sequence $\{(\lambda_n,\mathbf u_n)\}_{n\in\N}\subset\mathscr C_{k}^{\rm reg}$ and a point $(\lambda,\mathbf u)\in
		\partial_{\rm{ext}}\mathcal Z_{\rm reg}\setminus\{(\l_k,0)\}$
		such that
		\[
		\lim_{n\to+\infty}(\lambda_n,\mathbf u_n)=(\lambda,\mathbf u) \quad \text{in }	[0,+\infty)\times H^{2}(\mathbb T,\mathbb R^{3}).
		\]
	\end{enumerate}
\end{theorem}

\begin{proof}
	The identities $p=1$, $q=k$ on \(\mathscr C_{k}^{\rm reg}\) follow from the local bifurcation expansion and
	from the local constancy of \(p\) and \(q\) on \(\mathcal Z_{\rm reg}\).
	
	We now prove the alternative. Assume that alternatives \({\rm(ii)}\) and
	\({\rm(iii)}\) do not occur. We shall prove that \({\rm(i)}\) holds. Since \({\rm(ii)}\) does not occur, the parameter \(\lambda\) is bounded above on \(\mathscr C_{k}^{\rm reg}\). Hence there exists \(\Lambda>0\) such that $0<\lambda\le\Lambda$ on $\mathscr C_{k}^{\rm reg}$. Let $\{(\lambda_n,\mathbf u_n)\}_{n\in\mathbb N}
	\subset \mathscr C_{k}^{\rm reg}$ be an arbitrary sequence, and set $\mathbf x_n:=\mathbf x_0+\mathbf u_n$.
	Since \(q=k\) on \(\mathscr C_k^{\rm reg}\), the a priori estimate at fixed
	vertical nodal number gives
	\[
	\|\mathbf x_n\|_{W^{2,\infty}}\le C(R,\Lambda,k).
	\]
	In particular, after passing to a subsequence,
	\[
	\lim_{n\to+\infty}\lambda_n=\lambda\in[0,\Lambda],
	\]
	and
	\[
	\lim_{n\to+\infty}\mathbf x_n=\mathbf x \quad\text{in } \mc{C}^{1}(\mathbb T,\mathbb R^{3}), \qquad \mathbf x_n\rightharpoonup\mathbf x\qquad\text{weakly in }H^{2}(\mathbb T,\mathbb R^{3}),
	\]
	for some \(\mathbf x\in H^2(\mathbb T,\mathbb R^{3})\). 
	We claim that the convergence is actually strong in \(H^{2}\). Indeed, by
	\eqref{Eq4.7},
	\[
	\mathbf{x}_{n}''=-\frac{\|\mathbf{x}_{n}'\|^{2}}{R^{2}}\mathbf{x}_{n}-\lambda_{n}e_{3}+\frac{\lambda_{n}(x_{n})_{3}}{R^{2}}\mathbf{x}_{n}+\frac{\lambda_{n}}{R^{2}}\mathbf{x}_{n}\times\mathbf{x}_{n}'.
	\]
	The right-hand side depends continuously on $(\lambda_n,\mathbf x_n,\mathbf x_n')$
	in the \(\mc{C}^{0}\)-topology. Since $\lambda_n\to\lambda$ and $\mathbf x_n\to\mathbf x$ in $\mc{C}^{1}$, we obtain
	\[
	\mathbf x_n''\to-\frac{\|\mathbf{x}'\|^{2}}{R^{2}}\mathbf{x}-\lambda e_{3}+\frac{\lambda x_{3}}{R^{2}}\mathbf{x}+\frac{\lambda}{R^{2}}\mathbf{x}\times\mathbf{x}'
	\]
	strongly in \(\mc{C}^{0}\), and hence strongly in \(L^{2}\). On the other hand, $\mathbf x_n''\rightharpoonup \mathbf x''$ weakly in $L^{2}$.
	Therefore $\mathbf x_n''\to \mathbf x''$ strongly in $L^{2}$. Consequently,
	\[
	\lim_{n\to+\infty}\mathbf x_n=\mathbf x\quad\text{strongly in }H^{2}(\mathbb T,\mathbb R^{3}).
	\]
	Equivalently, if $\mathbf u:=\mathbf x-\mathbf x_0$, then
	\[
	\lim_{n\to+\infty}\mathbf u_n=\mathbf u\quad\text{strongly in }H^{2}(\mathbb T,\mathbb R^{3}).
	\]
	We now identify the possible limit $(\lambda,\mathbf u)
	\in
	[0,+\infty)\times H^2(\mathbb T,\mathbb R^3)$.
	First suppose that $\lambda=0$.
	Then $(\lambda_n,\mathbf u_n)\to(0,\mathbf u)$
	in $[0,+\infty)\times H^2(\mathbb T,\mathbb R^3)$.
	Since every \((\lambda_n,\mathbf u_n)\) belongs to \(\mathcal Z_{\rm reg}\),
	the point \((0,\mathbf u)\) belongs to the closure of \(\mathcal Z_{\rm reg}\)
	in the extended space. Moreover, $(0,\mathbf u)\notin \mathcal Z_{\rm reg}$,
	because
	\[
	\mathcal Z_{\rm reg}\subset (0,+\infty)\times H^2(\mathbb T,\mathbb R^3).
	\]
	Hence $(0,\mathbf u)\in \partial_{\rm ext}\mathcal Z_{\rm reg}$.
	Since $\lambda_k=R\sqrt{k^2-1}>0$,
	we have $(0,\mathbf u)\neq (\l_k,0)$.
	Therefore alternative \({\rm(iii)}\) occurs, contrary to our assumption.
	Thus $\lambda>0$ and for \(n\) large enough,
	\[
	\lambda_n\ge \frac{\lambda}{2}>0.
	\]
	Since $(\lambda_n,\mathbf u_n)\to(\lambda,\mathbf u)$
	strongly in $(0,+\infty)\times H^2(\mathbb T,\mathbb R^3)$
	and $\mathcal G(\lambda_n,\mathbf u_n)=0$
	for every \(n\), the continuity of \(\mathcal G\) gives	$\mathcal G(\lambda,\mathbf u)=0$.
	Hence $(\lambda,\mathbf u)\in\mathcal Z$.
	If $(\lambda,\mathbf u)\in\mathcal Z_{\rm reg}$,
	then the limit belongs to the desired enlarged space. Suppose, on the other
	hand, that $(\lambda,\mathbf u)\notin\mathcal Z_{\rm reg}$.
	Since \((\lambda,\mathbf u)\) is a limit, in the extended space, of points of
	\(\mathcal Z_{\rm reg}\), it belongs to $\partial_{\rm ext}\mathcal Z_{\rm reg}$.
	If $(\lambda,\mathbf u)\neq (\l_k,0)$,
	then alternative \({\rm(iii)}\) occurs, contrary to our assumption. Therefore
	the only remaining possibility is $(\lambda,\mathbf u)=(\l_k,0)$.
	In all cases we have proved that
	\[
	(\lambda,\mathbf u)\in \mathcal Z_{\rm reg}\cup\{(\l_k,0)\}.
	\]
	Thus every sequence in \(\mathscr C_k^{\rm reg}\) admits a subsequence
	converging, in $[0,+\infty)\times H^2(\mathbb T,\mathbb R^3)$,
	to a point of $\mathcal Z_{\rm reg}\cup\{(\l_k,0)\}$.
	Equivalently, the closure of \(\mathscr C_k^{\rm reg}\) in
	\(\mathcal Z_{\rm reg}\cup\{(\l_k,0)\}\), endowed with the topology induced by the
	extended space, is sequentially compact.
	This is alternative \({\rm(i)}\), and the proof is complete.
\end{proof}

	The explicit elliptic reduction of Section~2 also gives a finite-dimensional
	system whose solutions describe the regular closed profiles, up to the elementary
	phase parameters. We record this system in order to make explicit the relation
	between the global components considered above and the classical closure
	conditions.
	
	Fix \(p\in\mathbb Z\) and \(q\in\mathbb N\). The unknowns of the system are $(e_1,e_2,e_3,C,\lambda)\in\mathbb R^5$,
	subject to the regularity constraints
	\[
	\lambda>0,\qquad -R<e_1<e_2<R,\qquad e_2<e_3.
	\]
	Here \(e_1,e_2,e_3\) are the roots of the vertical polynomial, \(C\) is the
	conserved vertical angular momentum, $C=(\mathbf x\times\mathbf x')_3+\lambda x_3$,
	and \(\lambda\) is the bifurcation parameter of the rescaled equation.
	Set
	\[
	m:=\frac{e_2-e_1}{e_3-e_1},
	\qquad
	\alpha:=
	\sqrt{\frac{\lambda}{2R^2}(e_3-e_1)}.
	\]
	Then the corresponding algebraic--transcendental system is
	\begin{equation}
		\label{Sis}
	\mathscr S_{p,q}:
	\begin{cases}
		(C-\lambda R)^2
		=
		2\lambda (R-e_1)(R-e_2)(e_3-R),
		\\[0.8em]
		(C+\lambda R)^2
		=
		2\lambda (R+e_1)(R+e_2)(R+e_3),
		\\[0.8em]
		\dfrac{2qK(m)}{\alpha}=2\pi,
		\\[1.2em]
		\dfrac{1}{\alpha}
		\left[
		\dfrac{C-\lambda R}{R(R-e_1)}
		\Pi\left(\dfrac{e_2-e_1}{R-e_1}\,\middle|\,m\right)
		+
		\dfrac{C+\lambda R}{R(R+e_1)}
		\Pi\left(-\dfrac{e_2-e_1}{R+e_1}\,\middle|\,m\right)
		\right]
		=
		\dfrac{2\pi p}{q}.
	\end{cases}
	\end{equation}
	The first two equations are the endpoint identities obtained from
	\[
	P(s)=\frac{2\lambda}{R^2}(s-e_1)(s-e_2)(s-e_3),
	\]
	together with
	\[
	P(R)=-\frac{(C-\lambda R)^2}{R^2},
	\qquad
	P(-R)=-\frac{(C+\lambda R)^2}{R^2}.
	\]
	The third equation imposes that the vertical component performs exactly \(q\)
	oscillations over the period \(2\pi\), since its minimal period is
	\[
	T_u=\frac{2K(m)}{\alpha}.
	\]
	The fourth equation is the azimuthal closure condition with horizontal degree
	\(p\).
	
	Given a solution of \(\mathscr S_{p,q}\), the associated vertical component is
	\[
	x_3(\eta)
	=
	e_1+(e_2-e_1)\sn^2\big(\alpha(\eta-\eta_0)\mid m\big),
	\]
	where \(\eta_0\in\mathbb R\) is a free phase. The azimuthal angle is then obtained
	by integrating
	\[
	\phi'(\eta)=\frac{C-\lambda x_3(\eta)}{R^2-x_3(\eta)^2},
	\]
	and the fourth equation ensures that the full spherical profile closes after
	\(q\) vertical oscillations with horizontal degree \(p\). The additional additive
	constant in \(\phi\) gives the second phase parameter.
	
	Thus, modulo the two elementary phase parameters \(\eta_0\) and \(\phi_0\), the
	regular closed profiles with invariants \((p,q)\) are described by the solution
	set of \(\mathscr S_{p,q}\). In particular, for the regular component
	\(\mathscr C_k^{\mathrm{reg}}\) bifurcating from the equator one has $p=1$ and $q=k$.
	Hence \(\mathscr C_k^{\mathrm{reg}}\), modulo phase, is contained in a connected
	component of the explicit system \(\mathscr S_{1,k}\). Although this system is finite-dimensional, its global analysis is still nontrivial, it involves complete elliptic integrals of the first and third
	kind, singular limits at the boundary of the regular regime, and possible
	changes in the root configuration of the vertical polynomial. In Section~\ref{SN}, we return to this explicit system and perform a numerical continuation of the branches \(\mathscr S_{1,k}\) corresponding to the regular components bifurcating from the equatorial branch. The full global analysis of \(\mathscr S_{p,q}\) for arbitrary \((p,q)\), including all possible boundary regimes, is left open.

\subsection{The boundary of the regular class}

We finally describe the possible ways in which a sequence of regular non-polar
oscillatory solutions may leave the regular class. The description is given in
terms of the roots of the polynomial \(P\). We do not attempt here to solve the
closure conditions on the boundary; rather, we identify the geometric meaning of
the possible degeneracies. Note that this concerns limits with positive limiting parameter. The extended boundary also contains possible limits with \(\lambda=0\), which are not classified by the root degeneracies below. For a classification of the case $\l=0$ we refer to Appendix \ref{ApB}.

\begin{lemma}
	\label{Le4.4}
	Let $\{(\lambda_{n},\mathbf u_{n})\}_{n\in\N}\subset\mathcal Z_{\rm reg}$
	be a sequence such that
	\[
	\lim_{n\to+\infty}(\lambda_{n},\mathbf u_{n})=(\lambda,\mathbf u) \quad \text{in }(0,+\infty)\times H^{2}(\mathbb T,\mathbb R^{3}),
	\]
	with $(\lambda,\mathbf u)\in\mathcal Z\setminus\mathcal Z_{\rm reg}$.
	For each \(n\), let
	\[
	P_{n}(s)=\frac{2\lambda_{n}}{R^{2}}(s-e_{1,n})(s-e_{2,n})(s-e_{3,n})
	\]
	be the regular factorization associated with
	\((\lambda_{n},\mathbf u_{n})\). Then, after passing to a subsequence, at least
	one of the following alternatives occurs:
	\begin{enumerate}
		\item[{\rm (i)}] $e_{1,n}\to -R$ as $n\to+\infty$.
		\item[{\rm (ii)}] $e_{2,n}\to R$ as $n\to+\infty$.
		\item[{\rm (iii)}] $e_{2,n}-e_{1,n}\to0$ as $n\to+\infty$.
		\item[{\rm (iv)}] $e_{3,n}-e_{2,n}\to0$ as $n\to+\infty$.
	\end{enumerate}
\end{lemma}

\begin{proof}
	Set $\mathbf{x}_{n}:=\mathbf{x}_{0}+\mathbf u_{n}$, $u_{n}:=(x_{n})_{3}$, $\mathbf{x}:=\mathbf{x}_{0}+\mathbf u$, $u:=x_{3}$.
	Since \(H^{2}(\mathbb T)\hookrightarrow \mc C^{1}(\mathbb T)\), we have
	\[
	\lim_{n\to+\infty}u_{n}= u	\qquad
	\text{in }\mc C^{1}(\mathbb T).
	\]
	In particular,
	\[
	\min_{\mathbb T}u_{n}\to\min_{\mathbb T}u, \qquad \max_{\mathbb T}u_{n}\to\max_{\mathbb T}u.
	\]
	Since
	\[
	e_{1,n}=\min_{\mathbb T}u_n,
	\qquad e_{2,n}=\max_{\mathbb T}u_n,
	\]
	we get
	\[
	e_{1,n}\to e_{1}:=\min_{\mathbb T}u,
	\qquad e_{2,n}\to e_{2}:=\max_{\mathbb T}u.
	\]
	
	The conserved quantities \(C_{n},E_{n}\) associated with \(\mathbf x_n\) converge
	to the conserved quantities \(C,E\) associated with \(\mathbf x\). Therefore the
	coefficients of the polynomials \(P_{n}\) converge to the coefficients of the
	limiting polynomial \(P\). Since
	\[
	\frac{2\lambda_n}{R^2}\to \frac{2\lambda}{R^2}>0,
	\]
	the leading coefficients stay bounded away from zero. Hence the roots of
	\(P_n\) remain bounded. After passing to a further subsequence, we may assume
	that $e_{3,n}\to e_3$,
	where \(e_3\) is a real root of \(P\). For every \(n\), the regularity conditions give
	\[
	-R<e_{1,n}<e_{2,n}<R,
	\qquad	e_{2,n}<e_{3,n}.
	\]
	If none of the four alternatives stated in the lemma occurred, then the limits
	would satisfy
	\[
	-R<e_{1}<e_{2}<R,\qquad e_{2}<e_{3}.
	\]
	Moreover, since \(u_{n}(\mathbb T)=[e_{1,n},e_{2,n}]\) and \(u_{n}\to u\)
	uniformly, we have
	\[
	u(\mathbb T)=[e_{1},e_{2}].
	\]
	Thus the limiting solution satisfies all the defining conditions of
	\(\mathcal Z_{\rm reg}\), contradicting $(\lambda,\mathbf u)\notin\mathcal Z_{\rm reg}$.
	Therefore at least one of the four alternatives must occur.
\end{proof}

We first interpret the alternatives $e_{2,n}\to R$ and $e_{1,n}\to -R$. Suppose that $e_{2,n}\to R$.
Passing to the limit in the identity $P_n(e_{2,n})=0$, and using the convergence of the coefficients of \(P_n\) to those of \(P\), we obtain $P(R)=0$. In the rescaled problem \(a=\Omega=\lambda\), the identity from
Section~\ref{S2}, $P(R)=-\frac{(C-\lambda R)^{2}}{R^{2}}$ therefore gives $C=\lambda R$.
Thus this boundary alternative corresponds to contact with the north pole. Similarly, if $e_{1,n}\to -R$,
then $P(-R)=0$. Since $P(-R)=-\frac{(C+\lambda R)^{2}}{R^{2}}$, we obtain $C=-\lambda R$. Thus this boundary alternative corresponds to contact with the south pole.

The alternative $e_{2,n}-e_{1,n}\to0$
corresponds to collapse of the vertical oscillation. Indeed, since $u_n(\mathbb T)=[e_{1,n},e_{2,n}]$,
we have
\[
\operatorname{osc}(u_n)=e_{2,n}-e_{1,n}\to0.
\]
Hence any \(\mc{C}^0\)-limit satisfies $x_3(\eta)\equiv u_0$
for some constant \(u_0\in[-R,R]\). Let us recall the corresponding solutions. Suppose first that \(|u_0|<R\), so
that the limiting curve is non-polar. Then the horizontal projection has a
constant radius
\[
\rho=\sqrt{R^{2}-u_0^{2}}>0.
\]
Writing its winding number as \(p\in\mathbb Z\), the curve has the form
\[
\mathbf x(\eta)=\big(\rho\cos(p\eta+\varphi),\rho\sin(p\eta+\varphi),u_0\big).
\]
The case \(p=0\) would give a constant non-polar curve, which does not solve
\eqref{Eq4.1} for \(\lambda>0\). Hence $p\in\mathbb Z\setminus\{0\}$.
For such a curve,
\[
\mathbf x'=p\,e_{3}\times\mathbf x,\qquad\mathbf x''=-p^{2}(\mathbf x-u_{0}e_{3}).
\]
Substitution into \eqref{Eq4.1} gives
\[
(-p^{2}u_{0}-\lambda+\lambda p)\,e_{3}\times\mathbf x=0.
\]
Since \(|u_0|<R\), the horizontal projection is nonzero, and therefore
\[
u_{0}=\frac{\lambda(p-1)}{p^{2}}.
\]
Conversely, every curve of the above form satisfying this relation solves
\eqref{Eq4.1}. In particular, in horizontal degree \(p=1\), collapse of the vertical oscillation
inside the non-polar class can only lead to $u_0=0$.
Thus, for the branch bifurcating from the equator, collapse without prior contact
with the poles can only lead back to the equatorial branch. The polar constant solutions $\mathbf x\equiv Re_{3}$, $\mathbf x\equiv -Re_{3}$ are the limiting cases \(|u_0|=R\) and solve the equation for every
\(\lambda>0\).

The remaining root degeneration is $e_{3,n}-e_{2,n}\to0$.
If this occurs without simultaneous collapse of the oscillation, then in the
limit one obtains a double root at the upper turning level:
\[
P(s)
=\frac{2\lambda}{R^{2}}(s-e_{1})(s-e_{2})^{2},
\qquad e_{1}<e_{2}.
\]
This is the separatrix degeneration of the vertical equation. Indeed, near the
double root \(e_{2}\) one has $P(s)=O((s-e_{2})^{2})$,
and therefore
\[
\int^{e_{2}}\frac{ds}{\sqrt{P(s)}}=+\infty.
\]
Thus the limiting vertical motion is of separatrix type rather than a regular
periodic oscillation. 

For branches with fixed vertical nodal number and bounded \(\lambda\), this pure
separatrix degeneration cannot occur. More precisely, it can occur only together
with collapse of the vertical oscillation.

\begin{lemma}
	\label{Le4.5}
	Let \(q\in\mathbb N\) and \(\Lambda>0\). Let $\{(\lambda_{n},\mathbf u_{n})\}_{n\in\N}\subset\mathcal Z_{\rm reg}$ be a sequence with vertical nodal number \(q\) and $0<\lambda_{n}\le\Lambda$. Let $e_{1,n}<e_{2,n}<e_{3,n}$
	be the corresponding roots. If $e_{3,n}-e_{2,n}\to0$,
	then
	\[
	e_{2,n}-e_{1,n}\to0.
	\]
\end{lemma}

\begin{proof}
	Set
	\[
	\mu_{n}:=\frac{e_{2,n}-e_{1,n}}{e_{3,n}-e_{1,n}}.
	\]
	For a regular oscillatory solution with vertical nodal number \(q\), the period
	condition gives
	\[
	\lambda_{n}(e_{3,n}-e_{1,n})=\frac{2R^{2}q^{2}}{\pi^{2}}K(\mu_{n})^{2}.
	\]
	From the proof of Proposition~\ref{Pr4.1}
	there exists \(M=M(R,\Lambda,q)>0\) such that
	\[
	\lambda_{n}(e_{3,n}-e_{1,n})\le M.
	\]
	Therefore $K(\mu_n)\le C(R,\Lambda,q)$.
	Since $K(\mu)\to+\infty$ as $\mu\to1^{-}$, there exists \(\mu_{*}<1\), depending only on \(R,\Lambda,q\), such that $\mu_n\le\mu_*$ for all \(n\) large enough. Now
	\[
	e_{3,n}-e_{2,n}=(1-\mu_n)(e_{3,n}-e_{1,n}).
	\]
	Since $1-\mu_n\ge1-\mu_*>0$ and $e_{3,n}-e_{2,n}\to0$, we obtain $e_{3,n}-e_{1,n}\to0$.
	Consequently, $0\le e_{2,n}-e_{1,n}\le e_{3,n}-e_{1,n}\to0$.
	This proves the claim.
\end{proof}

\begin{proposition}[Boundary of the regular class]
	\label{Pr4.3}
	Let $\{(\lambda_{n},\mathbf u_{n})\}_{n\in\mathbb N}
	\subset\mathcal Z_{\rm reg}$
	be a sequence such that
	\[
	\lim_{n\to+\infty}(\lambda_n,\mathbf u_n)=(\lambda,\mathbf u)
	\quad \text{in }(0,+\infty)\times H^{2}(\mathbb T,\mathbb R^{3}),
	\]
	with $(\lambda,\mathbf u)\in \mathcal Z\setminus\mathcal Z_{\rm reg}$.
	Let $e_{1,n}<e_{2,n}<e_{3,n}$
	be the roots associated with \((\lambda_{n},\mathbf u_{n})\). Then, after passing
	to a subsequence, at least one of the following degeneracies occurs:
	\begin{enumerate}
		\item[\rm(i)] \(e_{1,n}\to -R\). The sequence approaches loss of non-polarity
		through contact with the south pole, and in the limit \(C=-\lambda R\).
		
		\item[\rm(ii)] \(e_{2,n}\to R\). The sequence approaches loss of non-polarity
		through contact with the north pole, and in the limit \(C=\lambda R\).
		
		\item[\rm(iii)] \(e_{2,n}-e_{1,n}\to0\). The vertical oscillation collapses,
		and the limiting solution satisfies \(x_{3}\equiv u_{0}\)
		for some constant \(u_{0}\in[-R,R]\). If \(|u_0|<R\), then the limiting
		non-polar solution is a constant-latitude solution of the form
		\begin{equation}
			\label{Eq4.12}
		\mathbf x(\eta)=\big(\rho\cos(p\eta+\varphi),\rho\sin(p\eta+\varphi),u_{0}\big),
		\qquad \rho=\sqrt{R^2-u_0^2},
		\end{equation}
		with \(p\in\mathbb Z\setminus\{0\}\), and
		\[
		u_{0}=\frac{\lambda(p-1)}{p^{2}}.
		\]
		The limiting cases \(|u_0|=R\) are the polar constant solutions
		\[
		\mathbf x\equiv Re_3, \qquad \mathbf x\equiv -Re_3.
		\]
		
		\item[\rm(iv)] \(e_{3,n}-e_{2,n}\to0\). The vertical polynomial approaches a
		double-root degeneration at the upper turning level. If this degeneration
		is not accompanied by collapse of the vertical oscillation, then the limiting
		vertical dynamics is of separatrix type.
	\end{enumerate}
	Moreover, along a sequence with fixed vertical nodal number and bounded
	\(\lambda\), alternative \({\rm(iv)}\) can occur only together with
	alternative \({\rm(iii)}\).
	
	In particular, for a regular branch \(\mathscr C_k^{\rm reg}\) bifurcating
	from the equator, one has \(p=1\). Hence, if such a branch reaches the collapse
	alternative \({\rm(iii)}\) without simultaneous contact with the poles, then
	necessarily \(u_{0}=0\). Thus, for \(\mathscr C_k^{\rm reg}\) with bounded
	\(\lambda\), the effective boundary alternatives are contact with one of the
	poles, or collapse back to the equatorial branch.
\end{proposition}

\begin{proof}
	The fact that at least one of the four alternatives must occur follows from
	Lemma~\ref{Le4.4}. The interpretations of
	\({\rm(i)}\), \({\rm(ii)}\), and \({\rm(iii)}\) were established above. In
	particular, in the collapse case, if the limiting constant level satisfies
	\(|u_0|<R\), the constant-latitude classification gives \eqref{Eq4.12} with
	\[
	u_0=\frac{\lambda(p-1)}{p^2}.
	\]
	The polar cases \(|u_0|=R\) give the constant solutions
	\(\mathbf x\equiv \pm Re_3\).
	
	The interpretation of \({\rm(iv)}\) follows from the double-root form of the
	limiting polynomial \(P\). When this double-root degeneration is not accompanied
	by collapse, the limiting vertical equation has a separatrix-type trajectory.
	The fact that \({\rm(iv)}\) can occur only together with \({\rm(iii)}\), under
	bounded \(\lambda\) and fixed vertical nodal number, is precisely
	Lemma~\ref{Le4.5}.
	
	Finally, on the component \(\mathscr C_k^{\rm reg}\) one has \(p=1\) by
	\eqref{Eq4.9}. Therefore, in the non-polar collapse
	case \({\rm(iii)}\), the constant-latitude classification gives
	\[
	u_0=\frac{\lambda(p-1)}{p^2}=0.
	\]
	This proves the final assertion.
\end{proof}

After Theorem~\ref{Th4.1} and Proposition~\ref{Pr4.3}, the effective global
alternatives for a regular component can be summarized schematically as in
Figure~\ref{F3}. The figure is not a numerical bifurcation diagram, its purpose
is only to illustrate the possible behaviours in a finite-dimensional picture.
The horizontal axis represents the bifurcation parameter $\lambda=\frac{a^2}{\Omega}$,
whereas the vertical direction should be understood merely as a symbolic
amplitude coordinate in the infinite-dimensional space
\(H^2(\mathbb T,\mathbb R^3)\). 
The blue points on the horizontal axis indicate bifurcation points on the
trivial equatorial branch, namely $\lambda_k=R\sqrt{k^2-1}$, $k\ge2$.
Starting from such a point, the local Crandall--Rabinowitz branch enters the
regular non-polar class. Its connected component
\(\mathscr C_k^{\mathrm{reg}}\) may then behave in, at least, one of the following ways.
It may contain a sequence with \(\lambda\to+\infty\). It may approach the
extended boundary through the limiting regime \(\lambda\to0\). It may approach a
limit profile with pole contact. Finally, if the component remains compact and
does not approach the extended boundary, then the only possible loss of
regularity is the return to the original bifurcation point
\((\lambda_k,0)\), corresponding to collapse of the vertical oscillation onto
the equator. Thus, in the schematic picture, the four effective possibilities are: escape to
\(\lambda=+\infty\), return to the bifurcation point, approach to
\(\lambda=0\), or pole contact. The double-root degeneration appearing in
Proposition~\ref{Pr4.3} is not drawn as a separate global alternative here,
because, along a fixed regular component with fixed vertical number, it is
absorbed by the vertical-collapse scenario.

\begin{figure}[h!]
	\centering

	\tikzset{every picture/.style={line width=0.75pt}} 
	
	\begin{tikzpicture}[x=0.7pt,y=0.7pt,yscale=-1,xscale=1]
		
		\draw [color={rgb, 255:red, 74; green, 144; blue, 226 }  ,draw opacity=1 ][line width=1.5]    (86.51,218.83) -- (390.72,219.55) ;
		\draw    (86.86,30.42) -- (86.51,218.83) ;
		\draw [color={rgb, 255:red, 184; green, 233; blue, 134 }  ,draw opacity=1 ][line width=2.25]    (150,219) .. controls (152,163) and (188,131) .. (210,126) ;
		\draw [color={rgb, 255:red, 245; green, 166; blue, 35 }  ,draw opacity=1 ][line width=2.25]    (47,104) .. controls (108,116) and (337,41) .. (210,126) ;
		\draw [color={rgb, 255:red, 245; green, 166; blue, 35 }  ,draw opacity=1 ][line width=2.25]    (210,126) .. controls (330,80) and (228,33) .. (308,34) ;
		\draw [color={rgb, 255:red, 184; green, 233; blue, 134 }  ,draw opacity=1 ][line width=2.25]    (277,219) .. controls (279,163) and (360,149) .. (382,144) ;
		\draw  [color={rgb, 255:red, 74; green, 144; blue, 226 }  ,draw opacity=1 ][fill={rgb, 255:red, 74; green, 144; blue, 226 }  ,fill opacity=1 ][line width=2.25]  (274.5,219) .. controls (274.5,217.62) and (275.62,216.5) .. (277,216.5) .. controls (278.38,216.5) and (279.5,217.62) .. (279.5,219) .. controls (279.5,220.38) and (278.38,221.5) .. (277,221.5) .. controls (275.62,221.5) and (274.5,220.38) .. (274.5,219) -- cycle ;
		\draw [color={rgb, 255:red, 245; green, 166; blue, 35 }  ,draw opacity=1 ][line width=2.25]    (207.5,126) .. controls (324,102) and (277,177) .. (242,257) ;
		\draw    (86.51,218.83) -- (9,319) ;
		\draw [color={rgb, 255:red, 245; green, 166; blue, 35 }  ,draw opacity=1 ][line width=2.25]  [dash pattern={on 2.53pt off 3.02pt}]  (242,257) .. controls (200,353) and (143,300) .. (150,219) ;
		\draw  [color={rgb, 255:red, 74; green, 144; blue, 226 }  ,draw opacity=1 ][fill={rgb, 255:red, 74; green, 144; blue, 226 }  ,fill opacity=1 ][line width=2.25]  (147.5,219) .. controls (147.5,217.62) and (148.62,216.5) .. (150,216.5) .. controls (151.38,216.5) and (152.5,217.62) .. (152.5,219) .. controls (152.5,220.38) and (151.38,221.5) .. (150,221.5) .. controls (148.62,221.5) and (147.5,220.38) .. (147.5,219) -- cycle ;
		\draw  [color={rgb, 255:red, 245; green, 166; blue, 35 }  ,draw opacity=1 ][fill={rgb, 255:red, 74; green, 144; blue, 226 }  ,fill opacity=1 ][line width=2.25]  (42,104) .. controls (42,102.62) and (43.12,101.5) .. (44.5,101.5) .. controls (45.88,101.5) and (47,102.62) .. (47,104) .. controls (47,105.38) and (45.88,106.5) .. (44.5,106.5) .. controls (43.12,106.5) and (42,105.38) .. (42,104) -- cycle ;
		\draw  [color={rgb, 255:red, 245; green, 166; blue, 35 }  ,draw opacity=1 ][fill={rgb, 255:red, 74; green, 144; blue, 226 }  ,fill opacity=1 ][line width=2.25]  (308,34) .. controls (308,32.62) and (309.12,31.5) .. (310.5,31.5) .. controls (311.88,31.5) and (313,32.62) .. (313,34) .. controls (313,35.38) and (311.88,36.5) .. (310.5,36.5) .. controls (309.12,36.5) and (308,35.38) .. (308,34) -- cycle ;
		\draw [color={rgb, 255:red, 245; green, 166; blue, 35 }  ,draw opacity=1 ][line width=2.25]    (210,126) .. controls (300.09,102.24) and (309.81,75.54) .. (431.28,72.1) ;
		\draw [shift={(435,72)}, rotate = 178.63] [color={rgb, 255:red, 245; green, 166; blue, 35 }  ,draw opacity=1 ][line width=2.25]    (17.49,-5.26) .. controls (11.12,-2.23) and (5.29,-0.48) .. (0,0) .. controls (5.29,0.48) and (11.12,2.23) .. (17.49,5.26)   ;
		
		\draw (96,14.4) node [anchor=north west][inner sep=0.75pt]    {$H^{2}\left(\mathbb{T} ,\mathbb{R}^{3}\right)$};
		\draw (398,205.4) node [anchor=north west][inner sep=0.75pt]  [color={rgb, 255:red, 74; green, 144; blue, 226 }  ,opacity=1 ]  {$\lambda =\frac{a^{2}}{\Omega }$};
		\draw (152,222.4) node [anchor=north west][inner sep=0.75pt]    {$\lambda _{2} =R\sqrt{3}$};
		\draw (279,222.4) node [anchor=north west][inner sep=0.75pt]    {$\lambda _{3} =2R\sqrt{2}$};
		\draw (116,171.4) node [anchor=north west][inner sep=0.75pt]    {$\mathscr{C}_{2}^{\rm{reg}}$};
		\draw (321,25) node [anchor=north west][inner sep=0.75pt]   [align=left] {pole contact};
		\draw (20,80.4) node [anchor=north west][inner sep=0.75pt]    {$\lambda =0$};
		\draw (437,48.4) node [anchor=north west][inner sep=0.75pt]    {$\lambda \rightarrow +\infty $};
		\draw (133,316) node [anchor=north west][inner sep=0.75pt]   [align=left] {collapse vertical oscillation};

	\end{tikzpicture}
	\caption{Schematic global alternative for the regular component $\mathscr{C}^{\rm{reg}}_{2}$.}
	\label{F3}
\end{figure}
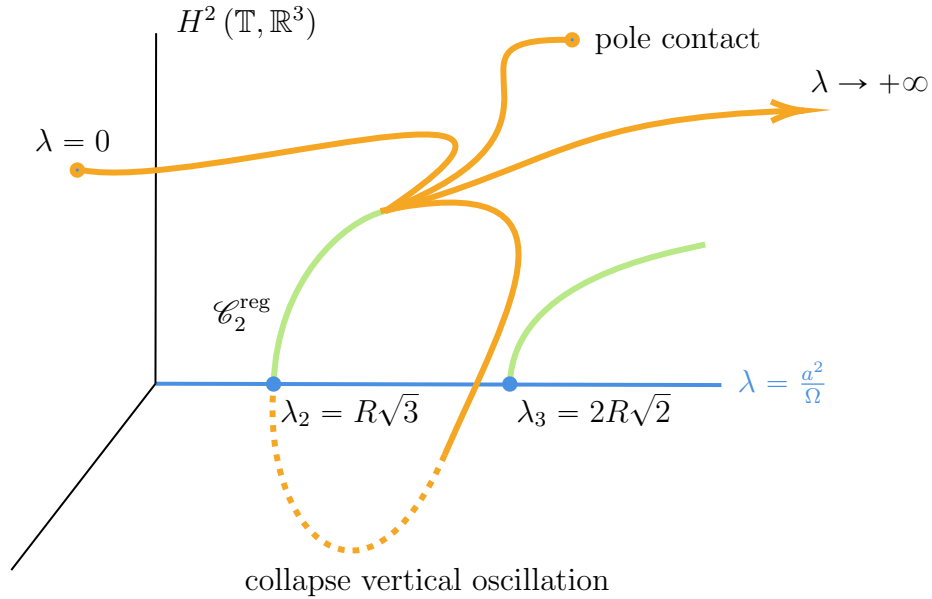

The possible boundary degenerations described in Proposition~\ref{Pr4.3} are
schematically represented in Figure~\ref{F2}. In contrast with
Figure~\ref{F1}, which depicts the global continuation alternative, this figure
only illustrates the mechanisms by which a sequence of regular profiles may
approach the boundary of the regular class at bounded parameter.

\begin{figure}[h!]
	\centering
	\includegraphics[scale=0.2]{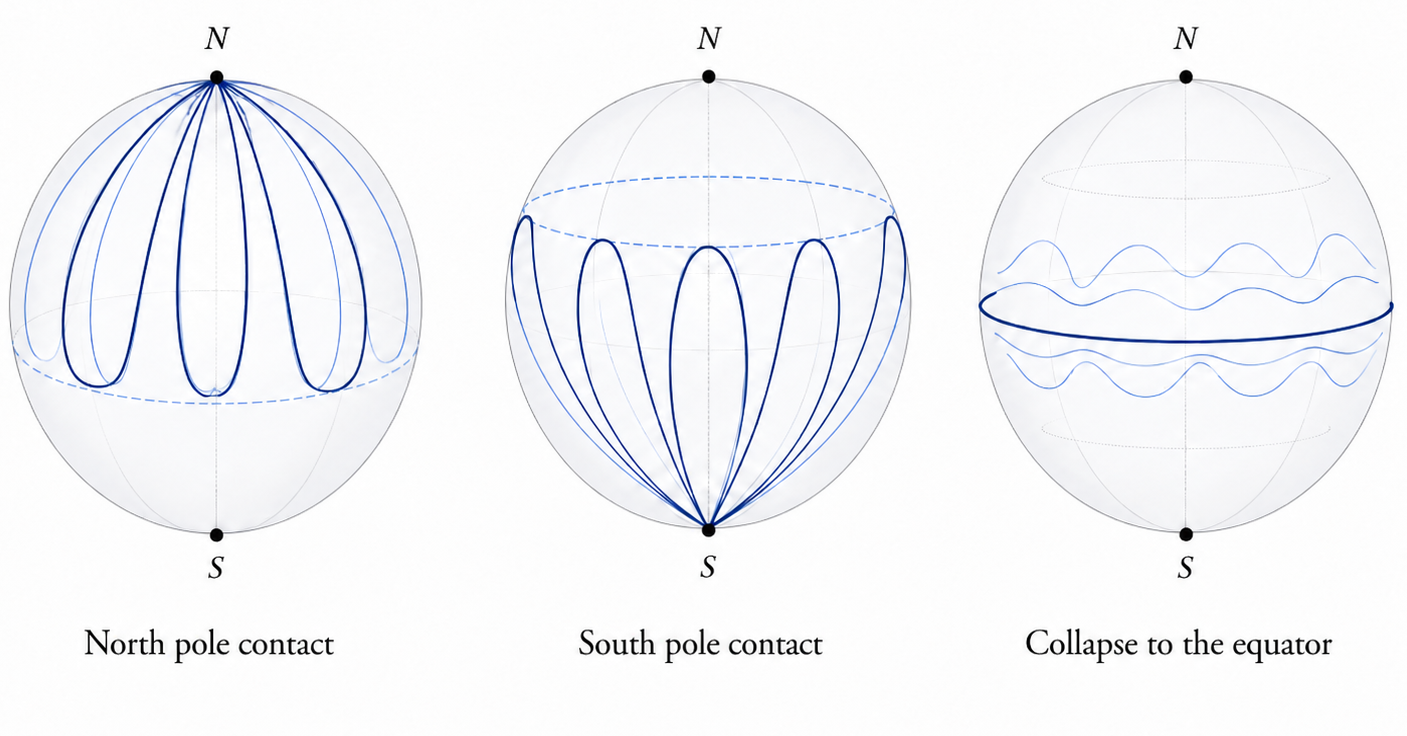}
	\caption{Schematic representation of the possible boundary degenerations of the regular class.}
	\label{F2}
\end{figure}

\subsection{Geometric interpretation of polar contact}

We record one final consequence of the boundary analysis above. In Proposition~\ref{Pr4.3}, loss of non-polarity may occur through one of the two polar regimes $e_{1,n}\to -R$ or $e_{2,n}\to R$. The following result explains what this means geometrically for the limiting spherical profile. The point is that, along a regular sequence, the vertical nodal number is fixed. Thus, if a sequence with vertical nodal number \(q\) converges to a non-degenerate polar-contact profile, then the contact with the pole occurs once during each vertical oscillation. In particular, the limiting profile does not merely touch the pole at a single parameter value: it reaches the pole exactly \(q\) times over one period.

\begin{proposition}[Multiple polar contact]
	\label{Pr}
	Let \(q\geq1\), and let $(\lambda_n,\mathbf u_n)\in\mathcal Z_{\rm reg}$, $\mathbf x_n:=\mathbf x_0+\mathbf u_n$,
	be a sequence of regular profiles with vertical nodal number \(q\). Assume that
	\[
	\lim_{n\to+\infty}(\lambda_n,\mathbf u_n)=(\lambda,\mathbf u)
	\quad\text{in }(0,+\infty)\times H^2(\mathbb T,\mathbb R^3),
	\]
	with \(\lambda>0\), and write $\mathbf x:=\mathbf x_0+\mathbf u$.
	For each \(n\), let $e_{1,n}<e_{2,n}<e_{3,n}$
	be the regular root data associated with \(\mathbf x_n\).
	
	Assume first that, after passing to a subsequence,
	\[
	e_{1,n}\to e_1,\qquad e_{2,n}\to R,\qquad e_{3,n}\to e_3,
	\]
	with $-R<e_1<R<e_3$. Then there exist exactly \(q\) distinct points $\eta_0^+,\ldots,\eta_{q-1}^+\in\mathbb T$
	such that
	\[
	\mathbf x(\eta_j^+)=R\mathbf e_3, \qquad j=0,\ldots,q-1.
	\]
	
	Similarly, assume that, after passing to a subsequence,
	\[
	e_{1,n}\to -R,\qquad
	e_{2,n}\to e_2,\qquad
	e_{3,n}\to e_3,
	\]
	with $-R<e_2<R<e_3$. Then there exist exactly \(q\) distinct points $\eta_0^-,\ldots,\eta_{q-1}^-\in\mathbb T$
	such that
	\[
	\mathbf x(\eta_j^-)=-R\mathbf e_3, \qquad j=0,\ldots,q-1.
	\]
	In particular, if \(q\geq2\), the limiting spherical profile has a multiple
	self-intersection at the corresponding pole.
\end{proposition}

\begin{proof}
	We prove the north-pole case since the south-pole case is analogous. For each \(n\), since \(\mathbf x_n\in\mathcal Z_{\rm reg}\) and has vertical
	nodal number \(q\), its vertical component $u_n=(x_n)_3$
	has the elliptic representation
	\[
	u_n(\eta)
	=
	e_{1,n}
	+
	(e_{2,n}-e_{1,n})
	\sn^2\big(\alpha_n(\eta-\eta_n)\mid m_n\big),
	\]
	for some phase \(\eta_n\in\mathbb T\), where
	\[
	m_n=\frac{e_{2,n}-e_{1,n}}{e_{3,n}-e_{1,n}},
	\qquad
	\frac{2qK(m_n)}{\alpha_n}=2\pi.
	\]
	Passing to a subsequence, we may assume that \(\eta_n\to\eta_*\) in
	\(\mathbb T\). Under the north-pole assumptions,
	\[
	\lim_{n\to+\infty}m_n = m:=\frac{R-e_1}{e_3-e_1}\in(0,1),
	\]
	and therefore
	\[
	\alpha_n=\frac{qK(m_n)}{\pi}
	\to
	\alpha:=\frac{qK(m)}{\pi}.
	\]
	Since \(u_n\to u=x_3\) in \(\mc{C}^1(\mathbb T)\), passing to the limit in the
	elliptic representation gives
	\[
	u(\eta)=e_1+(R-e_1)\sn^2\big(\alpha(\eta-\eta_*)\mid m\big).
	\]
	The maxima of \(\sn^2(s\mid m)\) occur precisely at $s=(2j+1)K(m)$, $j\in\mathbb Z$.
	Hence the maxima of \(u\) in one period occur at
	\[
	\eta_j^+=\eta_*+\frac{(2j+1)K(m)}{\alpha}\quad\operatorname{mod}2\pi, \qquad j=0,\ldots,q-1.
	\]
	These points are distinct modulo \(2\pi\), because consecutive maxima are
	separated by
	\[
	\frac{2K(m)}{\alpha}=\frac{2\pi}{q}.
	\]
	At each of these points, $u(\eta_j^+)=R$.
	Since \(|\mathbf x(\eta)|=R\), the identity \(x_3(\eta_j^+)=R\) forces
	\[
	x_1(\eta_j^+)=x_2(\eta_j^+)=0.
	\]
	Therefore $\mathbf x(\eta_j^+)=R\mathbf e_3$, $j=0,\ldots,q-1$.
	
	Conversely, if \(\mathbf x(\eta)=R\mathbf e_3\), then \(u(\eta)=R\), so
	\(\eta\) is a maximum point. The listed points
	are exactly the \(q\) maxima in one period. Thus the limiting profile has
	exactly \(q\) north-pole contacts.
\end{proof}

\section{Numerical continuation of the closure system}
\label{SN}

The purpose of this section is to complement the global information obtained
in Section~\ref{S4} with a numerical continuation of the explicit closure system
\(\mathscr S_{p,q}\). The role of the computation is not to replace the global
alternative of Theorem~\ref{Th4.1}, but to indicate which of the possible
boundary mechanisms is selected by the regular components bifurcating from the
equatorial branch.

Recall that, on the regular component \(\mathscr C_k^{\mathrm{reg}}\) obtained
from the local bifurcation at \(\lambda_k=R\sqrt{k^2-1}\), the two discrete
invariants satisfy $p=1$, $q=k$.
Thus, modulo the two elementary phase parameters in \(\eta\) and in the
azimuthal angle, the profiles in \(\mathscr C_k^{\mathrm{reg}}\) are represented
inside the solution set of the finite-dimensional system \(\mathscr S_{1,k}\)
introduced in \eqref{Sis}. The numerical problem considered below is precisely
to continue the branch of \(\mathscr S_{1,k}\) which starts at the equatorial
root configuration.

The computations consistently indicate that this branch reaches the polar
boundary \(e_2=R\) at a finite value of \(\lambda\). Hence the numerical
evidence selects the north-pole-contact alternative in Proposition~\ref{Pr4.3}.
By Proposition~\ref{Pr}, the corresponding limiting
profile has exactly \(k\) contacts with the north pole.

For the numerical implementation we work in the normalized case $R=1$.
This is the convention used in the figures below. The general case is recovered
by the natural scaling of the root variables and of the parameter. With this
normalization, the equatorial bifurcation values are $\lambda_q=\sqrt{q^2-1}$, $q=2,3,\ldots$,
and the corresponding systems are $\mathscr S_{1,q}$.

\subsection{Algebraic reduction of system $\mathscr{S}_{1,q}$} We start from the system \(\mathscr S_{1,q}\) in \eqref{Sis}, specialized to
\(R=1\). The unknowns are $(e_1,e_2,e_3,C,\lambda)$,
subject to $\lambda>0$, $-1<e_1<e_2<1$ and $e_2<e_3$.
The vertical component \(x_3\) oscillates between the two turning levels
\(e_1\) and \(e_2\), while \(e_3\) is the third root of the vertical polynomial. For \(R=1\), the first two equations of \(\mathscr S_{1,q}\) are
\[
(C-\lambda)^2=2\lambda(1-e_1)(1-e_2)(e_3-1), \quad (C+\lambda)^2=2\lambda(1+e_1)(1+e_2)(1+e_3).
\]
These two identities determine the algebraic part of the system up to a choice
of signs. The branch issuing from the equatorial root configuration determines the sign
choice in the first two algebraic equations. Indeed, for the equatorial profile
one has $C=1$. Hence
\[
C-\lambda_q=1-\lambda_q<0,
\qquad
C+\lambda_q=1+\lambda_q>0.
\]
We therefore select, near the equatorial root configuration, the sign choice
\[
C-\lambda<0,
\qquad
C+\lambda>0.
\]
Set $A=(1-e_1)(1-e_2)(e_3-1)$ and $B=(1+e_1)(1+e_2)(1+e_3)$.
On the branch under consideration we therefore write
\[
C-\lambda=-\sqrt{2\lambda A},
\qquad
C+\lambda=\sqrt{2\lambda B}.
\]
Thus, along the branch selected by the equatorial sign choice, the quantities
\(C\) and \(\lambda\) are explicit functions of the three roots
\((e_1,e_2,e_3)\):
\begin{equation}
	\label{EqN1}
\lambda=\frac{(\sqrt A+\sqrt B)^2}{2}, \qquad C=\frac{B-A}{2}.
\end{equation}
Now, we use the upper turning level $e_2=\max_{\eta\in\mathbb T}x_3(\eta)$ as continuation parameter. Thus, for fixed \(q\) and fixed \(e_2\in(0,1)\), the remaining unknowns are $(e_1,e_3)$. The inequalities to be preserved are $-1<e_1<e_2<1$, $1<e_3$. For each candidate pair \((e_1,e_3)\), we compute $\l$ and $C$ in \eqref{EqN1}.
Setting $m=\frac{e_2-e_1}{e_3-e_1}$,
the vertical component has the Jacobi representation
\[
x_3(\eta)=e_1+(e_2-e_1)\sn^2(\alpha\eta\mid m),
\]
where $\alpha^2=\frac{\lambda}{2}(e_3-e_1)$.
Since \(\sn^2(\,\cdot\,\mid m)\) has period \(2K(m)\), the condition that the
vertical component performs exactly \(q\) oscillations over the interval
\([0,2\pi]\) is
\[
\frac{2qK(m)}{\alpha}=2\pi.
\]
Equivalently, $\alpha=\frac{qK(m)}{\pi}$.
Combining this identity with $\alpha^2=\frac{\lambda}{2}(e_3-e_1)$,
we obtain the first reduced equation
\[
F_1(e_1,e_3;e_2,q)=0, \quad F_1(e_1,e_3;e_2,q)
:=\lambda(e_3-e_1)-\frac{2q^2K(m)^2}{\pi^2}.
\]

The second reduced equation comes from the azimuthal closure condition. For
horizontal degree \(p=1\), the total azimuthal increment over one full period
is \(2\pi\). Since the vertical motion has \(q\) oscillations, the azimuthal
increment over one vertical oscillation must be $\frac{2\pi}{q}$.
Using the formula
\[
\phi'(\eta)=\frac{C-\lambda x_3(\eta)}{1-x_3(\eta)^2},
\]
the azimuthal increment over one vertical oscillation is
\[
\frac{1}{\alpha}\left[\frac{C-\lambda}{1-e_1}\Pi\left(\frac{e_2-e_1}{1-e_1}\,\middle|\,m\right)+\frac{C+\lambda}{1+e_1}\Pi\left(-\frac{e_2-e_1}{1+e_1}\,\middle|\,m\right)\right].
\]
Therefore the second reduced equation is
\[
F_2(e_1,e_3;e_2,q)=0,
\]
\[
F_2(e_1,e_3;e_2,q):=\frac{1}{\alpha}\left[\frac{C-\lambda}{1-e_1}\Pi\left(\frac{e_2-e_1}{1-e_1}\,\middle|\,m\right)+\frac{C+\lambda}{1+e_1}\Pi\left(-\frac{e_2-e_1}{1+e_1}\,\middle|\, m\right)\right]-\frac{2\pi}{q}.
\]
Here \(C\), \(\lambda\), \(m\), and \(\alpha\) are understood as functions of
\((e_1,e_2,e_3)\), with \(e_2\) fixed.

Thus the actual system solved at each continuation step is the two-dimensional
system
\begin{equation}
	\label{EqN2}
	\begin{cases}
		F_1(e_1,e_3;e_2,q)=0,\\
		F_2(e_1,e_3;e_2,q)=0.
	\end{cases}
\end{equation}

\subsection{Initialization near the equatorial root configuration}

The equatorial profile satisfies $x_3\equiv0$.
In terms of the root variables this corresponds to the limiting configuration $e_1=e_2=0$.
At this degenerate configuration the vertical oscillation has zero amplitude.
Nevertheless, the bifurcating branches are labelled by the integer \(q\ge2\),
which becomes the vertical nodal number once the amplitude is positive. The
corresponding bifurcation value is $\lambda_q=\sqrt{q^2-1}$. We now determine the limiting value of the third root. At the equatorial root
configuration one has $m=\frac{e_2-e_1}{e_3-e_1}=0$ and $K(0)=\frac{\pi}{2}$.
The period equation
\[
\lambda(e_3-e_1)=
\frac{2q^2K(m)^2}{\pi^2}
\]
therefore gives, at \(e_1=e_2=0\), $\lambda_q e_3^0=\frac{q^2}{2}$.
Hence
\[
e_3^0=\frac{q^2}{2\lambda_q}=\frac{q^2}{2\sqrt{q^2-1}}.
\]
Equivalently, since \(q^2=1+\lambda_q^2\), this can be written as
\[
e_3^0
=
\frac{1+\lambda_q^2}{2\lambda_q}.
\]

The numerical continuation cannot be initialized exactly at
\(e_1=e_2=0\), because the root parametrization is degenerate there. We
therefore start at a small positive value $e_2=\varepsilon>0$.
The first Newton seed is chosen as $e_1=-\varepsilon$, $e_3=e_3^0$.
This values are only used as an
initial guess, the actual values of \(e_1\) and \(e_3\) are then determined by
solving the reduced system \eqref{EqN2}.

After the first point has been computed, the solution obtained at
\(e_2=e_2^{(j)}\) is used as the initial seed for the next value
\(e_2=e_2^{(j+1)}\). In this way the branch is followed along an increasing
mesh
\[
0<e_2^{(0)}<e_2^{(1)}<\cdots<e_2^{(N)}<1.
\]
Once a solution \((e_1,e_3)\) has been computed for a fixed \(e_2\), we recover $\lambda$, $C$, $m$, $\alpha$ and the corresponding profile.

\subsection{Computation of the north-pole endpoint}

The reconstruction of the full spherical profile becomes numerically delicate
as \(e_2\to1\). Indeed, at north-pole contact one has $x_3=1$ at the contact points, and the azimuthal coordinate degenerates. Therefore the endpoint is not computed by extrapolating the
reconstructed profile. Instead, we pass to the limiting form of the reduced
closure system $\mathscr S_{1,q}$.

At the north-pole boundary \(e_2=1\), the quantities \(A\) and \(B\) become $A=0$ and $B=2(1+e_1)(1+e_3)$.
Hence
\[
\lambda=\frac{B}{2}=(1+e_1)(1+e_3),
\qquad
C=\frac{B}{2}=\lambda.
\]
Moreover, $m=\frac{1-e_1}{e_3-e_1}$ and he period equation becomes
\[
(1+e_1)(1+e_3)(e_3-e_1)=\frac{2q^2K(m)^2}{\pi^2}.
\]

The limiting azimuthal equation is obtained by taking the limit \(e_2\to1\) in
the fourth equation of \(\mathscr S_{1,q}\). The first \(\Pi\)-term is singular
in this limit, but its full contribution has a finite limit. The polar
contribution is $-\pi$.
Consequently the limiting azimuthal equation is
\[
-\pi
+
\frac{2(1+e_3)}{\alpha}
\Pi\left(
-\frac{1-e_1}{1+e_1}
\,\middle|\,
m
\right)
=
\frac{2\pi}{q},
\qquad
\alpha=\frac{qK(m)}{\pi}.
\]
Thus the endpoint is computed by solving the two equations
\begin{equation}
	\label{EqN4}
	\begin{cases}
		(1+e_1)(1+e_3)(e_3-e_1)
		=
		\dfrac{2q^2K(m)^2}{\pi^2},
		\\[1.2em]
		-\pi
		+
		\dfrac{2(1+e_3)}{\alpha}
		\Pi\left(
		-\dfrac{1-e_1}{1+e_1}
		\,\middle|\,
		m
		\right)
		=
		\dfrac{2\pi}{q},
	\end{cases}
\end{equation}
with
\[
m=\frac{1-e_1}{e_3-e_1},
\qquad
\alpha=\frac{qK(m)}{\pi}.
\]
The corresponding endpoint values are denoted by
\[
e_1^\ast,
\qquad
e_2^\ast=1,
\qquad
e_3^\ast,
\qquad
\lambda_q^\ast=(1+e_1^\ast)(1+e_3^\ast).
\]

\subsection{Numerical outcome}

We now describe the numerical outcome of the continuation procedure. For every
\(q=2,\ldots,10\) computed, the branch of \(\mathscr S_{1,q}\) issuing from the
equatorial root configuration approaches the north-pole boundary $e_2=1$
at a finite value of \(\lambda\). In the computed range, we found no numerical
evidence for escape to \(\lambda=+\infty\), approach to \(\lambda=0\), return
to the equator away from the initial bifurcation point, or contact with the
south pole. Thus the computation consistently selects the north-pole-contact
alternative among the boundary mechanisms described in Proposition~\ref{Pr4.3}.

The endpoint values obtained from the limiting polar system are
\[
\begin{array}{c|c|c|c|c|c}
	q & \lambda_q=\sqrt{q^2-1} & \lambda_q^\ast & e_1^\ast & e_2^\ast & e_3^\ast \\ \hline
	2 & 1.732051 & 3.506561 & 0.488013 & 1 & 1.356539\\
	3 & 2.828427 & 4.491194 & 0.595482 & 1 & 1.814945\\
	4 & 3.872983 & 5.488063 & 0.664474 & 1 & 2.297176\\
	5 & 4.898979 & 6.487742 & 0.713053 & 1 & 2.787241\\
	6 & 5.916080 & 7.488215 & 0.749237 & 1 & 3.280847\\
	7 & 6.928203 & 8.488904 & 0.777272 & 1 & 3.776368\\
	8 & 7.937254 & 9.489621 & 0.799647 & 1 & 4.273045\\
	9 & 8.944272 & 10.490301 & 0.817926 & 1 & 4.770477\\
	10 & 9.949874 & 11.490924 & 0.833142 & 1 & 5.268431
\end{array}
\]
The column \(e_2^\ast\) has been included in order to make explicit that the
computed endpoint lies on the north-pole boundary. In particular, $e_2^\ast=1$, $e_1^\ast>0$.
Hence, at the limiting profile, $e_1^\ast\le x_3(\eta)\le e_2^\ast=1$.
Thus the vertical oscillation is entirely contained in the northern hemisphere
and reaches the north pole.

Figure~\ref{F4} shows the numerical continuation in the
\((\lambda,\|\mathbf x\|_{H^2})\)-plane. The plotted curves are stopped slightly before
\(e_2=1\), whereas the endpoint values in the table above are obtained from the
limiting system \eqref{EqN4}.

The figure shows that the full \(H^2\)-norm increases along the computed
branches. This is consistent with the geometric picture: as the branch
approaches \(e_2=1\), the spherical tangent profile develops polar contacts,
and the resulting concentration is detected by the \(H^2\)-norm. The growth is
more pronounced for larger values of \(q\), as expected from the increasing
number of vertical oscillations.

We stress that Figure~\ref{F4} is a two-dimensional projection of
the continuation data. In particular, the apparent proximity of the first two
curves in the figure does not indicate an intersection of the corresponding
branches. The branches with \(q=2\) and \(q=3\) issue from different bifurcation
values, $\lambda_2=\sqrt3$, $\lambda_3=\sqrt8$,
and, in the regular regime, they carry different vertical nodal numbers. Since
the vertical nodal number is constant along regular components, the two
branches are distinct in the solution space. The apparent closeness is only a
visual effect of plotting the data in the projected plane
\((\lambda,\|\mathbf x\|_{H^2})\).

By Proposition~\ref{Pr}, the limiting profile
associated with the branch \(\mathscr S_{1,q}\) has exactly \(q\) contacts with
the north pole. Therefore, for the component \(\mathscr C_k^{\mathrm{reg}}\),
where \(q=k\), the numerical endpoint has \(k\) north-pole contacts.

\begin{figure}[h!]
	 \centering
	 \includegraphics[width=0.72\textwidth]{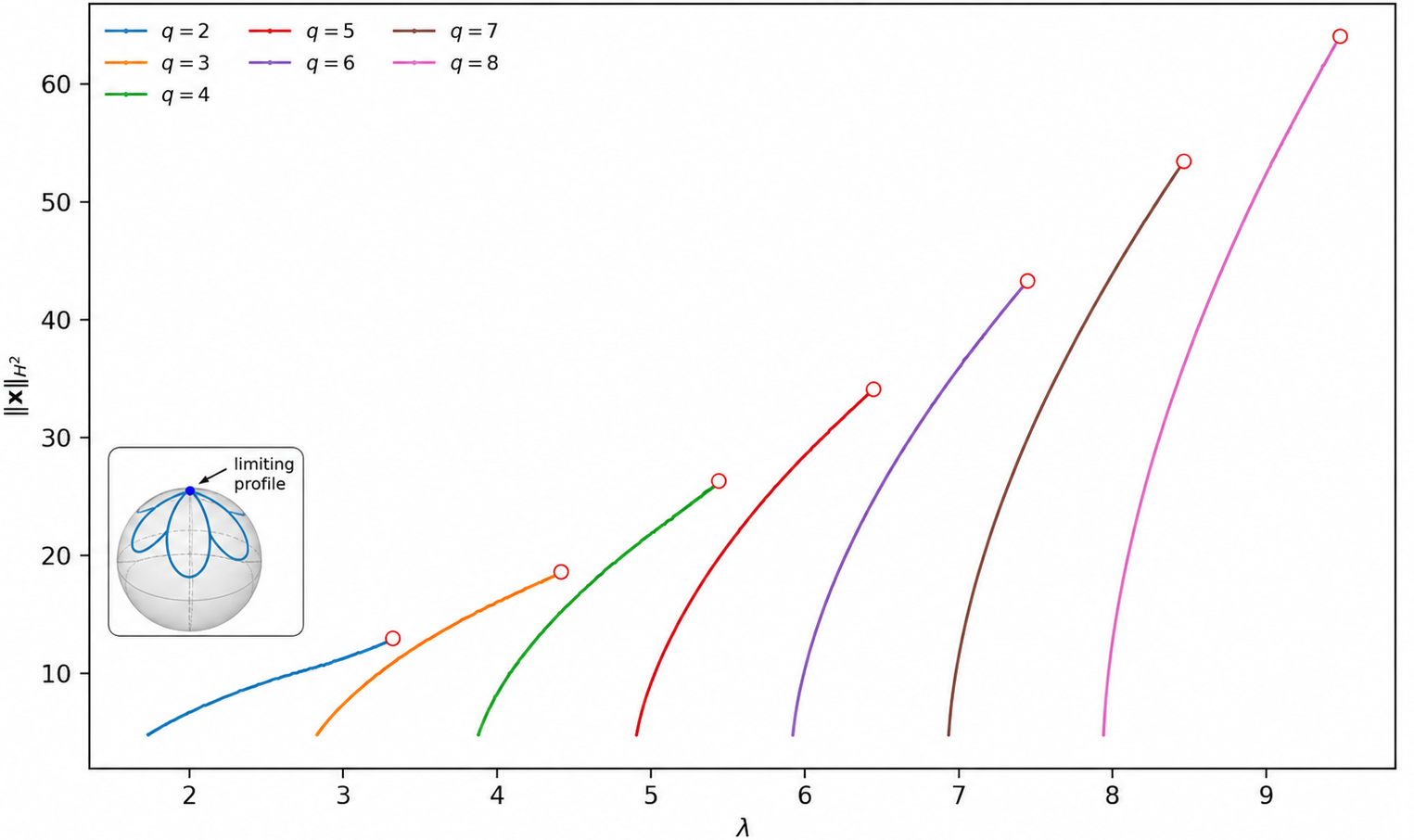} \caption{Numerical continuation of the branches $\mathscr S_{1,q}$, $q=2,\ldots,10$, issuing from the equatorial root configuration, plotted in the \((\lambda,\|\mathbf x\|_{H^2})\)-plane.}
	\label{F4} 
\end{figure}

\section{Relation with the Vortex Filament Equation}
\label{S5}

We close the paper by explaining how the travelling--rotating Schrödinger maps
studied above arise as tangent indicatrices of rigid travelling--rotating
solutions of the Vortex Filament Equation (VFE). Throughout this section we restrict
to the unit-sphere case \(R=1\), so that the profile \(\mathbf x\) can be
identified with the unit tangent field of an arclength-parametrized filament. Recall the Vortex Filament Equation, or binormal flow,
\begin{equation}
	\label{Eq5.1}
	\mathbf X_t=\mathbf X_s\times \mathbf X_{ss}, \qquad	(t,s)\in\mathbb R^2.
\end{equation}
If \(\mathbf X\) is parametrized by arclength, namely $\|\mathbf X_s\|\equiv1$,
then its unit tangent field
\[
\mathbf T(t,s):=\mathbf X_s(t,s)\in\mathbb S^2
\]
solves the Schrödinger map equation into \(\mathbb S^2\).

We compare the travelling--rotating Schrödinger map ansatz
\begin{equation}
	\label{Eq5.2}
	\mathbf T(t,s)=\mathscr R^{\Omega t}\mathbf x(\xi), \qquad \xi=s-at,
\end{equation}
with rigid travelling--rotating solutions of the Vortex Filament Equation.

\subsection{Rigid travelling--rotating filaments and tangent profiles}

Consider the VFE ansatz
\begin{equation}
	\label{Eq5.3}
	\mathbf X(t,s)=	\mathscr R^{\Omega t}\mathbf X_0(\xi)+\mathbf b(t), \qquad	\xi=s-at,
\end{equation}
where \(\mathbf X_0:\mathbb R\to\mathbb R^3\) is a profile curve and
\(\mathbf b:\mathbb R\to\mathbb R^3\) is a time-dependent translation. Then
\[
\mathbf X_s(t,s)=\mathscr R^{\Omega t}\mathbf X_0'(\xi).
\]
Therefore \eqref{Eq5.3} produces a tangent field of the form
\eqref{Eq5.2} if and only if $\mathbf X_0'(\xi)=\mathbf x(\xi)$.

The following proposition makes precise the equivalence between the tangent
profile equation and the rigid-motion profile equation for the filament.

\begin{proposition}
	Let \(\Omega,a\in\mathbb R\). Let $\mathbf x:\mathbb R\to\mathbb S^2$
	be a smooth solution of
	\begin{equation}
		\label{Eq5.4}
		\mathbf x\times\mathbf x''-\Omega e_3\times\mathbf x+a\mathbf x'=0.
	\end{equation}
	Let \(\mathbf X_0\) be any primitive of \(\mathbf x\), namely $\mathbf X_0'=\mathbf x$.
	Then the quantity
	\begin{equation*}
		\mathbf d:=\mathbf x\times\mathbf x'-\Omega e_3\times\mathbf X_0+a\mathbf x
	\end{equation*}
	is independent of \(\xi\). If \(\mathbf b\) solves $\mathbf b'(t)=\mathscr R^{\Omega t}\mathbf d$, then
	\begin{equation*}
		\mathbf X(t,s)=\mathscr R^{\Omega t}\mathbf X_0(s-at)+\mathbf b(t)
	\end{equation*}
	solves the Vortex Filament Equation \eqref{Eq5.1}, and its tangent field is
	\[
	\mathbf X_s(t,s)=\mathscr R^{\Omega t}\mathbf x(s-at).
	\]
	
	Conversely, if a filament of the form \eqref{Eq5.3} solves
	\eqref{Eq5.1} and satisfies \(\mathbf X_0'=\mathbf x\), then there exists a
	constant vector \(\mathbf d\in\mathbb R^3\) such that
	\begin{equation}
		\label{Eq5.5}
		\mathbf x\times\mathbf x'-\Omega e_3\times\mathbf X_0+a\mathbf x=\mathbf d,\qquad\mathbf X_0'=\mathbf x,
	\end{equation}
	and the tangent profile \(\mathbf x\) satisfies
	\eqref{Eq5.4}.
\end{proposition}

\begin{proof}
	Assume first that \(\mathbf x\) solves
	\eqref{Eq5.4} and that \(\mathbf X_0'=\mathbf x\). Define
	\[
	\mathbf d(\xi)=\mathbf x(\xi)\times\mathbf x'(\xi)-\Omega e_3\times\mathbf X_0(\xi)+a\mathbf x(\xi).
	\]
	Then $\mathbf d'=(\mathbf x\times\mathbf x')'-\Omega e_3\times\mathbf X_0'+a\mathbf x'$.
	Since $(\mathbf x\times\mathbf x')'=\mathbf x\times\mathbf x''$ and \(\mathbf X_0'=\mathbf x\), we get
	\[
	\mathbf d'=\mathbf x\times\mathbf x''-\Omega e_3\times\mathbf x+a\mathbf x'.
	\]
	By \eqref{Eq5.4}, $\mathbf d'=0$.
	Hence \(\mathbf d\) is constant. Let \(\mathbf b\) solve $\mathbf b'(t)=\mathscr R^{\Omega t}\mathbf d$.
	Define
	\[
	\mathbf X(t,s)=\mathscr R^{\Omega t}\mathbf X_0(\xi)+\mathbf b(t), \qquad \xi=s-at.
	\]
	Then $\mathbf X_s=\mathscr R^{\Omega t}\mathbf X_0'(\xi)=\mathscr R^{\Omega t}\mathbf x(\xi)$
	and $\mathbf X_{ss}=\mathscr R^{\Omega t}\mathbf x'(\xi)$.
	Therefore, using invariance of the cross product under rotations,
	\[
	\mathbf X_s\times\mathbf X_{ss}=\mathscr R^{\Omega t}(\mathbf x\times\mathbf x').
	\]
	On the other hand,
	\[
	\mathbf X_t=\Omega e_3\times\mathscr R^{\Omega t}\mathbf X_0(\xi)-a\mathscr R^{\Omega t}\mathbf X_0'(\xi)+\mathbf b'(t).
	\]
	Since rotations around \(e_3\) commute with \(v\mapsto e_3\times v\), and since
	\(\mathbf X_0'=\mathbf x\), this becomes
	\[
	\mathbf X_t=\mathscr R^{\Omega t}\left(\Omega e_3\times\mathbf X_0-a\mathbf x+\mathbf d\right).
	\]
	Using $\mathbf d=\mathbf x\times\mathbf x'-\Omega e_3\times\mathbf X_0+a\mathbf x$,
	we obtain
	\[
	\mathbf X_t=\mathscr R^{\Omega t}(\mathbf x\times\mathbf x')=\mathbf X_s\times\mathbf X_{ss}.
	\]
	Thus \(\mathbf X\) solves VFE.
	
	Conversely, suppose that \(\mathbf X\) has the form
	\[
	\mathbf X(t,s)=\mathscr R^{\Omega t}\mathbf X_0(s-at)+\mathbf b(t),\qquad\mathbf X_0'=\mathbf x,
	\]
	and solves VFE. The equality $\mathbf X_t=\mathbf X_s\times\mathbf X_{ss}$
	is equivalent, by the computation above, to
	\[
	\Omega e_3\times\mathbf X_0-a\mathbf x+\mathscr R^{-\Omega t}\mathbf b'(t)=\mathbf x\times\mathbf x'.
	\]
	Therefore
	\[
	\mathscr R^{-\Omega t}\mathbf b'(t)=\mathbf x\times\mathbf x'-\Omega e_3\times\mathbf X_0+a\mathbf x.
	\]
	The left-hand side depends only on \(t\), while the right-hand side depends only
	on \(\xi\). Hence both sides must be equal to a constant vector
	\(\mathbf d\). This gives $\mathbf b'(t)=\mathscr R^{\Omega t}\mathbf d$
	and
	\[
	\mathbf x\times\mathbf x'-\Omega e_3\times\mathbf X_0+a\mathbf x=\mathbf d.
	\]
	Differentiating this identity with respect to \(\xi\), and using
	\(\mathbf X_0'=\mathbf x\), gives
	\[
	\mathbf x\times\mathbf x''-\Omega e_3\times\mathbf x+a\mathbf x'=0.
	\]
	This is \eqref{Eq5.4}.
\end{proof}

Thus, at the level of profiles, the equation studied in this paper is the
differentiated form of the rigid travelling--rotating filament equation
\eqref{Eq5.5}. Equivalently, Kida's profile equation is an integral
of our tangent-profile equation.

\subsection{Normalization of the translational part}

The constant vector \(\mathbf d\) in \eqref{Eq5.5} determines the
translation through $\mathbf b'(t)=\mathscr R^{\Omega t}\mathbf d$.
Writing
\[
\mathbf d=\mathbf d_\perp+d_3e_3,\qquad\mathbf d_\perp\cdot e_3=0,
\]
we see that the translation velocity consists of a rotating horizontal component
and a constant vertical drift. When \(\Omega\neq0\), the horizontal component can
be removed by shifting the filament profile by a constant vector.

\begin{lemma}[Gauge normalization]
	Assume \(\Omega\neq0\). Up to replacing $\mathbf X_0$ by 
	$\widetilde{\mathbf X}_0=\mathbf X_0+\mathbf c$, $\mathbf c\in\mathbb R^3$,
	one may suppose that the constant vector \(\mathbf d\) in
	\eqref{Eq5.5} is vertical:
	\[
	\mathbf d=d_3e_3.
	\]
	With this normalization, $\mathbf b(t)=d_3t\,e_3+\mathbf b_0$.
\end{lemma}

\begin{proof}
	The shift $\widetilde{\mathbf X}_0=\mathbf X_0+\mathbf c$ does not change the tangent profile, since $\widetilde{\mathbf X}_0'=\mathbf X_0'=\mathbf x$.
	However,
	\[
	-\Omega e_3\times \widetilde{\mathbf X}_0=-\Omega e_3\times \mathbf X_0-\Omega e_3\times \mathbf c.
	\]
	Thus the corresponding constant vector becomes $\widetilde{\mathbf d}=\mathbf d-\Omega e_3\times\mathbf c$.
	Write $\mathbf d=\mathbf d_\perp+d_3e_3$, $\mathbf d_\perp\cdot e_3=0$, and choose
	\[
	\mathbf c=-\frac1\Omega e_3\times\mathbf d_\perp.
	\]
	Then, using $e_3\times(e_3\times\mathbf d_\perp)=-\mathbf d_\perp$,
	we get $\Omega e_3\times\mathbf c=\mathbf d_\perp$.
	Hence $\widetilde{\mathbf d}=\mathbf d-\mathbf d_\perp=d_3e_3$.
	With this normalization,
	\[
	\mathbf b'(t)=\mathscr R^{\Omega t}(d_3e_3)=d_3e_3,
	\]
	and therefore $\mathbf b(t)=d_3t\,e_3+\mathbf b_0$.
\end{proof}

	The case \(\Omega=0\) is different. In that case the shift \(\mathbf X_0\mapsto \mathbf X_0+\mathbf c\) does not modify the constant vector \(\mathbf d\) through a horizontal term, and the horizontal component of the
	translational velocity cannot in general be gauged away by this argument.

\appendix

\section{Jacobi elliptic functions and elliptic integrals}
\label{ApA}

For the convenience of the reader, we recall in this appendix the basic facts
about Jacobi elliptic functions and elliptic integrals used throughout the
paper. We refer to \cite{AS} for a standard reference. We follow the convention in which the elliptic parameter is denoted by \(k\in(0,1)\).

\subsection{Elliptic integrals of the first kind}

For \(k\in(0,1)\), the incomplete elliptic integral of the first kind is defined by
\[
F(\varphi\mid k):=\int_0^\varphi\frac{d\theta}{\sqrt{1-k\sin^2\theta}}.
\]
Equivalently, after the change of variable \(s=\sin\theta\),
\[
F(\varphi\mid k)
=\int_0^{\sin\varphi}\frac{ds}{\sqrt{(1-s^2)(1-ks^2)}}.
\]
The complete elliptic integral of the first kind is
\[
K(k):=F\left(\frac{\pi}{2}\mid k\right)=\int_0^{\pi/2}
\frac{d\theta}{\sqrt{1-k\sin^2\theta}}=\int_0^1\frac{ds}{\sqrt{(1-s^2)(1-ks^2)}}.
\]
We shall repeatedly use that $K(k)<+\infty$
for $0<k<1$,
whereas $K(k)\to+\infty$ as $k\uparrow1$.

\subsection{Jacobi elliptic functions}

The Jacobi sine function \(\operatorname{sn}(u\mid k)\) is defined as the inverse
of the elliptic integral of the first kind. More precisely, if $u=F(\varphi\mid k)$,
then
\[
\operatorname{sn}(u\mid k)=\sin\varphi.
\]
One also defines
\[
\operatorname{cn}(u\mid k):=\cos\varphi,\quad
\operatorname{dn}(u\mid k):=\sqrt{1-k\sin^2\varphi}.
\]
Equivalently,
\[
\operatorname{cn}^{2}(u\mid k)
=1-\operatorname{sn}^{2}(u\mid k), \quad
\operatorname{dn}^{2}(u\mid k)=
1-k\,\operatorname{sn}^{2}(u\mid k).
\]
The function \(\operatorname{sn}(u\mid k)\) is \(4K(k)\)-periodic in the real
variable \(u\), while \(\operatorname{sn}^{2}(u\mid k)\) is \(2K(k)\)-periodic.
More precisely,
\[
\operatorname{sn}(u+4K(k)\mid k)=\operatorname{sn}(u\mid k), \qquad 
\operatorname{sn}^{2}(u+2K(k)\mid k)=\operatorname{sn}^{2}(u\mid k).
\]
For \(0<k<1\), the minimal positive real period of
\(\operatorname{sn}^{2}(\cdot\mid k)\) is \(2K(k)\).

\subsection{Elliptic integrals of the third kind}

For \(n\in\mathbb R\) and \(k\in(0,1)\), the incomplete elliptic integral of the
third kind is defined by
\[
\Pi(\varphi\mid n,k):=\int_0^\varphi\frac{d\theta}{(1-n\sin^2\theta)\sqrt{1-k\sin^2\theta}}.
\]
Equivalently, in Jacobi form,
\[
\Pi(u\mid n,k):=\int_0^u\frac{d\tau}{1-n\operatorname{sn}^{2}(\tau\mid k)}.
\]
These two definitions are related by \(u=F(\varphi\mid k)\). Indeed, if $s=\operatorname{sn}(\tau\mid k)$,
then
\[
d\tau=\frac{ds}{\sqrt{(1-s^2)(1-ks^2)}}.
\]
The complete elliptic integral of the third kind is
\[
\Pi(n\mid k):=\Pi\left(\frac{\pi}{2}\mid n,k\right)=\int_0^{\pi/2}\frac{d\theta}{(1-n\sin^2\theta)\sqrt{1-k\sin^2\theta}}.
\]
In Jacobi form,
\[
\Pi(n\mid k)=\int_0^{K(k)}\frac{d\tau}{1-n\operatorname{sn}^{2}(\tau\mid k)}.
\]
Since \(\operatorname{sn}^2(\cdot\mid k)\) is symmetric with respect to
\(K(k)\), namely
\[
\operatorname{sn}^2(2K(k)-\tau\mid k)=\operatorname{sn}^2(\tau\mid k),
\]
we have
\[
\int_0^{2K(k)}\frac{d\tau}{1-n\operatorname{sn}^{2}(\tau\mid k)}=2\int_0^{K(k)}\frac{d\tau}{1-n\operatorname{sn}^{2}(\tau\mid k)}=2\Pi(n\mid k),
\]
provided that the denominator does not vanish on \([0,2K(k)]\).

\section{Degenerate parameter regimes}
\label{ApB}

In the main body of the paper we have restricted ourselves to the non-degenerate
travelling--rotating regime
\[
a>0,\qquad \Omega>0.
\]
As explained in Section~\ref{S2}, this entails no loss of generality
as long as \(a\Omega\neq0\), after reversing the orientation of the profile
variable and, if necessary, applying the antipodal symmetry on the sphere.
The cases \(a=0\) and \(\Omega=0\), however, are genuinely degenerate for the
bifurcation analysis. In this appendix we record their elementary classification.
Throughout, we assume that the profile lies on a fixed sphere,
\[
\|\mathbf x(\xi)\|\equiv R>0.
\]

\subsection{The case \(a=0\)}

If \(\Omega\neq0\), the antipodal symmetry \(\mathbf x\mapsto-\mathbf x\)
changes \(\Omega\) into \(-\Omega\). Hence, up to this symmetry, it is enough to
describe the case \(\Omega>0\). When \(a=0\), the profile equation becomes
\begin{equation}
	\label{B.1}
	\mathbf x\times \ddot{\mathbf x}-\Omega e_3\times\mathbf x=0.
\end{equation}
Equivalently, $\mathbf x\times(\ddot{\mathbf x}+\Omega e_3)=0$,
so \(\ddot{\mathbf x}+\Omega e_3\) is pointwise parallel to \(\mathbf x\).
This case no longer corresponds to travelling profiles; it corresponds to
purely rotating profiles.

The equation still has two first integrals. The vertical angular momentum is $C:=(\mathbf x\times\dot{\mathbf x})_3$,
because taking the third component of \eqref{B.1} gives
\[
\frac{d}{d\xi}(\mathbf x\times\dot{\mathbf x})_3=0.
\]
Moreover, $H:=\frac12\|\dot{\mathbf x}\|^2+\Omega x_3$
is constant. Indeed, since \(\ddot{\mathbf x}+\Omega e_3\) is parallel to \(\mathbf x\), and
\(\dot{\mathbf x}\perp\mathbf x\), we have
\[
\frac{d}{d\xi}\left(\frac12\|\dot{\mathbf x}\|^2+\Omega x_3\right)=\dot{\mathbf x}\cdot(\ddot{\mathbf x}+\Omega e_3)=0.
\]
Let \(u:=x_3\). For non-polar solutions, write
\[
\mathbf x(\xi)=\big(\sqrt{R^2-u(\xi)^2}\cos\phi(\xi),\sqrt{R^2-u(\xi)^2}\sin\phi(\xi),u(\xi)\big).
\]
Then
\[
C=(\mathbf x\times\dot{\mathbf x})_3=(R^2-u^2)\dot\phi,
\]
and
\[
\|\dot{\mathbf x}\|^2=\frac{R^2\dot u^2}{R^2-u^2}+\frac{C^2}{R^2-u^2}.
\]
Using the conservation of \(H\), we obtain the scalar equation
\begin{equation*}
	\dot u^2=\frac{2(H-\Omega u)(R^2-u^2)-C^2}{R^2}.
\end{equation*}
Thus the vertical component is again governed by a cubic polynomial, $\dot u^2=P_0(u)$, where
\[
P_0(u):=\frac{2(H-\Omega u)(R^2-u^2)-C^2}{R^2}.
\]
The azimuthal equation is
\begin{equation*}
	\dot\phi=\frac{C}{R^2-u^2}.
\end{equation*}
Consequently, the classification is analogous to the non-degenerate case:
regular non-polar oscillations occur when \(P_0\) has three real roots
\(e_1<e_2<e_3\), with
\[
-R<e_1<e_2<R,
\]
and \(u\) oscillates between \(e_1\) and \(e_2\). These solutions are expressed
through Jacobi elliptic functions, and the full profile is periodic precisely
when the azimuthal increment over one vertical period is rationally related to
\(2\pi\), otherwise the profile is quasi-periodic on the corresponding spherical
band. Degenerate double-root configurations give separatrix-type solutions.

The non-polar constant-latitude solutions are obtained by taking
\[
\mathbf x(\xi)=\big(\rho\cos(\nu\xi+\phi_0),\rho\sin(\nu\xi+\phi_0),e\big), \qquad \rho=\sqrt{R^2-e^2}, \qquad |e|<R.
\]
Substitution into \eqref{B.1} gives $e\nu^2+\Omega=0$.
Thus such solutions exist precisely when \(e\Omega<0\). In particular, if
\(\Omega>0\), the non-polar constant latitudes lie in the lower hemisphere
\(e<0\). The polar constant solutions $\mathbf x\equiv Re_3$, $\mathbf x\equiv -Re_3$,
also solve \eqref{B.1}.

\subsection{The case \(\Omega=0\)}

When \(\Omega=0\), the profile equation becomes
\begin{equation*}
	\mathbf x\times\ddot{\mathbf x}+a\dot{\mathbf x}=0.
\end{equation*}
Assume first that \(a\neq0\). Then, as in Proposition~2.1, every nonzero solution
has constant norm. In this case the equation is completely integrable in a
particularly simple way. Indeed,
\[
\frac{d}{d\xi}(\mathbf x\times\dot{\mathbf x})=\mathbf x\times\ddot{\mathbf x}=-a\dot{\mathbf x}.
\]
Therefore
\begin{equation}
	\label{B.2}
	\mathbf K:=\mathbf x\times\dot{\mathbf x}+a\mathbf x
\end{equation}
is a constant vector. Taking the scalar product of \eqref{B.2} with \(\mathbf x\), we obtain $\mathbf{K}\cdot\mathbf x=aR^2$.
Thus every solution lies in the intersection of the sphere
\(\|\mathbf x\|=R\) with the fixed plane $\mathbf K\cdot\mathbf x=aR^2$. Hence every nonconstant solution is a circle on the sphere.

Conversely, taking the cross product of \eqref{B.2} with \(\mathbf x\),
and using \(\mathbf x\cdot\dot{\mathbf x}=0\), gives
\[
\mathbf K\times\mathbf x=(\mathbf x\times\dot{\mathbf x})\times\mathbf x=R^2\dot{\mathbf x}.
\]
Therefore
\begin{equation*}
	\dot{\mathbf x}=\frac1{R^2}\mathbf K\times\mathbf x.
\end{equation*}
Thus the motion is a uniform rotation around the fixed axis \(\mathbf K\). Setting
\[
\nu:=\frac{|\mathbf K|}{R^2},\qquad\mathbf n:=\frac{\mathbf K}{|\mathbf K|},
\]
we have
$\dot{\mathbf x}=\nu\,\mathbf n\times\mathbf x$ and 
$\mathbf n\cdot\mathbf x=\frac{a}{\nu}$.
In particular, after a rigid rotation sending \(\mathbf n\) to \(e_3\), every
nonconstant solution can be written as
\[
\mathbf x(\xi)=\big(\rho\cos(\nu\xi+\phi_0),\rho\sin(\nu\xi+\phi_0),e\big),
\]
where $e=\frac{a}{\nu}$, $\rho=\sqrt{R^2-e^2}$.
Equivalently, the constant-latitude relation is $e\nu=a$.
This is exactly the specialization to \(\Omega=0\) of the general relation
\[
e\nu^2-a\nu+\Omega=0,
\]
after excluding the trivial angular velocity \(\nu=0\).

Thus, when \(\Omega=0\) and \(a\neq0\), all nonconstant spherical profiles are
circles, possibly after a rigid rotation of the ambient space. There are no
genuine vertical oscillatory profiles of the type studied in the main
bifurcation analysis.

Finally, if \(a=0\) and \(\Omega=0\), the equation reduces to $\mathbf x\times\ddot{\mathbf x}=0$.
On the sphere this means that \(\ddot{\mathbf x}\) is parallel to \(\mathbf x\).
Since $\mathbf x\cdot\ddot{\mathbf x}=-\|\dot{\mathbf x}\|^2$,
we get
\[
\ddot{\mathbf x}=-\frac{\|\dot{\mathbf x}\|^2}{R^2}\mathbf x.
\]
The speed is constant, and the nonconstant solutions are precisely great
circles parametrized with constant speed. Constant solutions are also allowed.


\begin{thebibliography}{99}

\bibitem{AS}
M. Abramowitz and I. A. Stegun,
\emph{Handbook of Mathematical Functions with Formulas, Graphs, and Mathematical Tables},
National Bureau of Standards Applied Mathematics Series, Vol. 55,
U.S. Government Printing Office, Washington, D.C., 1964.

\bibitem{ArmsHama1965}
R. J. Arms and F. R. Hama,
\emph{Localized-induction concept on a curved vortex and motion of an elliptic vortex ring},
Phys. Fluids \textbf{8} (1965), no. 4, 553--559.

\bibitem{BanicaVega2009}
V. Banica and L. Vega,
\emph{On the stability of a singular vortex dynamics},
Comm. Math. Phys. \textbf{286} (2009), 593--627.

\bibitem{BanicaVega2012}
V. Banica and L. Vega,
\emph{Scattering for 1D cubic NLS and singular vortex dynamics},
J. Eur. Math. Soc. (JEMS) \textbf{14} (2012), no. 1, 209--253.

\bibitem{BanicaVega2015}
V. Banica and L. Vega,
\emph{The initial value problem for the binormal flow with rough data},
Ann. Sci. \'Ec. Norm. Sup\'er. (4) \textbf{48} (2015), no. 6, 1423--1455.

\bibitem{Betchov1965}
R. Betchov,
\emph{On the curvature and torsion of an isolated vortex filament},
J. Fluid Mech. \textbf{22} (1965), no. 3, 471--479.

\bibitem{CR1971}
M. G. Crandall and P. H. Rabinowitz,
\emph{Bifurcation from simple eigenvalues},
J. Funct. Anal. \textbf{8} (1971), 321--340.

\bibitem{DaRios1906}
L. S. Da Rios,
\emph{Sul moto d'un liquido indefinito con un filetto vorticoso di forma qualunque},
Rend. Circ. Mat. Palermo \textbf{22} (1906), 117--135.

\bibitem{DeLaHozVega2014}
F. de la Hoz and L. Vega,
\emph{Vortex filament equation for a regular polygon},
Nonlinearity \textbf{27} (2014), no. 12, 3031--3057.

\bibitem{GarciaVega2024}
C. Garc\'ia and L. Vega,
\emph{Steady solutions for the Schr\"odinger map equation},
Comm. Partial Differential Equations \textbf{49} (2024), no. 5--6, 505--542.

\bibitem{GutierrezRivasVega2003}
S. Guti\'errez, J. Rivas and L. Vega,
\emph{Formation of singularities and self-similar vortex motion under the localized induction approximation},
Comm. Partial Differential Equations \textbf{28} (2003), no. 5--6, 927--968.

\bibitem{Hasimoto1972}
H. Hasimoto,
\emph{A soliton on a vortex filament},
J. Fluid Mech. \textbf{51} (1972), no. 3, 477--485.

\bibitem{JerrardSmets2012}
R. L. Jerrard and D. Smets,
\emph{On Schr\"odinger maps from \(\mathbb T^1\) to \(S^2\)},
Ann. Sci. \'Ec. Norm. Sup\'er. (4) \textbf{45} (2012), no. 4, 637--680.

\bibitem{Kida1981}
S. Kida,
\emph{A vortex filament moving without change of form},
J. Fluid Mech. \textbf{112} (1981), 397--409.

\bibitem{Kielhofer2012}
H. Kielh\"ofer,
\emph{Bifurcation Theory: An Introduction with Applications to Partial Differential Equations},
second ed.,
Applied Mathematical Sciences, vol. 156,
Springer, New York, 2012.

\bibitem{LangerPerline1991}
J. Langer and R. Perline,
\emph{Poisson geometry of the filament equation},
J. Nonlinear Sci. \textbf{1} (1991), 71--93.

\bibitem{LangerSinger1984}
J. Langer and D. A. Singer,
\emph{Knotted elastic curves in \(\mathbb R^3\)},
J. London Math. Soc. (2) \textbf{30} (1984), no. 3, 512--520.

\bibitem{LG} 
J. L\'{o}pez-G\'{o}mez,
\emph{Spectral Theory and Nonlinear Functional Analysis},
Chapman and Hall/CRC Research Notes in Mathematics, vol.~426, CRC Press, Boca Raton (FL), 2001.

\bibitem{Rabinowitz1971}
P. H. Rabinowitz,
\emph{Some global results for nonlinear eigenvalue problems},
J. Funct. Anal. \textbf{7} (1971), 487--513.

\bibitem{Saffman1992}
P. G. Saffman,
\emph{Vortex Dynamics},
Cambridge Monographs on Mechanics and Applied Mathematics,
Cambridge University Press, Cambridge, 1992.
	
\end{thebibliography}
\end{document}